\documentclass[12pt,a4paper,bibtotoc]{article}
\usepackage[margin=3cm]{geometry}
\usepackage{etoolbox}

\makeatletter
\newcommand{\fancybreak}[1]{\par
  \penalty -100
  \noindent\parbox{\linewidth}{\centering #1}
  \par
  \@afterindentfalse
  \@afterheading}
\makeatother

\usepackage[raggedright]{titlesec}
\titleformat*\section{\Large\bfseries\boldmath\raggedright}
\titleformat*\subsection{\large\bfseries\boldmath\raggedright}
\titleformat*\paragraph{\bfseries}
\titlespacing*\paragraph{0em}{0.5em plus 0.5em}{1em}

\usepackage{amsthm,amsmath,amssymb,mathtools,color,hyperref,tocloft}

\usepackage[shortlabels]{enumitem} 
\usepackage[all]{xy}
\numberwithin{equation}{section}
\theoremstyle{plain}
\newtheorem{thm}{Theorem}[section]
\newtheorem*{thm*}{Theorem}
\newtheorem{prop}[thm]{Proposition}
\newtheorem{lem}[thm]{Lemma}
\newtheorem{cor}[thm]{Corollary}
\newtheorem{fact}[thm]{Fact}
\theoremstyle{definition}

\newtheorem{rem}[thm]{Remark}

\newcommand{\td}{\,\mathrm{d}}
\newcommand{\di}{\mathrm{d}}
\newcommand{\const}{\textup{const}}
\DeclareMathOperator{\Real}{Re}
\DeclareMathOperator{\Imaginary}{Im}
\newcommand{\id}{\textup{id}}
\newcommand{\GL}{\mathrm{GL}}

\newcommand{\SL}{\mathrm{SL}}

\newcommand{\so}{\mathfrak{so}}
\newcommand{\SU}{\mathrm{SU}}

\newcommand{\RR}{\mathbb{R}}

\newcommand{\CC}{\mathbb{C}}
\newcommand{\ZZ}{\mathbb{Z}}
\newcommand{\NN}{\mathbb{N}}
\newcommand{\HH}{\mathbb{H}}
\newcommand{\FF}{\mathbb{F}}
\newcommand{\Ind}{\textup{Ind}}
\newcommand{\calF}{\mathcal{F}}
\newcommand{\calO}{\mathcal{O}}
\newcommand{\calS}{\mathcal{S}}
\newcommand{\calB}{\mathcal{B}}
\newcommand{\calH}{\mathcal{H}}
\newcommand{\calU}{\mathcal{U}}
\newcommand{\calD}{\mathcal{D}}
\newcommand{\calL}{\mathcal{L}}
\newcommand{\frakg}{\mathfrak{g}}

\newcommand{\frakk}{\mathfrak{k}}
\newcommand{\frakp}{\mathfrak{p}}
\newcommand{\frakn}{\mathfrak{n}}
\newcommand{\fraka}{\mathfrak{a}}
\newcommand{\frakm}{\mathfrak{m}}
\newcommand{\frakh}{\mathfrak{h}}
\newcommand{\frakl}{\mathfrak{l}}
\newcommand{\frakq}{\mathfrak{q}}
\newcommand{\Residue}{\textup{res}}
\newcommand{\diag}{\operatorname{diag}}
\newcommand{\linspan}{\operatorname{span}}
\newcommand{\Spin}{\textup{Spin}}
\newcommand{\pr}{\textup{pr}}
\newcommand{\sub}{\textup{sub}}
\renewcommand{\1}{\mathbf{1}}
\DeclareMathOperator{\supp}{supp}
\DeclareMathOperator{\sign}{sign}

\newcommand\emptyarg{-}

\DeclarePairedDelimiterX\abs[1]\lvert\rvert{\ifblank{#1}{\emptyarg}{#1}}
\DeclarePairedDelimiterX\norm[1]\lVert\rVert{\ifblank{#1}{\emptyarg}{#1}}
\DeclarePairedDelimiterX\Set[2]\{\}{#1:#2}
\DeclarePairedDelimiterX\Hermit[2](){%
  \ifblank{#1}{\emptyarg}{#1}
  \:\delimsize\vert\:
  \ifblank{#2}{\emptyarg}{#2}}
\DeclarePairedDelimiterX\innerp[2]\langle\rangle{%
  \ifblank{#1}{\emptyarg}{#1}
  ,
  \ifblank{#2}{\emptyarg}{#2}}

\newcommand\minuszero{\setminus\{0\}}
\newcommand\restrictedto[1][]{#1\vert\sb} 

\newcommand\qqtext{\qquad\text}
\newcommand\qtextq[1]{\quad\text{#1}\quad}

\title{Restriction of some unitary representations of $O(1,N)$ to symmetric subgroups}
\author{Jan M\"ollers, Yoshiki Oshima}
\date{}

\begin{document}

\maketitle

\begin{center}
{\footnotesize On the occasion of the centennial anniversary of Professor Kunihiko Kodaira's birthday.}
\end{center}
\vspace{1cm}

\begin{abstract}
  
  We find the complete branching law for the restriction of
  certain unitary representations of $O(1,n+1)$ to the 
  subgroups $O(1,m+1)\times O(n-m)$, $0\leq m\leq n$. The unitary
  representations we consider belong either to the unitary spherical
  principal series, the spherical complementary series or are unitarizable
  subquotients of the spherical principal series.\\
  In the crucial case $0<m<n$ the decomposition consists of a
  continuous part and a discrete part.
  The continuous part is given by a direct integral of unitary principal
  series representations whereas the discrete part consists of finitely
  many representations which either belong to the complementary series
  or are unitarizable subquotients of the principal series. The explicit
  Plancherel formula is computed on the Fourier transformed side of
  the non-compact realization of the representations by using the
  spectral decomposition of a certain hypergeometric type ordinary
  differential operator. The main tool
  connecting this differential operator with the representations
  are second order Bessel operators which describe the Lie algebra
  action in this realization.\\
  To derive the spectral decomposition of the ordinary differential
  operator we use Kodaira's formula for the spectral decomposition of Schr\"{o}dinger type operators.

  \medskip\noindent \emph{2010 MSC:} Primary 22E46; Secondary 33C05,
  34B24.

  \medskip\noindent \emph{Key words and phrases:} unitary
  representation, complementary series, principal series, relative
  discrete series, branching law, Bessel operators,
  hypergeometric function, Kodaira--Titchmarsh formula.
\end{abstract}

\newpage

\tableofcontents


\section*{Introduction}\label{sec:intro}
\addcontentsline{toc}{section}{Introduction}

Among his various different mathematical contributions, the spectral theory
 of self-adjoint differential operators is one of the earlier works of Professor
 Kunihiko Kodaira.
The so-called Weyl--Stone--Kodaira--Titchmarsh theory gives
 eigenfunction expansions for self-adjoint second order differential operators in one variable.
It provides a uniform treatment of classical eigenfunction expansions such
 as spectral decompositions into Bessel functions, Hermite polynomials or
 Laguerre functions.
 
Kodaira~\cite{KodSugaku1} considered a differential operator 
 $L=\frac{\di}{\di x}p(x)\frac{\di}{\di x}+q(x)$
 on a possibly unbounded interval $(a,b)\subset \RR$
 (see \cite{KodSugaku1,Kod49} for the precise setting).
Then $L$ extends to a self-adjoint operator on the space of square integrable functions on $(a,b)$
 with domain given by functions satisfying certain boundary conditions, and we have an expansion into
 eigenfunctions of the form:
\begin{align*}
u(x) = \sum_{j=1}^2\sum_{k=1}^2
 \int_{-\infty}^{\infty} s_j(x,\lambda) \int_a^b
 s_k(y,\lambda) u(y) \td y \td \rho_{jk}(\lambda).
\end{align*}
Here, the functions $s_1(\cdot, \lambda)$ and $s_2(\cdot, \lambda)$
 are linearly independent solutions to the equation $Lu=\lambda u$.
 
The existence of the density measure $\td\rho_{jk}$
 for which the above expansion holds was first proved by
 Weyl~\cite{Wey10}.
Later Stone gave in \cite[Theorem 10.22]{Sto32} a different proof using the general theory
 of operators on Hilbert spaces.
About forty years after Weyl's result,
 Kodaira~\cite{KodSugaku1, Kod49}
 found an explicit formula for $\td\rho_{jk}$ in terms of
 the characteristic functions, revealing the explicit relation 
 between the density measures and the asymptotic behaviour of eigenfunctions.
The same formula was also obtained independently by Titchmarsh~\cite{Tit46}
 using a different method, and it is called the Kodaira--Titchmarsh formula.
In the second half of \cite{Kod49} and in \cite{KodSugaku2} Kodaira
 studied the eigenfunctions and the density matrix in detail
 for some particular cases, which are important for applications.
Moreover, in a subsequent paper~\cite{Kod50} he generalized this formula
 to differential operators of any even order.
 
In his Gibbs lecture~\cite{Wey50} Weyl wrote about Kodaira:
`The formula~\cite[(12)]{Wey50} was rediscovered by Kunihiko Kodaira (who of course
 had been cut off from our Western mathematical literature since
 the end of 1941); his construction of $\rho$ and his proofs for 
 \cite[(12)]{Wey50}
 and the expansion formula \cite[(9)]{Wey50},
 still unpublished, seem to clinch the issue.
 It is remarkable that forty years had to pass before such a thoroughly
 satisfactory direct treatment emerged; the fact is a reflection
 on the degree to which mathematicians during this period got absorbed
 in abstract generalizations and lost sight of their task of finishing up
 some of the more concrete problems of undeniable importance.'
 
The Kodaira--Titchmarsh formula makes it possible to apply the spectral decomposition theorem
 to concrete settings and in particular it has a significant impact on 
 the harmonic analysis on Lie groups.
For a given variety $X$ and a Lie group $G$ acting on it, a fundamental problem
 in the harmonic analysis on $X$ is to expand arbitrary function on $X$ into
 joint eigenfunctions for the $G$-invariant differential operators on $X$.
An explicit description of such an expansion is called Plancherel Theorem.
When $X$ is a symmetric space of rank one, it amounts to
 eigenfunction expansions for the Laplacian.
In this case the problem can be reduced to an eigenfunction expansion
 for a self-adjoint second order differential operator in one variable 
 and hence the Kodaira--Titchmarsh formula can be applied directly.
In fact, a special case studied in the second half of \cite{Kod49}
 and in \cite{KodSugaku2} is enough to deduce the Plancherel Theorem
 for all symmetric spaces of rank one.
For Riemannian symmetric spaces $G/K$ of arbitrary rank the Plancherel Theorem was established
 by Harish-Chandra (see \cite{HC58}).
In this case the Plancherel measure is given by the $c$-function which can be
 explicitly written in terms of the Gamma function by the work of
 Gindikin--Karpelevi{\v{c}} \cite{GK62}.
As the $c$-function is defined in terms of the asymptotic behaviour of
 joint eigenfunctions, we see once more the spirit of the Kodaira--Titchmarsh formula
 in this setting.
For pseudo-Riemannian symmetric spaces $G/H$, special cases like
 hyperboloids $O(p,q)/O(p,q-1)$ were studied in the sixties
 by Shintani~\cite{Shi66} and Mol{\v{c}}anov~\cite{Mol66}.
The Plancherel Theorem for general
 semisimple symmetric spaces of arbitrary rank
 was established by the works of T. Oshima,
 Delorme~\cite{Del98}, and van den Ban--Schlichtkrull~\cite{BS05}.

\fancybreak{}

Plancherel Theorems for reductive homogeneous spaces $G/H$ can be viewed as
 induction problems, decomposing the induced representation
 $L^2(G/H)=\Ind_H^G(\1)$ into irreducible $G$-representations.
As well as induction problems we may consider restriction problems 
 as advocated in \cite{Kob05}, namely, we may ask how a representation
 decomposes when restricted to a subgroup.
The restriction problem was e.g.\ solved in \cite{KOP11,MS14} for
 the most degenerate principal series representations
 of $G=GL(n,\RR)$ and $G=GL(n,\CC)$ with respect to any symmetric pair $(G,H)$.
Since (degenerate) principal series representations
 are realized on $L^2$-sections of line bundles on a flag variety $G/P$,
 Mackey theory relates these restriction problems to
 the Plancherel type problems for the open $H$-orbits in $G/P$.
Our focus is on the indefinite orthogonal group
$O(1,n+1)$, $n\geq1$, for which we study the restriction of certain unitary
 representations using the Kodaira--Titchmarsh formula.

\fancybreak{}

We now introduce some notation in order to describe our results.
Let $G=O(1,n+1)$.
It is known that on the level of $(\frakg,K)$-modules
 all irreducible unitary representations of $G$ are obtained as
 subrepresentations of representations induced from a parabolic
 subgroup $P=MAN$.
Up to conjugation $P$ is unique and
 there are group isomorphisms
 $M\cong O(n)\times(\ZZ/2\ZZ)$,
$A\cong\RR_+$ and $N\cong\RR^n$. We restrict our attention to
representations induced from characters of $P$. Denote by
$\pi_{\sigma,\varepsilon}^{O(1,n+1)}$ the representation of $G$, which
is induced from the character of $P$ given by the character
$\sigma\in\CC$ of $A$ and the character $\varepsilon\in\ZZ/2\ZZ$ of
the second factor of $M\cong O(n)\times(\ZZ/2\ZZ)$ (normalized
parabolic induction).\\
In our parameterization
$\pi_{\sigma,\varepsilon}^{O(1,n+1)}$ is irreducible and unitarizable
if and only if $\sigma\in i\RR\cup(-n,n)$. By abuse of notation we
denote by $\pi_{\sigma,\varepsilon}^{O(1,n+1)}$ also the corresponding
irreducible unitary representations. For $\sigma\in i\RR$ these
representations are called
\textit{unitary principal series representations} and for
$\sigma\in(-n,0)\cup(0,n)$ they are called \textit{complementary series
representations}. We have natural isomorphisms
$\pi_{-\sigma,\varepsilon}^{O(1,n+1)}\cong\pi_{\sigma,\varepsilon}^{O(1,n+1)}$
for $\sigma\in i\RR\cup(-n,n)$.\\
Further, for $\sigma=n+2u$, $u\in\NN$, the representation
$\pi_{\sigma,\varepsilon}^{O(1,n+1)}$ has a unique non-trivial subrepresentation
$\pi_{\sigma,\varepsilon,\sub}^{O(1,n+1)}$. This subrepresentation is
irreducible and unitarizable and we use the same notation to also denote
the corresponding irreducible unitary representation. Its underlying $(\frakg,K)$-module is isomorphic to Zuckerman's module $A_\frakq(\lambda)$ for certain $\frakq$ and $\lambda$ and it occurs discretely in the Plancherel
formula for the hyperboloid $O(1,n+1)/O(1,n)$. We refer to these representations
as \textit{discrete series representations for the hyperboloid}.

\fancybreak{}

In this paper we study the restriction
of $\pi_{\sigma,\varepsilon}^{O(1,n+1)}$, $\sigma\in i\RR\cup(-n,n)$,
and $\pi_{\sigma,\varepsilon,\sub}^{O(1,n+1)}$, $\sigma\in n+2\NN$,
$\varepsilon\in\ZZ/2\ZZ$, with respect to any symmetric pair $(G,H)$. By
Berger's list \cite{Ber57} any non-trivial symmetric subgroup of $G$
is either conjugate to
\begin{align*}
  H &= O(1,m+1)\times O(n-m), && -1\leq m<n.
\end{align*}
Since $H$ is a maximal compact subgroup of $G$ if $m=-1$,
 the branching law for
 the restriction of $\pi_{\sigma,\varepsilon}^{O(1,n+1)}$ and
$\pi_{\sigma,\varepsilon,\sub}^{O(1,n+1)}$ to $O(1)\times O(n+1)$ is
simply the $K$-type decomposition \eqref{eq:KtypeDecomp} or
\eqref{eq:KtypeDecompSub} which is well-known. Moreover,
for $H=O(1,1)\times O(n)$, i.e.\ the case $m=0$, the branching law
can easily be derived using classical Fourier analysis, see
Section~\ref{sec:Branchingm=0}. The most interesting case is
 the branching to $H$ for $0<m<n$. In
the formulation of the branching law we use the conventions
$[0,\alpha)=\emptyset$ for $\alpha\leq0$ and $[0,\alpha]=\emptyset$ for $\alpha<0$.

\begin{thm*}[see Theorem~\ref{thm:RepDecomp}]
  The unitary representations $\pi_{\sigma,\varepsilon}^G$ and
  $\pi_{\sigma,\varepsilon,\sub}^G$ of $G=O(1,n+1)$ decompose into
  irreducible representations of $H=O(1,m+1)\times O(n-m)$, $0<m<n$,
  as follows: for $\sigma\in i\RR\cup(-n,n)$ and $\varepsilon\in\ZZ/2\ZZ$ we have
  \begin{multline*}
    \pi_{\sigma,\varepsilon}^G \restrictedto[\big] H 
    \cong 
    \sideset{}{^\oplus}\sum_{k=0}^\infty
    \Bigg(\int_{i\RR_+}^\oplus
    \pi_{\tau,\varepsilon+k}^{O(1,m+1)}\td\tau
    \\
    \oplus\bigoplus_{j\in\NN\cap
      \left[0,\frac{\abs{ \Real\sigma }-n+m-2k}{4}\right)}
    \pi_{\abs{ \Real\sigma }-n+m-2k-4j,\varepsilon+k}^{O(1,m+1)}
    \Bigg)\boxtimes\calH^k(\RR^{n-m}),
  \end{multline*}
  and for $\sigma=n+2u$, $u\in\NN$, and $\varepsilon\in\ZZ/2\ZZ$ we have
  \begin{multline*}
    \pi_{\sigma,\varepsilon,\sub}^G \restrictedto[\big] H 
    \cong 
    \sideset{}{^\oplus}\sum_{k=0}^\infty
    \Bigg(\int_{i\RR_+}^\oplus
    \pi_{\tau,\varepsilon+k}^{O(1,m+1)}\td\tau
    \oplus\bigoplus_{j\in\NN\cap
      \left[0,\frac{u-k}{2}\right]}
    \pi_{m+2u-2k-4j,\varepsilon+k,\sub}^{O(1,m+1)}
    \\
    \oplus\bigoplus_{j\in\NN\cap
      \left(\frac{u-k}{2},\frac{m+2u-2k}{4}\right)}
    \pi_{m+2u-2k-4j,\varepsilon+k}^{O(1,m+1)}
    \Bigg)\boxtimes\calH^k(\RR^{n-m}),
  \end{multline*}
  where $\calH^k(\RR^{n-m})$ denotes the irreducible representation of
  $O(n-m)$ on the space of solid spherical harmonics of degree $k$ on
  $\RR^{n-m}$.\\
\end{thm*}

The explicit Plancherel formula is given in
Theorem~\ref{thm:HIntertwiner}. First of all, the restriction
$\pi_{\sigma,\varepsilon}^G \restrictedto H$ resp.\ 
$\pi_{\sigma,\varepsilon,\sub}^G \restrictedto H$ is decomposed with
respect to the action of $O(n-m)$, the second factor of~$H$. Then the
decomposition of each $\calH^k(\RR^{n-m})$-isotypic component into
irreducible representations of $O(1,m+1)$ contains continuous and
discrete spectrum in general. The continuous spectrum is a direct integral
of unitary principal series
representations~$\pi_{\tau,\varepsilon+k}^{O(1,m+1)}$ of $O(1,m+1)$. The discrete
spectrum appears if and only if $k<\frac{\abs{ \Real\sigma }-n+m}{2}$ and is a
direct sum of finitely many complementary series representations in the case
$\pi_{\sigma,\varepsilon}^G \restrictedto H$ and additionally finitely
many discrete series representations for the hyperboloid in the case
$\pi_{\sigma,\varepsilon,\sub}^G \restrictedto H$. Therefore
the whole branching law of $\pi_{\sigma,\varepsilon}^G \restrictedto
H$ resp.\ $\pi_{\sigma,\varepsilon,\sub}^G \restrictedto H$
contains only finitely many discrete components and the discrete
spectrum is non-trivial if and only if $\abs{ \Real\sigma }>n-m$. In particular
for $m>0$ there is always at least one discrete component in the restriction
of the discrete series representations for the hyperboloid $\pi_{\sigma,\varepsilon,\sub}^G\restrictedto H$ and
also in the restriction of complementary series representations
$\pi_{\sigma,\varepsilon}^G \restrictedto H$ if $\sigma$ is sufficiently close to the
first reduction point $n$ or $-n$.

For $\sigma\in i\RR$ the decomposition is purely continuous. In this
case the branching law is actually equivalent to the Plancherel
formula for the Riemannian symmetric space $O(1,m+1)/(O(1)\times
O(m+1))$ (see Appendix~\ref{app:DecompPrincipalSeries}) and therefore
well-known. However, neither for the complementary series representations nor
the discrete series representations for the hyperboloid can the decomposition be
obtained in the same way.

The proof of the Plancherel formula we present works uniformly for all
$\sigma\in i\RR\cup(-n,n)\cup(n+2\NN)$, i.e.\ for both unitary principal series representations, complementary series, and discrete series representations for the hyperboloid. It uses the ``Fourier transformed
realization'' of $\pi_{\sigma,\varepsilon}^G$ resp.\ $\pi_{\sigma,\varepsilon,\sub}^G$ on
$L^2(\RR^n,\abs{ x }^{-\Real\sigma}\td x)$. For this consider first the
non-compact realization on the nilradical $\smash{\overline{N}}$ of
the parabolic subgroup $\overline{P}$ opposite to $P$. We then take
the Euclidean Fourier transform on $\overline{N}\cong\RR^n$ to obtain
a realization of $\pi_{\sigma,\varepsilon}^G$ resp.\ $\pi_{\sigma,\varepsilon,\sub}^G$ on
$L^2(\RR^n,\abs{ x }^{-\Real\sigma}\td x)$. The advantage of this
realization is that the invariant form is simply the $L^2$-inner
product. The Lie algebra action in the Fourier transformed picture is
given by differential operators up to order two, the crucial operators
being the second order Bessel operators studied in
\cite{HKM12}. We remark that similar differential operators for the minimal 
representation of $O(p,q)$ were previously studied in \cite{KM08} and 
are called fundamental differential operators there.
Using these operators in our case we reduce the branching law
to the spectral decomposition of an ordinary differential operator of
hypergeometric type on $L^2(\RR_+)$ (see
Section~\ref{sec:Reduction}). The spectral decomposition of this operator
is derived in Section~\ref{sec:SpectralDecomp} from Kodaira's result
on the Schr\"{o}dinger type operators and is used in
Section~\ref{sec:BranchingLawPlancherelFormula} to obtain the
branching law and the explicit Plancherel formula for the restriction
of the representations. An interesting
formula for the intertwining operators realizing the branching law in
the non-compact picture on $\overline{N}$ is computed in
Section~\ref{sec:IntertwinersNoncptPicture}. These intertwining
operators will be subject of a subsequent paper.

\fancybreak{}

Up to now only partial results regarding the branching of
$\pi_{\sigma,\varepsilon}^G$, $\sigma\in(0,n)$, and
$\pi_{\sigma,\varepsilon,\sub}^G$, $\sigma\in(n+2\NN)$, to $H$ were known.
Here are some related results: 
\begin{itemize}
\item For $n=2$ and $m=1$ the full decomposition of the complementary series
  was given by Mukunda \cite{Muk68} using the non-compact picture. This case corresponds to
  the branching law $\SL(2,\CC)\searrow\SL(2,\RR)$.
\item Extending his study on branching laws for 
  discretely decomposable restrictions~\cite{Kob93}, Kobayashi constructed
  discrete components for Zuckerman's modules of $O(p,q)$ when restricted
  to $O(p',q')\times O(p'',q'')$, which was announced in his talk \cite{Kob05a}.  The restriction of $\pi_{\sigma,\varepsilon,\sub}^{O(1,n+1)}$,
  $\sigma\in(n+2\NN)$ is a special case of his result.  
  By our Theorem it turns out that
  his construction gives all the discrete components of the restriction
  $\pi_{\sigma,\varepsilon,\sub}^{O(1,n+1)}|_{O(1,m+1)\times O(n-m)}$
  for $\sigma\in(n+2\NN)$.
\item After the announcement of Kobayashi's result \cite{Kob05a},
  Speh--Venkataramana \cite[Theorem 1]{SV11}
  proved the existence of the discrete component
  $\pi_{\sigma-1,0}^{O(1,n)}$ in $\pi_{\sigma,0}^{O(1,n+1)} \restrictedto {O(1,n)}$ for $n\geq2$, $m=n-1$ and 
  $\sigma\in(1,n)$ as well as the existence of the discrete component
  $\pi_{n-1,0,\sub}^{O(1,n)}$ in $\pi_{n,0,\sub}^{O(1,n+1)} \restrictedto {O(1,n)}$
  (special case $j=k=0$, $\sigma\in(1,n]$ in our Theorem). They also
  use the Fourier transformed picture for their proof. This is a
  special case of their more general result for complementary series
  representations of $G$ on differential forms, i.e.\ induced from
  more general (possibly non-scalar) $P$-representations.
\item The same special case was obtained by Zhang \cite[Theorem
  3.6]{Zha11}. He actually proved that for all rank one groups
  $G=\SU(1,n+1;\FF)$, $\FF=\RR,\CC,\HH$, resp.\ $G=F_{4(-20)}$ certain
  complementary series representations of $H=\SU(1,n;\FF)$ resp.\
  $H=\Spin(8,1)$ occur discretely in some spherical complementary
  series representations of $G$. His proof uses the compact picture
  and explicit estimates for the restriction of $K$-finite vectors.
\end{itemize}

\paragraph{Acknowledgements.} We thank Toshiyuki Kobayashi and Bent \O
rsted for helpful discussions. Most of this work was done during the
second author's visit to Aarhus University supported by the Department
of Mathematics.

\paragraph{Notation.} $\ZZ_+=\{1,2,3,\ldots\}$, $\NN=\ZZ_+\cup\{0\}$,
$\RR_+=\Set{ x\in\RR }{ x>0 }$.

\section{$L^2$-model of some representations of $O(1,n+1)$}\label{sec:reps}

In this section we recall the necessary geometry of the group
$G=O(1,n+1)$ and some of its representation theory. The $L^2$-models
discussed in Section~\ref{sec:FTpicture} were for the complementary series
previously constructed by Vershik--Graev~\cite{VG06} and are new for the
discrete series representations for the hyperboloid.

\subsection{Subgroups and decompositions}

Let $G=O(1,n+1)$, $n\geq1$, realized as the subgroup of $\GL(n+2,\RR)$
leaving the quadratic form
\begin{equation*}
  \RR^{n+2}\to\RR,\quad x=(x_1,\dots,x_{n+2})^t \mapsto
  x_1^2-(x_2^2+\cdots+x_{n+2}^2),
\end{equation*}
invariant. We fix the Cartan involution $\theta$ of $G$ given by
$\theta(g)=g^{-t}=(g^t)^{-1}$, $g\in G$, which corresponds to the
maximal compact subgroup $K:=G^\theta=O(1)\times O(n+1)$. On the Lie
algebra level the Lie algebra $\frakg$ of $G$ has the Cartan
decomposition $\frakg=\frakk\oplus\frakp$ into the $\pm1$ eigenspaces
$\frakk$ and $\frakp$ of $\theta$ where $\frakk$ is the Lie algebra of
$K$. Choose the maximal abelian subalgebra $\fraka:=\RR
H\subseteq\frakp$ spanned by the element
\begin{equation*}
  H := 2(E_{1,n+2}+E_{n+2,1}),
\end{equation*}
where $E_{ij}$ denotes the $(n+2)\times(n+2)$ matrix with $1$ in the
$(i, j)$-entry and $0$ elsewhere. The root system of the pair
$(\frakg,\fraka)$ consists only of the roots $\pm2\gamma$ where
$\gamma\in\fraka_\CC^*$ is defined by $\gamma(H):=1$. Put
\begin{alignat*}{2}
  \frakn &:= \frakg_{2\gamma}, &\qquad \overline{\frakn} &:=
  \frakg_{-2\gamma}=\theta\frakn
  \\
  \shortintertext{and let} N &:= \exp_G(\frakn), & \overline{N} &:=
  \exp_G(\overline{\frakn})=\theta N
\end{alignat*}
be the corresponding analytic subgroups of $G$. Since
$\dim(\frakn)=\dim(\overline{\frakn})=n$ the half sum of all positive
roots is given by $\rho=n\gamma$. We introduce the following
coordinates on $N$ and $\overline{N}$: For $1\leq j\leq n$ let
\begin{align*}
  N_j &:= E_{1,j+1}+E_{j+1,1}-E_{j+1,n+2}+E_{n+2,j+1},\\
  \overline{N}_j &:= E_{1,j+1}+E_{j+1,1}+E_{j+1,n+2}-E_{n+2,j+1}.
\end{align*}
For $x\in\RR^n$ let
\begin{equation*}
  n_x := \exp\Big(\sum_{j=1}^n{x_jN_j}\Big)\in N, 
  \qquad
  \overline{n}_x := \exp\Big(\sum_{j=1}^n{x_j\overline{N}_j}\Big)\in\overline{N}.
\end{equation*}
Further put $M:=Z_K(\fraka)$ and $A:=\exp(\fraka)$ and denote by
$\frakm$ the Lie algebra of $M$. We write $M=M^+\cup m_0M^+$ where
\begin{align*}
  M^+ &:= \Set{ \diag(1,k,1) }{ k\in O(n) }\cong O(n)\qqtext{and}
  \\
  m_0 &:= \diag(-1,1,\ldots,1,-1).
\end{align*}
Via conjugation the element $m_0$ acts on $N$ and $\overline{N}$ by
\begin{equation*}
  m_0n_xm_0^{-1} = n_{-x}  \qtextq{and}  m_0\overline{n}_xm_0^{-1}
  = \overline{n}_{-x}
\end{equation*}
and the action of $m\in M^+\cong O(n)$ on $N$ and $\overline{N}$ by
conjugation is given by
\begin{equation*}
  mn_xm^{-1} = n_{mx} \qtextq{and} m\overline{n}_xm^{-1} =
  \overline{n}_{mx}
\end{equation*}
for $x\in\RR^n$, where $mx$ is the usual action of $O(n)$ on
$\RR^n$. Further $A$ acts on $N$ and $\overline{N}$ by
\begin{equation*}
  e^{tH}n_xe^{-tH} = n_{e^{2t}x}  \qtextq{and} 
  e^{tH}\overline{n}_xe^{-tH} = \overline{n}_{e^{-2t}x}
\end{equation*}
for $x\in\RR^n$, $t\in\RR$. The following decomposition holds
\begin{equation*}
  \frakg = \overline{\frakn}\oplus\frakm\oplus\fraka\oplus\frakn 
  \qqtext{(Gelfand--Naimark decomposition)}.
\end{equation*}
The groups
\begin{equation*}
  P := MAN  \qtextq{and} \overline{P} := MA\overline{N}=\theta(P)
\end{equation*}
are opposite parabolic subgroups in $G$ and $\overline{N}P\subseteq G$
is an open dense subset. Let $W:=N_K(\fraka)/Z_K(\fraka)$ be the Weyl
group corresponding to $\fraka$. Then $W=\{\1,[w_0]\}$ where the
non-trivial element is represented by the matrix
\begin{equation*}
  w_0 = \diag(-1,1,\ldots,1)\in K.
\end{equation*}
The element $w_0$ has the property that $w_0Nw_0^{-1}=\overline{N}$
and hence $w_0Pw_0^{-1}=\overline{P}$. More precisely,
\begin{equation*}
  w_0n_xw_0^{-1} = \overline{n}_{-x}  \qtextq{and} 
  w_0e^{tH}w_0^{-1} = e^{-tH}.
\end{equation*}
We have the disjoint union
\begin{equation*}
  G = \overline{P}\cup\overline{P}w_0\overline{P} \qqtext{(Bruhat
    decomposition)}.
\end{equation*}
The following lemma is a straightforward calculation:

\begin{lem}\label{lem:w0nbarDecomposition}
  For $x\in\RR^n$, $x\neq0$, we have
  $w_0^{-1}\overline{n}_x=\overline{n}_yme^{tH}n_z\in\overline{N}P$
  with
  \begin{equation*}
    \begin{split}
      y &= -\abs{ x }^{-2}x,\\
      z &= \abs{ x }^{-2}x,\\
      t &= \log\abs{ x },
    \end{split}
    \qquad
     m = \left(
      \begin{array}{ccc}
        -1&&\\&\1_n-2\abs{ x }^{-2}xx^t&\\&&-1
      \end{array}
    \right).
  \end{equation*}
\end{lem}

Let $\tau$ be the involution of $G$ given by conjugation with the
matrix
\begin{equation*}
  \diag(\1_m,-\1_{n-m},1).
\end{equation*}
Then the symmetric subgroup $H:=G^\tau$ is isomorphic to
$O(1,m+1)\times O(n-m)$. The subgroup $H$ is generated by the
subgroups $N_H$, $\overline{N}_H$, $M_H$ and $A$, where (viewing
$\RR^m$ as the subspace $\RR^m\times\{0\}\subseteq\RR^n$)
\begin{equation*}
  N_H := \Set{ n_x }{ x\in\RR^m } \qtextq{and}  \overline{N}_H :=
  \Set{ \overline{n}_x }{ x\in\RR^m }
\end{equation*}
and $M_H:=M_H^+\cup m_0M_H^+$ with
\begin{equation*}
  M_H^+ := \Set{ \diag(1,k_1,k_2,1) }{ k_1\in O(m),k_2\in O(n-m)
  }\cong O(m)\times O(n-m).
\end{equation*}
Also denote by
\begin{equation*}
  P_H := M_HAN_H  \qtextq{and}  \overline{P}_H :=
  M_HA\overline{N}_H
\end{equation*}
the corresponding parabolic subgroups. We write $\frakh$ for the Lie
algebra of $H$.

\subsection{Principal series representations -- non-compact picture
  and standard intertwining operators}\label{sec:PrincipalSeriesReps}

We identify $\fraka_\CC^*$ with $\CC$ by $\lambda\mapsto\lambda(H)$,
i.e.\ $\sigma\in\CC$ corresponds to
$\sigma\gamma\in\fraka_\CC^*$. Under this identification $\rho$
corresponds to $n$. For $\sigma\in\CC$ let $e^\sigma$ be the character
of $A$ given by $e^\sigma(e^{tH})=e^{\sigma t}$, $t\in\RR$. Further,
for $\varepsilon\in\ZZ/2\ZZ$ denote by $\xi_\varepsilon$ the character
of $M=M^+\cup m_0 M^+$ with $\xi_\varepsilon(m_0)=(-1)^\varepsilon$
and $\xi_\varepsilon(m)=1$ for $m\in M^+$. For $\sigma\in\CC$ and
$\varepsilon\in\ZZ/2\ZZ$ we consider the character
$\chi_{\sigma,\varepsilon}:=\xi_\varepsilon\otimes e^\sigma\otimes\1$
on $P=MAN$ and induce it to a representation of $G$:
\begin{align*}
  \widetilde{I}^G_{\sigma,\varepsilon} :={}&
  \Ind_P^G(\chi_{\sigma,\varepsilon})
  \\
  ={}& \Set{ f\in C^\infty(G) }{
    f(gman)=\xi_\varepsilon(m)^{-1}a^{-\sigma-\rho}f(g)\,\forall\,g\in
    G,man\in P=MAN }.
\end{align*}
The group $G$ acts on $\widetilde{I}^G_{\sigma,\varepsilon}$ by
left-translations and this action will be denoted by
$\widetilde{\pi}_{\sigma,\varepsilon}^G$. Restricting to $K$ it is easy to see that the $K$-type
decomposition of the representations $\widetilde{\pi}_{\sigma,\varepsilon}^G$ is
given by
\begin{equation}
  \widetilde{\pi}_{\sigma,\varepsilon}^G\restrictedto[\big] K \cong
  \sideset{}{^\oplus}\sum_{k=0}^\infty
  {\,\sign^{\varepsilon+k}\boxtimes\,\calH^k(\RR^{n+1})},\label{eq:KtypeDecomp}
\end{equation}
where $\sign$ denotes the non-trivial character of $O(1)$, and $O(n+1)$ acts as usual on the space $\calH^k(\RR^{n+1})$ of spherical harmonics of degree $k$ on $\RR^n$, giving combined the action of $K\cong O(1)\times O(n+1)$.

The following fact on the structure of 
 $\widetilde{\pi}_{\sigma,\varepsilon}^G$ is known (see \cite{JW77}).
\begin{fact}
\begin{enumerate}
\item The representation $(\widetilde{\pi}_{\sigma,\varepsilon}^G,\widetilde{I}_{\sigma,\varepsilon}^G)$ is irreducible if and only if $\sigma\notin\pm(n+2\NN)$. It is unitarizable if and only if $\sigma\in(-n,n)\cup i\RR$.
\item For $\sigma=n+2u$, $u\in\NN$, the representation $(\widetilde{\pi}_{\sigma,\varepsilon}^G,\widetilde{I}_{\sigma,\varepsilon}^G)$ has a unique non-trivial subrepresentation $(\widetilde{\pi}_{\sigma,\varepsilon,\sub}^G,\widetilde{I}_{\sigma,\varepsilon,\sub}^G)$. This subrepresentation is irreducible and unitarizable and its $K$-type decomposition is given by
\begin{equation}
  \widetilde{\pi}_{\sigma,\varepsilon,\sub}^G\restrictedto[\big] K \cong
  \sideset{}{^\oplus}\sum_{k=u+1}^\infty
  {\,\sign^{\varepsilon+k}\boxtimes\,\calH^k(\RR^{n+1})},\label{eq:KtypeDecompSub}
\end{equation}
\item For $\sigma=-n-2u$, $u\in\NN$, the representation $(\widetilde{\pi}_{\sigma,\varepsilon}^G,\widetilde{I}_{\sigma,\varepsilon}^G)$ has a unique non-trivial subrepresentation $(\widetilde{\pi}_{\sigma,\varepsilon,\sub}^G,\widetilde{I}_{\sigma,\varepsilon,\sub}^G)$. This subrepresentation is finite-dimensional and irreducible and the quotient $\widetilde{I}_{\sigma,\varepsilon}^G/\widetilde{I}_{\sigma,\varepsilon,\sub}^G$ is unitarizable and isomorphic to $\widetilde{I}_{-\sigma,\varepsilon,\sub}^G$.
\end{enumerate}
\end{fact}

Since $\overline{N}P\subseteq G$ is dense, a function in
$\widetilde{I}_{\sigma,\varepsilon}^G$ is already uniquely determined
by its values on $\overline{N}$ and for
$f\in\widetilde{I}_{\sigma,\varepsilon}^G$ we put
\begin{equation*}
  f_{\overline{N}}(x) := f(\overline{n}_x),  \qqtext{$x\in\RR^n$}.
\end{equation*}
Let $I_{\sigma,\varepsilon}^G:=\Set{ f_{\overline{N}} }{
  f\in\widetilde{I}_{\sigma,\varepsilon}^G }$ and denote by
$\pi_{\sigma,\varepsilon}^G$ the corresponding induced action, i.e.
\begin{equation*}
  \pi_{\sigma,\varepsilon}^G(g)f_{\overline{N}} :=
  (\widetilde{\pi}_{\sigma,\varepsilon}^G(g)f)_{\overline{N}}, 
  \qqtext{$f\in\widetilde{I}_{\sigma,\varepsilon}^G$}.
\end{equation*}
Further, denote by $(\pi_{\sigma,\varepsilon,\sub}^G,I_{\sigma,\varepsilon,\sub}^G)$ the corresponding subrepresentations of $(\pi_{\sigma,\varepsilon}^G,I_{\sigma,\varepsilon}^G)$ for $\sigma\in\pm(n+2\NN)$.
Note that if $f\in\widetilde{I}_{\sigma,\varepsilon}^G$ is a $\frakk$-fixed vector, namely, in the $k=0$ term on the right hand side of \eqref{eq:KtypeDecomp}, then $f_{\overline{N}}$ is equal to $(1+|x|^2)^{-\frac{\sigma+n}{2}}$ up to a constant multiple.

In view of the Bruhat decomposition
$G=\overline{P}\cup\overline{P}w_0\overline{P}$ the action $\pi_{\sigma,\varepsilon}^G$
can be completely described by the action of $\overline{P}$ and $w_0$. Using
Lemma~\ref{lem:w0nbarDecomposition} we find
\begin{alignat*}{2}
  \pi_{\sigma,\varepsilon}^G(\overline{n}_a)f(x) &= f(x-a), &&\qquad
  \overline{n}_a\in\overline{N},
  \\
  \pi_{\sigma,\varepsilon}^G(m)f(x) &= f(m^{-1}x), &&\qquad m\in M^+\cong
  O(n),
  \\
  \pi_{\sigma,\varepsilon}^G(m_0)f(x) &= (-1)^\varepsilon f(-x),
  \\
  \pi_{\sigma,\varepsilon}^G(e^{tH})f(x) &= e^{(\sigma+n)t}f(e^{2t}x),
  &&\qquad e^{tH}\in A,
  \\
  \pi_{\sigma,\varepsilon}^G(w_0)f(x) &=
  (-1)^\varepsilon\abs{ x }^{-\sigma-n}f(-\abs{ x }^{-2}x).
\end{alignat*}
This also gives the following expressions for the differential action
$\di\pi_\sigma^G=\di\pi_{\sigma,\varepsilon}^G$ of the Lie algebra
$\frakg$, which is independent of $\varepsilon$:
\begin{alignat*}{2}
  \di\pi_\sigma^G(\overline{N}_j)f(x) &= -\frac{\partial f}{\partial
    x_j}(x), &&\qquad  j=1,\ldots,n,
  \\
  \di\pi_\sigma^G(T)f(x) &= -D_{Tx}f(x), &&\qquad T\in\frakm\cong\so(n),
  \\
  \di\pi_\sigma^G(H)f(x) &= \left(2E+\sigma+n\right)f(x),
  \\
  \di\pi_\sigma^G(N_j)f(x) &= -\abs{ x }^2\frac{\partial f}{\partial
    x_j}(x)+x_j\left(2E+\sigma+n\right)f(x), &&\qquad j=1,\ldots,n,
\end{alignat*}
where $D_a$ denotes the directional derivative in direction
$a\in\RR^n$ and $E=\smash{\sum_{j=1}^n{x_j\frac{\partial}{\partial x_j}}}$ is
the Euler operator on $\RR^n$. For the action of $\frakn$ we have used
the identity
$\di\pi_\sigma^G(N_a)=\pi_{\sigma,\varepsilon}^G(w_0)\di\pi_\sigma^G(\overline{N}_{-a})\pi_{\sigma,\varepsilon}^G(w_0^{-1})$.

Now consider the normalized Knapp--Stein
intertwining operators
$\widetilde{J}(\sigma,\varepsilon):\widetilde{I}_{\sigma,\varepsilon}^G
\to\widetilde{I}_{-\sigma,\varepsilon}^G$ 
which are for $\Real\sigma>0$ given by
\begin{equation*}
  \widetilde{J}(\sigma,\varepsilon)f(g) :=
  \frac{(-1)^\varepsilon}{\Gamma(\frac{\sigma}{2})}\int_{\overline{N}}{f(gw_0\overline{n})\td\overline{n}}, 
  \qqtext{$g\in G$,  $f\in\widetilde{I}_{\sigma,\varepsilon}^G$},
\end{equation*}
where $\di\overline{n}$ is the Haar measure on $\overline{N}$ given by
the push-forward of the Lebesgue measure on $\RR^n$ by the map
$\RR^n\to\overline{N},\,x\mapsto\overline{n}_x$, and extended analytically to all $\sigma\in\CC$. This intertwining
operator induces an intertwining operator
$J(\sigma,\varepsilon):I_{\sigma,\varepsilon}^G\to
I_{-\sigma,\varepsilon}^G$ by
$J(\sigma,\varepsilon)f_{\overline{N}}:=(\widetilde{J}(\sigma,\varepsilon)f)_{\overline{N}}$,
$f\in\widetilde{I}_{\sigma,\varepsilon}^G$. Using
Lemma~\ref{lem:w0nbarDecomposition} we obtain
\begin{align*}
  J(\sigma,\varepsilon)f_{\overline{N}}(x) &= \frac{(-1)^\varepsilon}{\Gamma(\frac{\sigma}{2})}\int_{\RR^n}{f(\overline{n}_xw_0\overline{n}_z)\td z}\\
  &=
  \frac{1}{\Gamma(\frac{\sigma}{2})}\int_{\RR^n}{\abs{ z }^{-\sigma-n}f_{\overline{N}}(x-\abs{ z }^{-2}z)\td
    z}.
\end{align*}
Consider the coordinate change $y:=x-\abs{ z }^{-2}z$. Its Jacobian
$\abs{ \det(\frac{\partial y}{\partial z}) }$ is homogeneous of degree
$-2n$, $O(n)$-invariant and has value $1$ for $z=e_1$. Hence it is
equal to $\abs{ z }^{-2n}$. This finally gives
\begin{equation}
  J(\sigma,\varepsilon)f(x) = \frac{1}{\Gamma(\frac{\sigma}{2})}\int_{\RR^n}{\abs{ x-y }^{\sigma-n}f(y)\td y} = \frac{1}{\Gamma(\frac{\sigma}{2})}(\abs{}^{\sigma-n}*f)(x),\label{eq:KnappSteinConvolution}
\end{equation}
so $J(\sigma,\varepsilon)$ is up to a constant given by convolution with the
distribution $\abs{}^{\sigma-n}$ and its residues. We define a $G$-invariant Hermitian form
$\Hermit{}{}_{\sigma,\varepsilon}$ on $I_{\sigma,\varepsilon}^G$ by
\begin{equation}
  \Hermit{f}{g}_{\sigma,\varepsilon} :=
  \Hermit{f}{J(\sigma,\varepsilon)g}_{L^2(\RR^n)} =
  \frac{1}{\Gamma(\frac{\sigma}{2})}\int_{\RR^n}{\int_{\RR^n}{\abs{ x-y }^{\sigma-n}f(x)\overline{g(y)}\td
      x}\td y}.\label{eq:InnerProductNonCptPicture}
\end{equation}

For $\sigma\in(-n,n)$ this form is in fact positive definite and in
this case the completion $\calH_{\sigma,\varepsilon}^G$ of
$I_{\sigma,\varepsilon}^G$ with respect to the inner product
$\Hermit{}{}_{\sigma,\varepsilon}$ gives an irreducible unitary
representation
$(\calH_{\sigma,\varepsilon}^G,\pi_{\sigma,\varepsilon}^G)$ of
$G$. The intertwining operator extends to an (up to scalar) unitary isomorphism
$J(\sigma,\varepsilon):\calH_{\sigma,\varepsilon}^G\to\calH_{-\sigma,\varepsilon}^G$.
These representations comprise the \textit{complementary series}.

For $\sigma=n+2u$, $u\in\NN$ the operator $J(-\sigma,\varepsilon)$ vanishes on the
finite-dimensional subrepresentation $I_{-\sigma,\varepsilon,\sub}$ and maps onto
the infinite-dimensional subrepresentation $I_{\sigma,\varepsilon,\sub}$. Therefore the Hermitian
form $\Hermit{}{}_{-\sigma,\varepsilon}$ vanishes on the finite-dimensional
subrepresentation $I_{-\sigma,\varepsilon,\sub}^G$ and induces a $G$-invariant
positive definite Hermitian form on the unitarizable quotient
$I_{-\sigma,\varepsilon}^G/I_{-\sigma,\varepsilon,\sub}^G$. Since this quotient is isomorphic
to the subrepresentation $I_{\sigma,\varepsilon,\sub}^G$ via the intertwining operator
$J(-\sigma,\varepsilon)$ we also obtain an irreducible
unitary representation $(\calH_{\sigma,\varepsilon,\sub}^G,\pi_{\sigma,\varepsilon,\sub}^G)$.
These representations are isomorphic to the unitarizations of certain
 Zuckerman's modules $A_\frakq(\lambda)$ of $G$
 and occur discretely in the Plancherel formula
for the hyperboloids $O(1,n+1)/O(1,n)$. We call them \textit{discrete series representations for the hyperboloid}.

Finally, for $\sigma\in i\RR$ the usual $L^2$-inner product on $\RR^n$ provides
unitarizations
$(\calH_{\sigma,\varepsilon}^G,\pi_{\sigma,\varepsilon}^G)$ on
$\calH_{\sigma,\varepsilon}^G=L^2(\RR^n)$ and these representations
form the \textit{unitary principal series}.

Note that for any
$\sigma\in i\RR\cup(-n,n)$ the intertwining operator
$J(\sigma,\varepsilon)$ extends to an isometry between the
irreducible unitary representations
$(\calH_{\sigma,\varepsilon}^G,\pi_{\sigma,\varepsilon}^G)$ and
$(\calH_{-\sigma,\varepsilon}^G,\pi_{-\sigma,\varepsilon}^G)$.

\subsection{The Fourier transformed picture}\label{sec:FTpicture}

Consider the Euclidean Fourier transform
$\calF_{\RR^n}:\calS'(\RR^n)\to\calS'(\RR^n)$ given by
\begin{equation}
  \calF_{\RR^n}u(x) =
  (2\pi)^{-\frac{n}{2}}\int_{\RR^n}{e^{-i\Hermit{x}{y}}u(y)\td
    y}.\label{eq:DefFourierTransform}
\end{equation}
For $\varepsilon\in\ZZ/2\ZZ$ and $\sigma\in i\RR\cup(-n,n)$ resp.\ $\sigma\in(n+2\NN)$ we define
a representation $\rho_{\sigma,\varepsilon}^G$ of $G$ on
$\calF_{\RR^n}^{-1}I_{\sigma,\varepsilon}^G$ resp.\ $\calF_{\RR^n}^{-1}I_{\sigma,\varepsilon,\sub}^G$ by
\begin{equation*}
  \pi_{\sigma,\varepsilon}^G(g)\circ\calF_{\RR^n} =
  \calF_{\RR^n}\circ\rho_{\sigma,\varepsilon}^G(g),  \qqtext{$g\in G$}.
\end{equation*}
resp.
\begin{equation*}
  \pi_{\sigma,\varepsilon,\sub}^G(g)\circ\calF_{\RR^n} =
  \calF_{\RR^n}\circ\rho_{\sigma,\varepsilon}^G(g),  \qqtext{$g\in G$}.
\end{equation*}
It is easy to calculate the group action of
$\overline{P}=MA\overline{N}$:
\begin{alignat}{2}
  \rho_{\sigma,\varepsilon}(\overline{n}_a)f(x) &= e^{i\Hermit{x}{a}}f(x), &\qquad
  \overline{n}_a&\in\overline{N},\label{eq:FPGroupAction1}
  \\
  \rho_{\sigma,\varepsilon}(m)f(x) &= f(m^{-1}x), & m&\in M^+\cong
  O(n),\label{eq:FPGroupAction2}
  \\
  \rho_{\sigma,\varepsilon}(m_0)f(x) &= (-1)^\varepsilon
  f(-x),\label{eq:FPGroupAction3}
  \\
  \rho_{\sigma,\varepsilon}(e^{tH})f(x) &= e^{(\sigma-n)t}f(e^{-2t}x),
  & t&\in\RR.\label{eq:FPGroupAction4}
\end{alignat}
The action of $w_0$ in the Fourier transformed picture is more
involved (see e.g.\ \cite[Proposition 2.3]{VG06}). Note that by these
formulas the restriction $\rho_{\sigma,\varepsilon}\restrictedto {\overline{P}}$
also acts on $C^\infty(\RR^m\minuszero )$. Using the classical
intertwining relations
\begin{align*}
  x_j\circ\calF_{\RR^n} &=
  \calF_{\RR^n}\circ(-i\tfrac{\partial}{\partial x_j}),
  \\
  \tfrac{\partial}{\partial x_j}\circ\calF_{\RR^n} &=
  \calF_{\RR^n}\circ(-ix_j)
\end{align*}
it is easy to compute the differential action $\di\rho_\sigma^G$ of
$\rho_{\sigma,\varepsilon}^G$:
\begin{alignat}{2}
  \di\rho_\sigma^G(\overline{N}_j)f(x) &= ix_jf(x), &\qquad
  j&=1,\ldots,n,\label{eq:FPAlgebraAction1}
  \\
  \di\rho_\sigma^G(T)f(x) &= -D_{Tx}f(x), &
  T&\in\frakm\cong\so(n),\label{eq:FPAlgebraAction2}
  \\
  \di\rho_\sigma^G(H)f(x) &=
  -\left(2E-\sigma+n\right)f(x),\label{eq:FPAlgebraAction3}
  \\
  \di\rho_\sigma^G(N_j)f(x) &= -i\calB_j^{n,\sigma}f(x), &
  j&=1,\ldots,n,\label{eq:FPAlgebraAction4}
\end{alignat}
where we abbreviate
\begin{equation*}
  \calB_j^{n,\sigma} := x_j\Delta
  -\left(2E-\sigma+n\right)\frac{\partial}{\partial x_j}.
\end{equation*}
The operators $\calB_j^{n,\sigma}$ are called \emph{Bessel
  operators} and were for $G=O(1,n+1)$ first studied in
\cite{HKM12}.  Before that similar operators
for minimal representations of $O(p,q)$ were studied in \cite{KM08} 
and are called fundamental differential operators there.
 Bessel operators are polynomial differential operators on $\RR^n$
and hence the action $\di\rho_\sigma^G$ defines a representation of
$\frakg$ on $C^\infty(\Omega)$ for every open subset
$\Omega\subseteq\RR^n$.

To describe the representation spaces
$\calH_{\sigma,\varepsilon}^G$ resp.\ 
$\calH_{\sigma,\varepsilon,\sub}^G$ in the Fourier
transformed picture we recall that the Fourier transform
$\calF_{\RR^n}$ intertwines convolution and multiplication
operators. Further, the Riesz distributions
$R_\lambda\in\calS'(\RR^n)$ given by
\begin{equation*}
  \innerp{ R_\lambda }{ \varphi } =
  \frac{2^{-\frac{\lambda}{2}}}{\Gamma(\frac{\lambda+n}{2})}\int_{\RR^n}{\varphi(x)\abs{ x }^\lambda\td
    x},  \qqtext{$\varphi\in\calS(\RR^n)$},
\end{equation*}
for $\Real\lambda>-n$ and extended analytically to $\lambda\in\CC$
satisfy the following classical functional equation (see
\cite[equation (2') in II.3.3]{GS64})
\begin{equation*}
  \calF_{\RR^n}R_\lambda = R_{-\lambda-n}.
\end{equation*}
With this observation as well as \eqref{eq:KnappSteinConvolution} and \eqref{eq:InnerProductNonCptPicture} we see that in the
Fourier transformed picture the $G$-invariant inner product is simply given by
the inner product of $L^2(\RR^n,\abs{ x }^{-\Real\sigma}\td x)$ and hence the map $\calF_{\RR^n}^{-1}:I_{\sigma,\varepsilon}^G\to L^2(\RR^n,\abs{ x }^{-\Real\sigma}\td x)$  extends to $\calF_{\RR^n}^{-1}:\calH_{\sigma,\varepsilon}^G\to L^2(\RR^n,\abs{ x }^{-\Real\sigma}\td x)$ for $\sigma\in i\RR\cup(-n,n)$.
Similarly, the map $\calF_{\RR^n}^{-1}:I_{\sigma,\varepsilon,\sub}^G\to L^2(\RR^n,\abs{ x }^{-\Real\sigma}\td x)$ extends to $\calF_{\RR^n}^{-1}:\calH_{\sigma,\varepsilon,\sub}^G\to L^2(\RR^n,\abs{ x }^{-\Real\sigma}\td x)$ for $\sigma\in(n+2\NN)$. We note that already the action of the parabolic subgroup $\overline{P}$ given by \eqref{eq:FPGroupAction1}--\eqref{eq:FPGroupAction4} extends to an irreducible unitary representation of $\overline{P}$ on $L^2(\RR^n,\abs{x}^{-\Real\sigma}\td x)$ by Mackey theory and therefore $\calF_{\RR^n}^{-1}\calH_{\sigma,\varepsilon}^G$ resp.\ $\calF_{\RR^n}^{-1}\calH_{\sigma,\varepsilon,\sub}^G$ is actually equal to $L^2(\RR^n,\abs{ x }^{-\Real\sigma}\td x)$. This yields the $L^2$-realizations
$$ (\rho_{\sigma,\varepsilon}^G,L^2(\RR^n,\abs{x}^{-\Real\sigma}\td x)) $$
for $\sigma\in i\RR\cup(-n,n)\cup(n+2\NN)$. The Fourier transform is a
unitary (up to scalar multiples) isomorphism
$\calF_{\RR^n}:L^2(\RR^n,\abs{ x }^{-\Real\sigma}\td
x)\to\calH_{\sigma,\varepsilon}^G$ resp.\ $\calF_{\RR^n}:L^2(\RR^n,\abs{ x }^{-\Real\sigma}\td
x)\to\calH_{\sigma,\varepsilon,\sub}^G$ intertwining the representations
$\rho_{\sigma,\varepsilon}^G$ and $\pi_{\sigma,\varepsilon}^G$ resp.\ $\pi_{\sigma,\varepsilon,\sub}^G$.
For $\sigma\in i\RR\cup(-n,n)$ the
standard intertwining operators $J(\sigma,\varepsilon)$ are in this
picture (up to scalar multiples) given by multiplication
\begin{equation*}
  L^2(\RR^n,\abs{ x }^{-\Real\sigma}\td x)\to L^2(\RR^n,\abs{ x }^{\Real\sigma}\td
  x),\qquad f(x)\mapsto\abs{ x }^{-\sigma}f(x).
\end{equation*}

The explicit $K$-type decompositions \eqref{eq:KtypeDecomp} and \eqref{eq:KtypeDecompSub} are difficult to see
in the Fourier transformed picture. However, one can still 
describe the spaces of $K$-finite vectors. 
For this consider the
renormalized $K$-Bessel function $\widetilde{K}_\alpha(z)$ from
Appendix~\ref{app:KBessel}. 
For $\sigma\in i\RR\cup(-n,n)\cup(n+2\NN)$ put
\begin{equation}
  \psi_\sigma^G(x) := \widetilde{K}_{-\frac{\sigma}{2}}(\abs{ x }), 
  \qqtext{$x\in\RR^n\minuszero$}. \label{eq:L2SphericalVector}
\end{equation}
By Appendix~\ref{app:FourierHankel} and the integral formula \eqref{eq:IntFormulaJBesselHypergeometric}, $\calF_{\RR^n}\psi_\sigma^G$ is equal to $(1+|x|^2)^{-\frac{\sigma+n}{2}}$ up to a constant multiple and hence $\psi_\sigma^G$ is a $\frakk$-fixed vector in $\calF_{\RR^n}^{-1}I_{\sigma,\varepsilon}^G$.
If $\sigma\in i\RR\cup(-n,n)$, it constitutes the minimal $\frakk$-type in $L^2(\RR^n,\abs{x}^{-\Real\sigma}\td x)$. Note that
 as $K$-representation the minimal $K$-type is for $\varepsilon\neq0$
not the trivial representation since $m_0\in K$ acts on it
by~$(-1)^\varepsilon$. To describe the underlying $(\frakg,K)$-module
we denote for $f\in C^\infty(\RR_+)$ and $k\in\NN$ by
$f\otimes\abs{ x }^{2k}$ the function $f(\abs{ x })\abs{ x }^{2k}$ and by
$f\otimes\abs{ x }^{2k}\CC[x]$ the space of all functions of
the form $f(\abs{ x })\abs{ x }^{2k}p(x)$ for some polynomial
$p\in\CC[x]$.

\begin{lem}\label{lem:gKmodule}
For $\sigma\in i\RR\cup(-n,n)\cup(n+2\NN)$ the underlying $(\frakg,K)$-module
 of $\calF_{\RR^n}^{-1}I_{\sigma,\varepsilon}^G$ is given by
  \begin{equation}
 (\calF_{\RR^n}^{-1}I_{\sigma,\varepsilon}^G)_K=
    \sum_{k=0}^\infty{\widetilde{K}_{-\frac{\sigma}{2}+k}
      \otimes\abs{ x }^{2k}\CC[x]} =
    \sum_{k=0,1}{\widetilde{K}_{-\frac{\sigma}{2}+k}
      \otimes\abs{ x }^{2k}\CC[x]}.\label{eq:gKmodulePs} 
  \end{equation}
If $\sigma\in i\RR\cup(-n,n)$ the underlying $(\frakg,K)$-module of the representation
  $(\rho_{\sigma,\varepsilon},L^2(\RR^n,\abs{ x }^{-\Real\sigma}\td x))$ is
  given by
  \begin{equation}
    L^2(\RR^n,\abs{ x }^{-\Real\sigma}\td x)_K =
    \sum_{k=0}^\infty{\widetilde{K}_{-\frac{\sigma}{2}+k}
      \otimes\abs{ x }^{2k}\CC[x]}. \label{eq:gKmodule} 
  \end{equation}
\end{lem}

\begin{proof}
  Since $\frakg=\frakk+\fraka+\overline{\frakn}$ the universal
  enveloping algebra $\calU(\frakg)$ of $\frakg$ decomposes by the
  Poincar\'{e}--Birkhoff--Witt Theorem into
  $\calU(\frakg)=\calU(\overline{\frakn})\calU(\fraka)\calU(\frakk)$. The
  $(\frakg,K)$-module $(\calF_{\RR^n}^{-1}I_{\sigma,\varepsilon}^G)_K$ is
  generated by the $\frakk$-fixed vector $\psi_\sigma^G$ and hence
  \begin{equation*}
    (\calF_{\RR^n}^{-1}I_{\sigma,\varepsilon}^G)_K = \calU(\frakg)\psi_\sigma^G
    = \calU(\overline{\frakn})\calU(\fraka)\psi_\sigma^G.
  \end{equation*}
  By \eqref{eq:FPAlgebraAction1} and \eqref{eq:FPAlgebraAction3} we
  have $\calU(\overline{\frakn})=\CC[x]$ and
  $\calU(\fraka)=\CC[E]$. Using \eqref{eq:DerivativeKBessel} we
  further find that the Euler operator $E$ acts on functions of the
  form $\widetilde{K}_{\alpha}(\abs{ x })\abs{ x }^{2k}$, $\alpha\in\RR$,
  $k\in\NN$, by
  \begin{equation*}
    E\left(\widetilde{K}_\alpha(\abs{ x })\abs{ x }^{2k}\right) =
    -\tfrac{1}{2}\widetilde{K}_{\alpha+1}(\abs{ x })
    \abs{ x }^{2k+2}+2k\widetilde{K}_\alpha(\abs{ x })\abs{ x }^{2k}
  \end{equation*}
  Hence
  $$ \calU(\overline{\frakn})\calU(\fraka)\psi_\sigma^G =
    \calU(\overline{\frakn})\sum_{k=0}^\infty
    \CC\left(\widetilde{K}_{-\frac{\sigma}{2}+k}\otimes\abs{ x }^{2k}\right) =
    \sum_{k=0}^\infty{\widetilde{K}_{-\frac{\sigma}{2}+k}\otimes\abs{ x }^{2k}\CC[x]} $$
  which proves the first equality in \eqref{eq:gKmodulePs}. The second equality in \eqref{eq:gKmodulePs} follows immediately from \eqref{eq:RecRelKBessel}.
Since the $K$-finite vectors do not depend on the globalization, \eqref{eq:gKmodule} follows.
\end{proof}

Now let $\sigma=n+2u$, $u\in\NN$. For $f\in C^\infty(\RR_+)$ and $j,k\in\NN$ we denote by $f\otimes\abs{x}^{2k}\calH^j(\RR^n)$ the space of functions of the form $f(|x|)|x|^{2k}p(x)$ with $p\in\calH^j(\RR^n)$. Further, let $\CC[x]_{>j}$ be the space of all polynomials which are sums of homogeneous polynomials of degree $>j$. Then $f\otimes\abs{x}^{2k}\CC[x]_{>j}$ denotes the space of functions of the form $f(|x|)|x|^{2k}p(x)$ with $p\in\CC[x]_{>j}$.

\begin{lem}
Let $\sigma=n+2u$, $u\in\NN$. The lowest $K$-type $\sign^{\varepsilon+u+1}\boxtimes\,\calH^{u+1}(\RR^{n+1})$ in the representation $(\rho_{\sigma,\varepsilon},L^2(\RR^n,\abs{ x }^{-\Real\sigma}\td x))$ is given by
\begin{equation}
 \bigoplus_{k=0}^{u+1} \widetilde{K}_{-\frac{\sigma}{2}+k}\otimes\abs{x}^{2k}\calH^{u-k+1}(\RR^n).\label{eq:LowestKTypeSubreps}
\end{equation}
and for the underlying $(\frakg,K)$-module the following inclusion holds:
\begin{equation}
 L^2(\RR^n,\abs{ x }^{-\Real\sigma}\td x)_K \subseteq
    \sum_{k=0}^\infty{\widetilde{K}_{-\frac{\sigma}{2}+k}
      \otimes\abs{ x }^{2k}\CC[x]_{>u-k}}.\label{eq:gKmoduleSub} 
\end{equation}
\end{lem}

\begin{proof}
By \eqref{eq:gKmodulePs}, we see that \eqref{eq:LowestKTypeSubreps} is contained in $(\calF^{-1}_{\RR^n}I_{\sigma,\varepsilon}^G)_K$. We now show that \eqref{eq:LowestKTypeSubreps} is a $K$-subrepresentation of $(\calF^{-1}_{\RR^n}I_{\sigma,\varepsilon}^G)_K$.  
The group $O(n)$ leaves $\calH^j(\RR^n)$ invariant for every $j\in\NN$ and hence $M$ leaves each summand in \eqref{eq:LowestKTypeSubreps} invariant. It remains to show that $\frakk\cap(\frakn+\overline{\frakn})=\linspan\{\overline{N}_j-N_j:j=1,\ldots,n\}$ leaves \eqref{eq:LowestKTypeSubreps} invariant. An easy calculation using \eqref{eq:DerivativeKBessel} and \eqref{eq:RecRelKBessel} shows that for $p\in\calH^{u-k+1}(\RR^n)$ we have
\begin{align*}
 & \di\rho_\sigma^G(\overline{N}_j-N_j)\left[\widetilde{K}_{-\frac{\sigma}{2}+k}(\abs{x})\abs{x}^{2k}p(x)\right]\\
 ={}& i(\calB_j^{n,\sigma}+x_j)\left[\widetilde{K}_{-\frac{\sigma}{2}+k}(\abs{x})\abs{x}^{2k}p(x)\right]\\
 ={}& 4ik\widetilde{K}_{-\frac{\sigma}{2}+k-1}(\abs{x})\abs{x}^{2k-2}p_j^+(x) + i(n+2u-k)\widetilde{K}_{-\frac{\sigma}{2}+k+1}(\abs{x})\abs{x}^{2k+2}p_j^-(x),
\end{align*}
where $x_jp=p_j^++\abs{x}^2p_j^-$ with $p_j^\pm\in\calH^{u-k+1\pm1}(\RR^n)$ given by
$$ p_j^+ = x_jp-\frac{\abs{x}^2}{2u-2k+n}\frac{\partial p}{\partial x_j}, \qquad p_j^- = \frac{1}{2u-2k+n}\frac{\partial p}{\partial x_j}. $$
As in \cite[Lemma B.1.2]{Moe10}, we can construct an explicit isomorphism between \eqref{eq:LowestKTypeSubreps} and $\sign^{\varepsilon+u+1}\boxtimes\,\calH^{u+1}(\RR^{n+1})$ which respects $K$-actions. In view of \eqref{eq:KtypeDecomp} the $K$-type $\sign^{\varepsilon+u+1}\boxtimes\,\calH^{u+1}(\RR^{n+1})$ occurs only once in $\calF^{-1}_{\RR^n}I_{\sigma,\varepsilon}^G|_K$ and hence \eqref{eq:LowestKTypeSubreps} must coincide with the lowest $K$-type in $\calF^{-1}_{\RR^n}I_{\sigma,\varepsilon,\sub}^G$ and in $L^2(\RR^n,\abs{ x }^{-\Real\sigma}\td x)$. This shows the first part of the claim.
Now an argument similar to the proof of Lemma~\ref{lem:gKmodule} shows the inclusion \eqref{eq:gKmoduleSub}.
\end{proof}

\subsection{Branching rule for $m=0$}\label{sec:Branchingm=0}

Using the $L^2$-model $(\rho_{\sigma,\varepsilon}^G,L^2(\RR^n,\abs{x}^{-\Real\sigma}\td x))$ we can now easily derive the branching rule for the restriction of $\rho_{\sigma,\varepsilon}^G$ to $H=O(1,1)\times O(n)$ (the case $m=0$). Taking conjugation if necessary we may and do assume that $H=MA\cup w_0MA$.  Note that $MA=SO(1,1)\times O(n)$ and the action of $MA$ is given by \eqref{eq:FPGroupAction2}, \eqref{eq:FPGroupAction3} and \eqref{eq:FPGroupAction4}. Therefore the isometric isomorphism
$$ L^2(\RR^n,\abs{x}^{-\Real\sigma}\td x)\to L^2(\RR^n), \quad f(x)\mapsto |x|^{-\frac{\sigma}{2}}f(x) $$
is an intertwining operator ${\rho_{\sigma,\varepsilon}^G\restrictedto {MA}}\to{\rho_{0,\varepsilon}^G\restrictedto {MA}}$ and the
restrictions $\rho_{\sigma,\varepsilon}^G \restrictedto {MA}$ are pairwise equivalent. The decomposition of $\rho_{0,\varepsilon}^G \restrictedto {MA}$ can be done by the Mellin transform with respect to the variable $|x|$ giving
$$ L^2(\RR^n) = \sideset{}{^\oplus}\sum_{k=0}^\infty \Bigg(\int_{i\RR}^\oplus
    {\pi'}_{\tau,\varepsilon+k}^{SO(1,1)}\td\tau\Bigg)\boxtimes\calH^k(\RR^n), $$
where ${\pi'}_{i\lambda,\delta}^{SO(1,1)}$ ($\lambda\in\RR$, $\delta\in\ZZ/2\ZZ$) denotes the unitary character of $SO(1,1)=A\cup m_0A$ given by
$$ {\pi'}_{i\lambda,\delta}^{SO(1,1)}(e^{tH}) = e^{i\lambda t}, \qquad {\pi'}_{i\lambda,\delta}^{SO(1,1)}(m_0) = (-1)^\delta. $$
Let $\pi_{i\lambda,\delta}^{O(1,1)}$ ($\lambda\in\RR_+$, $\delta\in\ZZ/2\ZZ$) denote the two-dimensional irreducible unitary representation of $O(1,1)=A\cup m_0A\cup w_0A\cup m_0w_0A$ given by
\begin{align*}
 &\pi_{i\lambda,\delta}^{O(1,1)}(e^{tH})
 = \left(
      \begin{array}{cc}
       e^{i\lambda t}&0\\0&e^{-i\lambda t}
     \end{array}\right), \quad 
 \pi_{i\lambda,\delta}^{O(1,1)}(m_0)=(-1)^{\delta}, \quad
 \pi_{i\lambda,\delta}^{O(1,1)}(w_0) = \left(
      \begin{array}{cc}
       0&1\\1&0
     \end{array}
    \right). 
\end{align*}
Any irreducible unitary representation of $O(1,1)$ is either isomorphic to $\pi_{i\lambda,\delta}^{O(1,1)}$ or a character which factors through $O(1,1)/A\simeq\ZZ/2\ZZ\times \ZZ/2\ZZ$.
Therefore, the only possibility for the branching law to $H$ is  
$$ L^2(\RR^n) = \sideset{}{^\oplus}\sum_{k=0}^\infty \Bigg(\int_{i\RR_+}^\oplus
 \pi_{\tau,\varepsilon+k}^{O(1,1)}\td\tau\Bigg)\boxtimes\calH^k(\RR^n). $$

\section{Reduction to an ordinary differential operator}
\label{sec:Reduction}

This section deals with the reduction of the branching problem for
$\rho_{\sigma,\varepsilon}^G \restrictedto H$ to an ordinary differential equation
on $\RR_+$. For this we assume $0<m<n$ throughout the rest of this paper.

Consider the $L^2$-realization $L^2(\RR^n,\abs{ (x,y) }^{-\Real\sigma}\td
x\td y)$ of the representation $\rho_{\sigma,\varepsilon}^G$ where we
split variables $(x,y)\in\RR^m\times\RR^{n-m}$. We realize the unitary
representations $\smash{\rho_{\tau,\delta}^{O(1,m+1)}}$ of the first factor
$O(1,m+1)$ of $H=O(1,m+1)\times O(n-m)$ in the same way on
$L^2(\RR^m,\abs{ x }^{-\Real\tau}\td x)$. For the second factor $O(n-m)$
denote by $\calH^k(\RR^{n-m})$ its representation on solid spherical
harmonics on $\RR^{n-m}$ of degree $k\in\NN$ by left-translation.

\begin{prop}\label{prop:IntertwiningProperty}
  Let $\sigma\in i\RR\cup(-n,n)\cup(n+2\NN)$ and $\tau\in i\RR\cup(-m,m)\cup(m+2\NN)$. For
  every solution $F\in C^\infty(\RR_+)$ to the second order ordinary
  differential equation
  \begin{align*}
    \MoveEqLeft
    t(1+t)u''(t)+\left(\tfrac{-\sigma+2k+n-m+2}{2}t+\tfrac{2k+n-m}{2}\right)u'(t)
    \\
    &+\tfrac{1}{4}\left(\left(\tfrac{-\sigma+2k+n-m}{2}\right)^2
      -\left(\tfrac{\tau}{2}\right)^2\right)u(t) = 0
  \end{align*}
  which is regular at $t=0$ the map
  \begin{align*}
    \Psi: C^\infty(\RR^m\minuszero )\boxtimes\calH^k(\RR^{n-m}) &\to
    C^\infty(\RR^n\setminus\{x=0\}),
    \\
    \Psi(f\otimes\phi)(x,y) &:=
    \abs{ x }^{\frac{\sigma-\tau-2k-n+m}{2}}F(\tfrac{\abs{ y }^2}{\abs{ x }^2})f(x)\phi(y),
  \end{align*}
  is $\overline{P}_H$- and $\frakh$-equivariant. 
Here, $\overline{P}_H$- and $\frakh$-actions on
  $C^\infty(\RR^m\minuszero)$ are given by
  \eqref{eq:FPGroupAction1}--\eqref{eq:FPAlgebraAction4}
  and actions on $C^\infty(\RR^n\setminus\{x=0\})$
  are given by the restriction of 
  \eqref{eq:FPGroupAction1}--\eqref{eq:FPAlgebraAction4}
  for $\overline{P}$ and $\frakg$.
\end{prop}

\begin{proof}
  Put $\mu:=2k+n-m$ and $\alpha:=\frac{\sigma-\tau-\mu}{2}$ so that
  \begin{equation*}
    \Psi(f\otimes\phi)(x,y) = \abs{ x }^\alpha
    F(\tfrac{\abs{ y }^2}{\abs{ x }^2})f(x)\phi(y).
  \end{equation*}
  Since $\frakh=\frakn_H+\frakm_H+\fraka+\overline{\frakn}_H$ it
  suffices to check the intertwining property for $\overline{N}_H$,
  $M_H$, $A$ and $\frakn_H$.
  \begin{enumerate}
  \item For $\overline{n}_a\in\overline{N}_H$ both
    $\rho^G_{\sigma,\varepsilon}(\overline{n}_a)$ and
    $\rho^{O(1,m+1)}_{\tau,\varepsilon+k}(\overline{n}_a)$ are by
    \eqref{eq:FPGroupAction1} the multiplication operators
    $e^{i\Hermit{x}{a}}$ and hence the intertwining property is clear.
  \item Let $m=\diag(1,k_1,k_2,1)\in M_H^+$, $k_1\in O(m)$, $k_2\in
    O(n-m)$. Then with $m'=\diag(1,k_1,\1_{n-m+1})$ we have by
    \eqref{eq:FPGroupAction2}
    \begin{align*}
      \rho_{\sigma,\varepsilon}^G(m)\Psi(f\otimes\phi)(x,y) &=
      \Psi(f\otimes\phi)(k_1^{-1}x,k_2^{-1}y)
      \\
      &= \abs{ k_1^{-1}x }^\alpha
      F(\tfrac{\abs{ k_2^{-1}y }^2}{\abs{ k_1^{-1}x }^2})f(k_1^{-1}x)\phi(k_2^{-1}y)
      \\
      &= \abs{ x }^\alpha F(\tfrac{\abs{ y }^2}{\abs{ x }^2})f(k_1^{-1}x)\phi(k_2^{-1}y)
      \\
      &=
      \Psi(\rho_{\tau,\varepsilon+k}^{O(1,m+1)}(m')f\otimes(k_2\cdot\phi))(x,y).
      \intertext{Further, for $m_0$ we have with
        \eqref{eq:FPGroupAction3}}
      \rho_{\sigma,\varepsilon}^G(m_0)\Psi(f\otimes\phi)(x,y) &=
      (-1)^\varepsilon\Psi(f\otimes\phi)(-x,-y)
      \\
      &= (-1)^\varepsilon\abs{ (-x) }^\alpha
      F(\tfrac{\abs{ (-y) }^2}{\abs{ (-x) }^2})f(-x)\phi(-y)
      \\
      &= (-1)^{\varepsilon+k}\abs{ x }^\alpha
      F(\tfrac{\abs{ y }^2}{\abs{ x }^2})f(-x)\phi(y)
      \\
      &=
      \Psi(\rho_{\tau,\varepsilon+k}^{O(1,m+1)}(m_0)f\otimes\phi)(x,y).
    \end{align*}
  \item For $a=e^{tH}\in A$ we obtain with \eqref{eq:FPGroupAction4}
    \begin{align*}
      \rho_{\sigma,\varepsilon}^G(a)\Psi(f\otimes\phi)(x,y) &=
      e^{(\sigma-n)t}\Psi(f\otimes\phi)(e^{-2t}x,e^{-2t}y)
      \\
      &= e^{(\sigma-n)t}\abs{ e^{-2t}x }^\alpha
      F(\tfrac{\abs{ e^{-2t}y }^2}{\abs{ e^{-2t}x }^2})f(e^{-2t}x)\phi(e^{-2t}y)
      \\
      &= e^{(\sigma-n-2\alpha-2k)t}\abs{ x }^\alpha
      F(\tfrac{\abs{ y }^2}{\abs{ x }^2})f(e^{-2t}x)\phi(y)
      \\
      &= e^{(\tau-m)t}\abs{ x }^\alpha
      F(\tfrac{\abs{ y }^2}{\abs{ x }^2})f(e^{-2t}x)\phi(y)
      \\
      &=
      \Psi(\rho_{\tau,\varepsilon+k}^{O(1,m+1)}(a)f\otimes\phi)(x,y).
    \end{align*}
  \item To show the intertwining property for $\frakn_H$ it suffices
    by \eqref{eq:FPAlgebraAction4} to show the identity
    \begin{equation*}
      \calB_j^{n,\sigma}\Psi(f\otimes\phi) =
      \Psi(\calB_j^{m,\tau}f\otimes\phi)
    \end{equation*}
    for $j=1,\ldots,m$ which follows from the next lemma.\qedhere
  \end{enumerate}
\end{proof}

For $\sigma,\mu\in\CC$ we introduce the ordinary differential operator
\begin{equation}
  \calD_{\sigma,\mu} := t(1+t)\frac{\di^2}{\di
    t^2}+\left(\frac{\mu-\sigma+2}{2}t+\frac{\mu}{2}\right)\frac{\di}{\di
    t}.\label{eq:DefDsigmamu}
\end{equation}

\begin{lem}\label{lem:BesselOpsAndDSigmaMu}
  Let $\sigma,\tau,\alpha\in\CC$, $k\in\NN$, $\mu=2k+n-m$, $F\in
  C^\infty([0,\infty))$, $f\in C^\infty(\RR^m\minuszero )$ and
  $\phi\in\calH^k(\RR^{n-m})$. Then for every $j=1,\ldots,m$ we have
  \begin{align*}
    \MoveEqLeft \calB_j^{n,\sigma}\left[\abs{ x }^\alpha
      F(\tfrac{\abs{ y }^2}{\abs{ x }^2})f(x)\phi(y)\right]
    \\
    &= \abs{ x }^\alpha F(\tfrac{\abs{ y }^2}{\abs{ x }^2})\calB_j^{m,\tau}f(x)\phi(y) +
    x_j\abs{ x }^{\alpha-2}f(x)\phi(y)
    \left(4\calD_{\sigma,\mu}+\alpha(\sigma-\mu-\alpha)\right)
    F(\tfrac{\abs{ y }^2}{\abs{ x }^2}).
  \end{align*}
\end{lem}

\begin{proof}
  We first note the following basic identities, where
  $\frac{\partial}{\partial x}$ and $\frac{\partial}{\partial y}$ are
  the gradients in $x\in\RR^m$ and $y\in\RR^{n-m}$ respectively, and
  $\Delta_x$ and $\Delta_y$ the Laplacians on $\RR^m$ and $\RR^{n-m}$
  respectively:
  \begin{align*}
    \frac{\partial}{\partial x}\abs{ x }^\alpha &= \alpha\abs{ x
    }^{\alpha-2}x, & \Delta_x\abs{ x }^\alpha &=
    \alpha(\alpha+m-2)\abs{ x }^{\alpha-2},
    \\
    \frac{\partial}{\partial x}F(\tfrac{\abs{ y }^2}{\abs{ x }^2}) &=
    -\frac{2\abs{ y }^2}{\abs{ x }^4}F'(\tfrac{\abs{ y }^2}{\abs{ x
      }^2})x, & \Delta_xF(\tfrac{\abs{ y }^2}{\abs{ x }^2}) &=
    4\frac{\abs{ y }^4}{\abs{ x }^6}F''(\tfrac{\abs{ y }^2}{\abs{ x
      }^2})-2(m-4)\frac{\abs{ y }^2}{\abs{ x }^4}F'(\tfrac{\abs{ y
      }^2}{\abs{ x }^2}),
    \\
    \frac{\partial}{\partial y}F(\tfrac{\abs{ y }^2}{\abs{ x }^2}) &=
    \frac{2}{\abs{ x }^2}F'(\tfrac{\abs{ y }^2}{\abs{ x }^2})y, &
    \Delta_yF(\tfrac{\abs{ y }^2}{\abs{ x }^2}) &=
    \frac{4\abs{ y }^2}{\abs{ x }^4}F''(\tfrac{\abs{ y }^2}{\abs{ x
      }^2})
    +\frac{2(n-m)}{\abs{ x }^2}F'(\tfrac{\abs{ y }^2}{\abs{ x }^2}).
  \end{align*}
  The calculation is split into several parts. In what follows we
  abbreviate $t:=\frac{\abs{ y }^2}{\abs{ x }^2}$.
  \begin{enumerate}
  \item\label{item:prf:lem:2.2:i} We begin with calculating
    $x_j\Delta_x\Psi(f\otimes\phi)$:
    \begin{align*}
      \MoveEqLeft[3] x_j\Delta_x\Psi(f\otimes\phi)(x,y)
      \\
      ={}& \Psi(x_j\Delta_xf\otimes\phi)(x,y) +
      x_j\Delta_x\abs{ x }^\alpha\cdot F(\tfrac{\abs{ y }^2}{\abs{ x }^2})f(x)\phi(y)
      \\
      & + x_j\Delta_xF(\tfrac{\abs{ y }^2}{\abs{ x }^2})\cdot\abs{ x }^\alpha
      f(x)\phi(y) + 2x_j\frac{\partial\abs{ x }^\alpha}{\partial
        x}\cdot\frac{\partial f}{\partial x}(x)\cdot
      F(\tfrac{\abs{ y }^2}{\abs{ x }^2})\phi(y)
      \\
      & + 2x_j\frac{\partial\abs{ x }^\alpha}{\partial x}\cdot\frac{\partial
        F(\tfrac{\abs{ y }^2}{\abs{ x }^2})}{\partial x}\cdot f(x)\phi(y) +
      2x_j\frac{\partial F(\tfrac{\abs{ y }^2}{\abs{ x }^2})}{\partial
        x}\cdot\frac{\partial f}{\partial x}(x)\cdot\abs{ x }^\alpha\phi(y)
      \\
      ={}& \Psi(x_j\Delta_xf\otimes\phi)(x,y) +
      x_j\abs{ x }^{\alpha-2}Ef(x)\phi(y)\left(-4tF'(t)+2\alpha F(t)\right)
      \\
      & +
      x_j\abs{ x }^{\alpha-2}f(x)\phi(y)
      \left(4t^2F''(t)-2(2\alpha+m-4)tF'(t)+\alpha(\alpha+m-2)F(t)\right).
    \end{align*}
  \item\label{item:prf:lem:2.2:ii} Next we calculate $x_j\Delta_y\Psi(f\otimes\phi)$:
    \begin{align*}
      \MoveEqLeft[3]  x_j\Delta_y\Psi(f\otimes\phi)(x,y)
      \\
      ={}& x_j\Delta_yF(\tfrac{\abs{ y }^2}{\abs{ x }^2})\cdot \abs{ x }^\alpha
      f(x)\phi(y)+x_j\Delta_y\phi(y)\cdot\abs{ x }^\alpha
      F(\tfrac{\abs{ y }^2}{\abs{ x }^2})f(x)
      \\
      & + 2x_j\frac{\partial F(\tfrac{\abs{ y }^2}{\abs{ x }^2})}{\partial
        y}\cdot\frac{\partial\phi}{\partial y}\cdot\abs{ x }^\alpha f(x)
      \\
      ={}&
      x_j\abs{ x }^{\alpha-2}f(x)\phi(y)\left(4tF''(t)+2(2k+n-m)F'(t)\right)
    \end{align*}
    since $E\phi=k\phi$ and $\Delta_y\phi=0$.
  \item\label{item:prf:lem:2.2:iii} We now calculate $\frac{\partial}{\partial
      x_j}\Psi(f\otimes\phi)$:
    \begin{align*}
      \MoveEqLeft[3]  \frac{\partial}{\partial x_j}\Psi(f\otimes\phi)(x,y)
      \\
      ={}& \frac{\partial\abs{ x }^\alpha}{\partial x_j}\cdot
      F(\tfrac{\abs{ y }^2}{\abs{ x }^2})f(x)\phi(y) + \frac{\partial
        F(\frac{\abs{ y }^2}{\abs{ x }^2})}{\partial x_j}\cdot\abs{ x }^\alpha
      f(x)\phi(y)
      \\
      & + \frac{\partial f}{\partial x_j}(x)\cdot\abs{ x }^\alpha
      F(\tfrac{\abs{ y }^2}{\abs{ x }^2})\phi(y)
      \\
      ={}& \frac{\partial f}{\partial x_j}(x)\cdot\abs{ x }^\alpha
      F(\tfrac{\abs{ y }^2}{\abs{ x }^2})\phi(y) +
      x_j\abs{ x }^{\alpha-2}f(x)\phi(y)\left(-2tF'(t)+\alpha F(t)\right).
    \end{align*}
  \item\label{item:prf:lem:2.2:iv} Next we find $(2E-\sigma+n)\frac{\partial}{\partial
      x_j}\Psi(f\otimes\phi)$ by using (iii):
    \begin{align*}
      \MoveEqLeft[3] (2E-\sigma+n)\frac{\partial}{\partial
        x_j}\Psi(f\otimes\phi)(x,y)
      \\
      ={}& (2E-\sigma+n+2(\alpha+k))\frac{\partial f}{\partial
        x_j}(x)\cdot\abs{ x }^\alpha F(\tfrac{\abs{ y }^2}{\abs{ x }^2})\phi(y)
      \\
      & + 2x_j\abs{ x }^{\alpha-2}Ef(x)\phi(y)\left(-2tF'(t)+\alpha
        F(t)\right)
      \\
      & +
      (2(\alpha+k-1)-\sigma+n)x_j\abs{ x }^{\alpha-2}f(x)\phi(y)\left(-2tF'(t)+\alpha
        F(t)\right)
    \end{align*}
    since $E\abs{ x }^\beta=\beta\abs{ x }^\beta$, $EF(\frac{\abs{ y }^2}{\abs{ x }^2})=0$ and
    $E\phi=k\phi$.
  \end{enumerate}
  Now, putting \ref{item:prf:lem:2.2:i}, \ref{item:prf:lem:2.2:ii} and
  \ref{item:prf:lem:2.2:iv} together gives the claimed identity.
\end{proof}

\section{Spectral decomposition of an ordinary second order differential operator}\label{sec:SpectralDecomp}

Proposition~\ref{prop:IntertwiningProperty} and
Lemma~\ref{lem:BesselOpsAndDSigmaMu} suggest that the decomposition
of the $O(n-m)$-isotypic component of
$\calH^k(\RR^{n-m})$ in $\rho_{\sigma,\varepsilon}^G$ into irreducible
$O(1,m+1)$-representations is given by the spectral decomposition of
the second order differential operator $\calD_{\sigma,\mu}$ defined in
\eqref{eq:DefDsigmamu} where $\mu=2k+n-m$. In this section we find the
spectral decomposition of $\calD_{\sigma,\mu}$ acting on
$L^2(\RR_+,t^{\frac{\mu-2}{2}}(1+t)^{-\frac{\Real\sigma}{2}}\td t)$
using the theory developed by Weyl--Stone--Kodaira--Titchmarsh.

\subsection{Kodaira's result}\label{subsec:Kodaira}

The spectral decomposition formula for general self-adjoint
 ordinary differential operators of the second order
 was established by Kodaira~\cite{KodSugaku1,Kod49}
 and Titchmarsh~\cite{Tit46}.
In \cite{Kod49} and \cite{KodSugaku2}, 
 Kodaira studied Schr\"{o}dinger type operators in detail
 and deduced a simpler formula for the spectral measure
 of these operators, which also laid a mathematical foundation
 for Heisenberg's $S$-matrix theory.
We can apply this simpler formula to our setting, because
 $\calD_{\sigma,\mu}$ turns out to be a Schr\"{o}dinger type operator
 after a suitable change of variables.

We first recall Kodaira's spectral decomposition theorem
 for Schr\"{o}dinger type operators (see the original papers
 \cite{Kod49} or \cite{KodSugaku2} for the proof).
Let 
\begin{align*}
L=-\frac{\di^2}{\di x^2}+\frac{\nu(\nu+1)}{x^2}+V(x) \quad 
(0<x<\infty),
\end{align*}
where $\nu\geq-\frac{1}{2}$ and $V(x)$ is a real-valued
 continuous function such that
\begin{align*}
V(x)=O(x^{-2+\epsilon})\ \text{as}\ x\to 0, \quad 
V(x)=\frac{\alpha+O(x^{-\epsilon})}{x}\ \text{as}\ x\to \infty
\end{align*}
for some $\alpha\in \RR$ and $\epsilon>0$.
A system of two solutions $s_1(x,\lambda)$, $s_2(x,\lambda)$ to
 $Lu=\lambda u\,(\lambda\in\CC)$
is called {\it a system of fundamental solutions} if it has the following
 three properties:
\begin{itemize}
\item $W(s_2,s_1)=1$, where
 $\displaystyle W(u,v)=u\frac{\di v}{\di x} - v\frac{\di u}{\di x}$
 denotes the Wronskian,
\item $s_j(x,\overline{\lambda})=\overline{s_j(x,\lambda)}$ for $j=1,2$, 
\item $s_j(x,\lambda)$ and
 $\displaystyle\frac{\di}{\di x}s_j(x,\lambda)$
 are holomorphic in $\lambda\in\CC$ for $j=1,2$.
\end{itemize}

For Schr\"{o}dinger type operators there exists a system of fundamental
 solutions $s_1$, $s_2$ with the following asymptotic behaviour as $x\to0$:
\begin{align*}
s_1(x,\lambda) &\sim x^{\nu+1}, & s_2(x,\lambda)
 &\sim \frac{1}{2\nu+1}x^{-\nu}
 && \text{ if } \nu> -\frac{1}{2},\\
s_1(x,\lambda) &\sim x^{\frac{1}{2}}, &
 s_2(x,\lambda) &\sim -x^{\frac{1}{2}}\log x
 && \text{ if } \nu= -\frac{1}{2}.
\end{align*}
We note that the function $s_1$ is uniquely determined
 because a solution to $Lu=\lambda u$ with
 $u(x) \sim x^{\nu+1}$ is unique.
Since $s_1$ is $L^2$ near $x=0$ for any $\nu\geq -\frac{1}{2}$ 
 and $s_2$ is $L^2$ near $x=0$ if and only if
 $\nu < \frac{1}{2}$, we conclude that
 $x=0$ is of limit point type (LPT) if $\nu\geq \frac{1}{2}$
 and of limit circle type (LCT) if $-\frac{1}{2}\leq\nu<\frac{1}{2}$. 
In the case of (LCT) at $x=0$ we impose the following additional boundary
 condition (which is automatic in the case of (LPT)):
\begin{equation}
  \lim_{x\to0}W(s_1(-,0),u)(x) =  0.\tag{BC}\label{eq:BoundaryCondition}
\end{equation}
Then in both the (LPT) and the (LCT) case $s_1(x,\lambda)$ is the unique solution to
$Lu=\lambda u$ which is $L^2$ near $x=0$ and satisfies
the boundary condition~\eqref{eq:BoundaryCondition}.

On the other hand, the point $x=\infty$ is always of (LPT)
 and we have:
\begin{thm}[{\cite[Theorem 5.1]{Kod49}}, {\cite[Theorem 26]{KodSugaku2}}]\label{thm:L2SolutionAtInfty}
If $\Imaginary \kappa\geq 0$ and $\kappa\neq 0$, the equation $Lu=\kappa^2u$
has one and only one solution $u_0(-,\kappa)$ such that
\begin{align*}
u_0(x,\kappa)\sim \exp\left(i\kappa x-\frac{i\alpha}{2\kappa}\log x\right)
\qquad \text{ as }x\to \infty.
\end{align*}
As functions of the two variables $x$ and $\kappa$, $u_0(x,\kappa)$ and 
$\frac{\di}{\di x}u_0(x,\kappa)$ are continuous in $0<x<\infty$,
 $\Imaginary \kappa \geq 0$ and $\kappa\neq 0$.
As functions of $\kappa$, they are holomorphic in $\Imaginary \kappa >0$.
\end{thm}
The differential operator $L$ defines a self-adjoint operator
 on $L^2(\RR_+)$ with domain
 the space of functions $u$ satisfying the following five conditions:
\begin{itemize}
\item $u\in L^2(\RR_+)$,
\item $u$ is differentiable,
\item $\displaystyle\frac{\di u}{\di x}$ is absolutely continuous
 in every closed interval $[a,b]\ (0<a<b<\infty)$,
\item $Lu\in L^2(\RR_+)$,
\item $u$ satisfies the boundary condition~\eqref{eq:BoundaryCondition}.
\end{itemize}

The spectral decomposition of $L$ is given in terms of
 the functions $A(\kappa)$ and $B(\kappa)$
 defined by
\begin{align*}
u_0(x,\kappa)=A(\kappa)s_2(x,\kappa^2)-B(\kappa)s_1(x,\kappa^2).
\end{align*}
This equation implies 
$$ A(\kappa)=W(u_0(-,\kappa), s_1(-,\kappa^2)) \qquad \text{ and } \qquad B(\kappa)=W(u_0(-,\kappa), s_2(-,\kappa^2)). $$
The functions $A(\kappa)$ and $B(\kappa)$
 are holomorphic in $\Imaginary\kappa >0$ and continuous
 in $\Imaginary \kappa \geq 0$ and $\kappa\neq 0$.
In $\Imaginary\kappa >0$, all zeros of $A(\kappa)$ lie on 
 the imaginary axis and are of order one.
Denote these zero points by $\kappa_j=i|\kappa_j|\ (j\in J)$.
Then it can be proved that
 the discrete spectrum of $L$ is $\lambda=\kappa_j^2$
 and possibly $\lambda=0$.
The continuous spectrum of $L$ is the interval $[0,\infty)$.
The eigenfunction expansion formula for $L$ is
\begin{thm}[\cite{Kod49,KodSugaku2}]
\label{thm:SpectralDecomposition}
In the setting and the notation above,
 we have an expansion of any $L^2$-function $u(x)$ of the following form
\begin{align}\label{eq:SpectralDecomposition}
u(x)&=\sum_{j\in J} s_1(x,\kappa_j^2)\rho_j
 \int_0^{\infty} s_1(y,\kappa_j^2)u(y) \td y
 + s_1(x,0)\rho^0 \int_0^{\infty} s_1(y,0)u(y) \td y \nonumber \\
&+\frac{2}{\pi} \int_0^{\infty} s_1(x,\kappa^2)
 \frac{\kappa^2}{|A(\kappa)|^2}
 \int_0^{\infty} s_1(y,\kappa^2)u(y) \td y\td \kappa,
\end{align}
where 
\begin{align*}
\rho_j=\frac{1}{\pi}|\kappa_j|B(\kappa_j)
 \oint_{\kappa_j}\frac{\td \kappa}{A(\kappa)}, \text{ and } 
\rho^0=\lim_{\epsilon\to +0}\frac{1}{\pi}
 \int_0^\pi \frac{B(\epsilon e^{i\theta})}{A(\epsilon e^{i\theta})}
 \epsilon^2e^{2i\theta} \td \theta.
\end{align*}
\end{thm}
We remark that $\rho^0=0$ in many cases.

To reformulate it as an isomorphism between Hilbert spaces put 
\begin{align*}
S:=\{\kappa_j^2:j\in J\}(\cup\{0\})\cup \RR_+,
\end{align*}
where $\{0\}$ is included if $\rho^0>0$.
Define a measure on $S$ by
\begin{align*}
  \int_{S}{g(\lambda)\td\rho(\lambda)} :=
\sum_{j\in J} \rho_j g(\kappa_j^2)
( + \rho^0 g(0) ) 
+\frac{1}{\pi} \int_0^{\infty} 
 \frac{\sqrt{\lambda}}{|A(\sqrt{\lambda})|^2}
  g(\lambda) \td \lambda.
\end{align*}
Then by \cite[Theorem 4.2]{Kod49} or \cite[Theorem 19]{KodSugaku2}:
\begin{thm}\label{thm:SpectralDecomposition2}
The map
  \begin{equation*}
    L^2(\RR_+) \xrightarrow{\ \sim\ }
    L^2(S,\td\rho), \qquad
    u \mapsto
    g(\lambda)=\int_0^\infty{s_1(x,\lambda)u(x)\td x},
  \end{equation*}
  is a unitary isomorphism with inverse
  \begin{equation*}
    L^2(S,\td\rho) \xrightarrow{\ \sim\ }
    L^2(\RR_+), \qquad g \mapsto
    u(x)=\int_{S}{s_1(x,\lambda)
      g(\lambda)\td\rho(\lambda)}.
  \end{equation*}
\end{thm}

\subsection{Simplifications}

In the rest of this section we apply the above result to find the spectral
 decomposition of $\calD_{\sigma,\mu}$.
We fix $\sigma\in i\RR\cup(0,\infty)$. (In the
case $\sigma\in(-\infty,0)$ only the derivation of the discrete spectrum 
 is slightly different. However,
since $\pi_{\sigma,\varepsilon}\cong\pi_{-\sigma,\varepsilon}$
for $\sigma \in (-n,n)$
the decomposition of the representations is again the same and it suffices
to consider $\sigma\in i\RR\cup(0,\infty)$ for our purpose.) 
Further fix $k\in\NN$ and
put $\mu:=2k+n-m$. We assume $m<n$ so that
$\mu>0$. Writing
\begin{equation*}
  \calD_{\sigma,\mu} = t(1+t)\frac{\di^2}{\di
    t^2}+\left((a+b+1)t+c\right)\frac{\di}{\di t}
\end{equation*}
with
\begin{equation*}
  a = -\frac{\sigma-\mu}{4}+\frac{\tau}{4}, 
  \qquad
  b =  -\frac{\sigma-\mu}{4}-\frac{\tau}{4}, 
  \qquad
  c = \frac{\mu}{2},
\end{equation*}
it is easy to see from \eqref{eq:DiffEqHypergeometric} that the
hypergeometric function
\begin{equation}
  F(t,\tau) := {_2F_1}\left(a,b;c;-t\right) \label{eq:DefF}
\end{equation}
solves the equation
\begin{equation*}
  \calD_{\sigma,\mu}f+\lambda^*f = 0, 
  \qquad
  \lambda^* = ab =
  \left(\frac{\sigma-\mu}{4}\right)^2-\left(\frac{\tau}{4}\right)^2.
\end{equation*}
We find a spectral decomposition of $\calD_{\sigma,\mu}$ in terms
of~$F(t,\tau)$.

First make the
transformation $t=\sinh^2(\frac{x}{2})$. Using $t\frac{\di}{\di
  t}=\tanh(\frac{x}{2})\frac{\di}{\di x}$ we write the operator
$\calD_{\sigma,\mu}$ as
\begin{align*}
  \calD_{\sigma,\mu} &= \frac{1}{t}\left((1+t)\left(t\frac{\di}{\di
        t}\right)^2+\left(\frac{\mu-\sigma}{2}t+\frac{\mu-2}{2}\right)t\frac{\di}{\di
      t}\right)
  \\
  &= \frac{\di^2}{\di x^2}+\beta(x)\frac{\di}{\di x}
\end{align*}
with
\begin{equation*}
  \beta(x) =
  \frac{\mu-1}{2}\tanh\left(\frac{x}{2}\right)^{-1}-\frac{\sigma-1}{2}\tanh\left(\frac{x}{2}\right).
\end{equation*}
Putting
\begin{equation*}
  u(x)=r(x)^{-1}f\left(\sinh^2\left(\frac{x}{2}\right)\right) 
  \qtextq{with}  r(x) =
  \sinh\left(\frac{x}{2}\right)^{-\frac{\mu-1}{2}}\cosh\left(\frac{x}{2}\right)^{\frac{\sigma-1}{2}}
\end{equation*}
we finally see that the differential equation
$\calD_{\sigma,\mu}f+\lambda^*f=0$ is equivalent to
\begin{equation*}
  -\frac{\di^2 u}{\di x^2}+q^*(x)u = \lambda^* u
\end{equation*}
with
\begin{align*}
  q^*(x) &= \tfrac{1}{4}\beta(x)^2+\tfrac{1}{2}\beta'(x)
  \\
  &= \frac{(\mu-1)(\mu-3)}{16}\tanh\left(\frac{x}{2}\right)^{-2}
  -\frac{\mu(\sigma-2)+1}{8}+\frac{(\sigma+1)(\sigma-1)}{16}\tanh\left(\frac{x}{2}\right)^2.
\end{align*}
To stay in line with the setting in Section~\ref{subsec:Kodaira} we shift the
eigenvalues by putting
$q(x):=q^*(x)-\left(\frac{\sigma-\mu}{4}\right)^2$ and
$\lambda:=\lambda^*-\left(\frac{\sigma-\mu}{4}\right)^2$ and obtain
\begin{equation}
  -\frac{\di^2u}{\di x^2}+q(x)u= \lambda u.
\label{eq:NormalizedDiffEq}
\end{equation}
Note that $q(x)$ is real-valued for $\sigma\in i\RR\cup\RR$ and hence
the operator $-\frac{\di^2}{\di x^2}+q(x)$ is formally self-adjoint on
$L^2(\RR_+)$.
Moreover by putting $\nu=\frac{\mu-3}{2}$ and $\alpha=0$,
 the differential operator
 $L=-\frac{\di^2}{\di x^2}+q(x)$ is of Schr\"{o}dinger type
 if $\mu\geq 2$.
This is also true for $\mu=1$ if we put $\nu=0$.
However, we should rather put $\nu=-1$ 
 in order to impose an appropriate boundary condition. 
Thus we cannot use the general result in Section~\ref{subsec:Kodaira}
 directly for $\mu=1$, but one can see that the proof of 
 Theorem~\ref{thm:SpectralDecomposition}
 in \cite{Kod49} or \cite{KodSugaku2}
 is still valid in this case
 and thus \eqref{eq:SpectralDecomposition} gives
 the correct formula.

\subsection{Singularities and the boundary condition}

We put $\nu=\frac{\mu-3}{2}$ for $\mu\geq 1$
 and $\kappa=\sqrt{\lambda}$.
The differential equation \eqref{eq:NormalizedDiffEq} has regular
singular points at $x=0$ and $x=\infty$. The corresponding asymptotic
behaviour of solutions at $x=0$ is given by $x^{\frac{\mu-1}{2}}$ and
$x^{-\frac{\mu-3}{2}}$ for $\mu\neq 2$ and by $x^{\frac{1}{2}}$ and
$\log(x)x^{\frac{1}{2}}$ for $\mu=2$. Hence $x=0$ is of limit point
type (LPT) if $\mu\geq 4$ and of limit circle type (LCT) if
$\mu=1,2,3$. The solution
\begin{equation*}
  s_1(x,\lambda) = 2^{\frac{\mu-1}{2}}
 r(x)^{-1}{_2F_1}(a,b;c;-\sinh^2(\tfrac{x}{2}))
\end{equation*}
 has asymptotic behaviour $x^{\frac{\mu-1}{2}}(=x^{\nu+1})$
 near $x=0$, where
\begin{equation*}
  a = -\frac{\sigma-\mu}{4}+i\kappa, 
  \qquad
  b =
  -\frac{\sigma-\mu}{4}-i\kappa, 
  \qquad c = \frac{\mu}{2}.
\end{equation*}
Note that $s_1(x,\lambda)$ is holomorphic in $\lambda\in\CC$ and
 $\overline{s_1(x,\lambda)}=s_1(x,\overline{\lambda})$
 if $\sigma\in\RR\cup i\RR$
 by Kummer's transformation formula \eqref{eq:KummerFormula}.
For $\mu\geq 2$, $s_1$ is the unique solution which has asymptotic behaviour
 $x^{\frac{\nu+1}{2}}$ near $x=0$. 
Hence we can find $s_2$ such that
 $s_1$, $s_2$ is a system of fundamental solutions.
For $\mu=1$, put
\begin{equation*}
  s_2(x,\lambda) = -2 r(x)^{-1} \sinh(\tfrac{x}{2})
 {_2F_1}(1+a-c,1+b-c;2-c;-\sinh^2(\tfrac{x}{2})).
\end{equation*}
Then $s_2$ has asymptotic behaviour $-x$ near $x=0$
 and  $s_1$, $s_2$ is a system of fundamental solutions.

In the case of (LCT)
at $x=0$ we impose the additional boundary condition~\eqref{eq:BoundaryCondition}.
Then in both the (LPT) and the (LCT) case (i.e.\ for every $\mu\geq 1$) 
 $s_1(x,\lambda)$ is the unique solution to
\eqref{eq:NormalizedDiffEq} which is $L^2$ near $x=0$ and satisfies
the boundary condition~\eqref{eq:BoundaryCondition}.

In view of Theorem~\ref{thm:L2SolutionAtInfty} we consider another solution
\begin{align*}
  u_0(x,\kappa) &= 2^{2i\kappa}
  r(x)^{-1}\sinh^{-2b}(\tfrac{x}{2}){_2F_1}(b,b-c+1;b-a+1;-\sinh^{-2}(\tfrac{x}{2})),
\end{align*}
which has asymptotic behaviour $e^{ix\kappa}$ as
$x\to \infty$ and hence is $L^2$ near $x=\infty$ for
$\Imaginary \kappa >0$. Note that a linearly independent solution is obtained by
interchanging $a$ and $b$ and has asymptotics $e^{-ix\kappa}$
whence $x=\infty$ is always of (LPT).

Altogether the operator in
\eqref{eq:NormalizedDiffEq} extends to a self-adjoint operator on
$L^2(\RR_+)$ under the boundary
condition~\eqref{eq:BoundaryCondition} and its spectral decomposition
is given by Theorem~\ref{thm:SpectralDecomposition}. We now make this
spectral decomposition explicit.

\subsection{The function $A(\kappa)$}

We calculate the Wronskian
\begin{align*}
A(\kappa)={}&W(u_0(-,\kappa),s_1(-,\kappa^2))
  \\
  ={}&
  2^{\frac{\mu-1}{2}+2i\kappa}
  r(x)^{-2}W( \sinh^{-2b}(\tfrac{-}{2})
  {_2F_1}(b,b-c+1;b-a+1;-\sinh^{-2}(\tfrac{-}{2})),
  \\
  & \qquad\qquad\qquad
  {_2F_1}(a,b;c;-\sinh^2(\tfrac{-}{2})))(x)
  \\
  ={}&
  2^{\frac{\mu-1}{2}+2i\kappa}
  r(x)^{-2}\sinh(\tfrac{x}{2})\cosh(\tfrac{x}{2})
  \\
  & \times
  W(z^{-b}{_2F_1}(b,b-c+1;b-a+1;-\tfrac{1}{z}),
  {_2F_1}(a,b;c;-z))(\sinh^2(\tfrac{x}{2}))
  \\
  ={}&
  2^{\frac{\mu-1}{2}+2i\kappa}
  r(x)^{-2}\sinh(\tfrac{x}{2})\cosh(\tfrac{x}{2})\frac{\Gamma(b-a)
    \Gamma(c)}{\Gamma(b)\Gamma(c-a)}
  \\
  & \times
  W(z^{-b}{_2F_1}(b,b-c+1;b-a+1;-\tfrac{1}{z}),\\
  &\qquad\qquad
  z^{-a}{_2F_1}(a,a-c+1;a-b+1;-\tfrac{1}{z}))(\sinh^2(\tfrac{x}{2}))
  \\
  ={}& 2^{\frac{\mu-1}{2}+2i\kappa}
  (b-a)\frac{\Gamma(b-a)\Gamma(c)}{\Gamma(b)\Gamma(c-a)}
  \\
  ={}& 2^{\frac{\mu-1}{2}+2i\kappa}
  \frac{\Gamma(-2i\kappa+1)
    \Gamma(\frac{\mu}{2})}{\Gamma(-\frac{\sigma-\mu}{4}
    -i\kappa)\Gamma(\frac{\sigma+\mu}{4}-i\kappa)}.
\end{align*}
Then we see that all the zeros of $A(\kappa)$
 in the upper half-plane $\Imaginary\kappa >0$ are of order one
 and exactly at the points
 where $-\frac{\sigma-\mu}{4}-i\kappa\in-\NN$. This gives
 $i\kappa=-\frac{\sigma-\mu}{4}+j$ and
 $\lambda=-\left(\frac{\sigma-\mu}{4}-j\right)^2$ for $j\in\NN$
 with $j<\frac{\sigma-\mu}{4}$.
Hence $A(\kappa)$ has zeros in $\Imaginary\kappa >0$
 if and only if $\sigma\in\RR$ and $\sigma>\mu$.
If this is the case, we put $\kappa_j=i(\frac{\sigma-\mu}{4}-j)$
 for $j\in [0,\frac{\sigma-\mu}{4})\cap \ZZ$.
Using $\Residue_{z=-n}\Gamma(z)=\frac{(-1)^n}{n!}$ we find
  \begin{align}   \label{eq:AResidue}
    \Residue_{\kappa=\kappa_j}
    \frac{1}{A(\kappa)}    ={}&
    2^{-\frac{\mu-1}{2}-2i\kappa_j}
    \frac{\Gamma(\frac{\sigma}{2}-j)}{
      \Gamma(\frac{\sigma-\mu}{2}-2j+1)\Gamma(\frac{\mu}{2})}
    \Residue_{\kappa=\kappa_j}
    {\Gamma\left(-\frac{\sigma-\mu}{4}-i\kappa\right)}
    \nonumber\\
    ={}&
    2^{-\frac{\mu-1}{2}-2i\kappa_j}
    \frac{\Gamma(\frac{\sigma}{2}-j)}{
     \Gamma(\frac{\sigma-\mu}{2}-2j+1)\Gamma(\frac{\mu}{2})}
    \frac{(-1)^j}{(-i)j!}.
  \end{align}
To calculate $B(\kappa)$ for $\kappa=\kappa_j$,
 we note that $b=-j$ and
 therefore, by \eqref{eq:2F1Identity1}
 \begin{align*}
   2^{-\frac{\mu-1}{2}} s_1(x,\kappa_j^2) = 
    \frac{(a)_j\Gamma(c)}{\Gamma(c+j)} 2^{-2i\kappa_j}u_0(x,\kappa_j) 
    =2^{-2i\kappa_j}
    \frac{(-\frac{\sigma-\mu}{2}+j)_j
      \Gamma(\frac{\mu}{2})}{\Gamma(\frac{\mu}{2}+j)}u_0(x,\kappa_j).
 \end{align*}
Here, $(a)_j= a(a+1)\cdots (a+j-1)$ denotes the Pochhammer symbol.
As a result,
\begin{align}   \label{eq:BKappaj}
 B(\kappa_j)= -2^{-\frac{\mu-1}{2}+2i\kappa_j}
 \frac{\Gamma(\frac{\mu}{2}+j)}
 {(-\frac{\sigma-\mu}{2}+j)_j\Gamma(\frac{\mu}{2})}.
\end{align}

\subsection{The spectral theorem for $\calD_{\sigma,\mu}$}

For $\mu\geq 2$ Theorem~\ref{thm:SpectralDecomposition}
 gives the spectral formula \eqref{eq:SpectralDecomposition}.
Following the proof in \cite{Kod49} or \cite{KodSugaku2}
 it is easy to see that \eqref{eq:SpectralDecomposition}
 is still valid for $\mu=1$.
For $\kappa>0$ we calculate
\begin{align*}
\frac{\kappa^2}{|A(\kappa)|^2}
&=2^{-(\mu-1)}\kappa^2
  \left|\frac{\Gamma(-\frac{\sigma-\mu}{4}
    -i\kappa)\Gamma(\frac{\sigma+\mu}{4}-i\kappa)}
  {\Gamma(-2i\kappa+1)\Gamma(\frac{\mu}{2})}\right|^2\\
&=2^{-(\mu+1)}
  \left|\frac{\Gamma(-\frac{\sigma-\mu}{4}+i\kappa)
  \Gamma(\frac{\sigma+\mu}{4}+i\kappa)}
  {\Gamma(2i\kappa)\Gamma(\frac{\mu}{2})}\right|^2.
\end{align*}
For $j\in [0,\frac{\Real\sigma-\mu}{4})\cap \ZZ$ we have 
\begin{align*}
\rho_j&=\frac{1}{\pi}|\kappa_j|B(\kappa_j)
 \oint_{\kappa_j}\frac{\td \kappa}{A(\kappa)} \\
&= \frac{1}{\pi}\left(\frac{\sigma-\mu}{4}-j\right)
 \times\frac{-2^{-\frac{\mu-1}{2}+2i\kappa_j}\Gamma(\frac{\mu}{2}+j)}
 {(-\frac{\sigma-\mu}{2}+j)_j\Gamma(\frac{\mu}{2})}
 \times 2\pi i 
    \frac{2^{-\frac{\mu-1}{2}-2i\kappa_j}i(-1)^j
      \Gamma(\frac{\sigma}{2}-j)}{
      j!\Gamma(\frac{\sigma-\mu}{2}-2j+1)\Gamma(\frac{\mu}{2})} \\
&= \frac{2^{-(\mu-1)}(\frac{\sigma-\mu}{2}-2j)\Gamma(\frac{\sigma}{2}-j)
      \Gamma(\frac{\mu}{2}+j)}{
     j!\Gamma(\frac{\mu}{2})^2
      \Gamma(\frac{\sigma-\mu}{2}-j+1)}
\end{align*}
by \eqref{eq:AResidue} and \eqref{eq:BKappaj}.
Moreover, since $\frac{B(\kappa)}{A(\kappa)}$ has at most a pole of
 order one at $x=0$,  
\begin{align*}
\rho^0=\lim_{\epsilon\to +0}\frac{1}{\pi}
 \int_0^\pi \frac{B(\epsilon e^{i\theta})}{A(\epsilon e^{i\theta})}
 \epsilon^2e^{2i\theta} \td \theta=0.
\end{align*}
Consequently, \eqref{eq:SpectralDecomposition} gives the expansion formula:
\begin{align*}
u(x)={}&\sum_{j\in [0,\frac{\Real \sigma-\mu}{4})\cap \ZZ} s_1(x,\kappa_j^2)
\frac{2^{-(\mu-1)}(\frac{\sigma-\mu}{2}-2j)\Gamma(\frac{\sigma}{2}-j)
      \Gamma(\frac{\mu}{2}+j)}{
     j!\Gamma(\frac{\mu}{2})^2
      \Gamma(\frac{\sigma-\mu}{2}-j+1)}
  \int_0^{\infty} s_1(y,\kappa_j^2)u(y) \td y \\
&+\frac{1}{\pi} \int_0^{\infty} s_1(x,\kappa^2)
  2^{-\mu}
  \left|\frac{\Gamma(-\frac{\sigma-\mu}{4}+i\kappa)
  \Gamma(\frac{\sigma+\mu}{4}+i\kappa)}
  {\Gamma(2i\kappa)\Gamma(\frac{\mu}{2})}\right|^2
 \int_0^{\infty} s_1(y,\kappa^2)u(y) \td y\td \kappa.
\end{align*}
Using the different normalization
\begin{align*}
\eta_1(x,\lambda):=r(x)^{-1}{_2F_1}(a,b;c;-\sinh^2(\tfrac{x}{2}))
 (=2^{-\frac{\mu-1}{2}}s_1(x,\lambda)),
\end{align*}
this can be rewritten as 
\begin{align*}
u(x)={}&\sum_{j\in [0,\frac{\Real \sigma-\mu}{4})\cap \ZZ}
 \eta_1(x, \textstyle-\left(\frac{\sigma-\mu}{4}-j\right)^2)
\displaystyle
\frac{(\frac{\sigma-\mu}{2}-2j)\Gamma(\frac{\sigma}{2}-j)
      \Gamma(\frac{\mu}{2}+j)}{
     j!\Gamma(\frac{\mu}{2})^2
      \Gamma(\frac{\sigma-\mu}{2}-j+1)} \\
&\qquad \qquad \qquad \qquad \qquad \qquad \qquad \qquad \times
 \int_0^{\infty} \eta_1(y,\textstyle-\left(\frac{\sigma-\mu}{4}-j\right)^2)u(y)
 \td y \\
&+\frac{1}{4\pi} \int_0^{\infty} \eta_1(x,\lambda)
  \left|\frac{\Gamma(-\frac{\sigma-\mu}{4}+i\sqrt{\lambda})
  \Gamma(\frac{\sigma+\mu}{4}+i\sqrt{\lambda})}
  {\Gamma(2i\sqrt{\lambda})\Gamma(\frac{\mu}{2})}\right|^2
 \int_0^{\infty} \eta_1(y,\lambda)u(y) \td y
 \frac{\td \lambda}{\sqrt{\lambda}}.
\end{align*}

To obtain an isomorphism between Hilbert spaces let
\begin{align*}
  S(\sigma,\mu) &:=
  (0,\infty)\cup\bigcup_{j\in[0,\frac{\Real\sigma-\mu}{4})\cap\ZZ}
  {\left\{-\left(\frac{\sigma-\mu}{4}-j\right)^2\right\}}.
\end{align*}
Note that $S(\sigma,\mu)=(0,\infty)$ for $\sigma\in i\RR$. On
$S(\sigma,\mu)$ we define a measure $\di\nu_{\sigma,\mu}$ by
\begin{multline*}
  \int_{S(\sigma,\mu)}{g(\lambda)\td\nu_{\sigma,\mu}(\lambda)} :=
  \frac{1}{4\pi}\int_0^\infty{g(\lambda) \abs*{
      \frac{\Gamma(-\frac{\sigma-\mu}{4}
        +i\sqrt{\lambda})\Gamma(\frac{\sigma+\mu}{4}
        +i\sqrt{\lambda})}{\Gamma(2i\sqrt{\lambda})
        \Gamma(\frac{\mu}{2})} }^2\frac{\td\lambda}{\sqrt{\lambda}}}
  \\
  +\sum_{j\in[0,\frac{\Real\sigma-\mu}{4})\cap\ZZ}
  {\frac{(\frac{\sigma-\mu}{2}-2j)\Gamma(\frac{\sigma}{2}-j)
      \Gamma(\frac{\mu}{2}+j)}{j!\Gamma(\frac{\mu}{2})^2
      \Gamma(\frac{\sigma-\mu}{2}-j+1)}
    g({\textstyle-\left(\frac{\sigma-\mu}{4}-j\right)^2})}.
\end{multline*}
Then by Theorem~\ref{thm:SpectralDecomposition2}:

\begin{thm}\label{thm:DiffOpSpectralDecomposition}
  For $\sigma\in i\RR\cup(0,\infty)$ and $\mu\in\ZZ_+$ the map
  \begin{equation*}
    L^2(\RR_+) \xrightarrow{\ \sim\ }
    L^2(S(\sigma,\mu),\td\nu_{\sigma,\mu}), \qquad
    u \mapsto
    g(\lambda)=\int_0^\infty{\eta_1(x,\lambda)u(x)\td x},
  \end{equation*}
  is a unitary isomorphism with inverse
  \begin{equation*}
    L^2(S(\sigma,\mu),\td\nu_{\sigma,\mu}) \xrightarrow{\ \sim\ }
    L^2(\RR_+), \qquad g \mapsto
    u(x)=\int_{S(\sigma,\mu)}{\eta_1(x,\lambda)
      g(\lambda)\td\nu_{\sigma,\mu}(\lambda)}.
  \end{equation*}
\end{thm}

For our application we need the spectral decomposition of the operator
$\calD_{\sigma,\mu}$ which follows from
Theorem~\ref{thm:DiffOpSpectralDecomposition} by the transformation
$f(t)\mapsto r(x)^{-1}f(\sinh^2(\frac{x}{2}))$. To state this put
\begin{equation*}
  T(\sigma,\mu) := i\RR_+ \cup
  \bigcup_{j\in[0,\frac{\Real\sigma-\mu}{4})\cap\ZZ}{\left\{\sigma-\mu-4j\right\}}
\end{equation*}
and define a measure $\di m_{\sigma,\mu}$ on $T(\sigma,\mu)$ by
\begin{multline}
  \int_{T(\sigma,\mu)}{g(\tau)\td m_{\sigma,\mu}(\tau)} :=
  \frac{1}{8\pi}\int_{i\RR_+}{g(\tau)\abs*{
      \frac{\Gamma(\frac{-\sigma+\mu+\tau}{4})
        \Gamma(\frac{\sigma+\mu+\tau}{4})}{\Gamma(\frac{\tau}{2})
        \Gamma(\frac{\mu}{2})} }^2\td\tau}
  \nonumber  \\
  +\sum_{j\in[0,\frac{\Real\sigma-\mu}{4})\cap\ZZ}{\frac{(\frac{\sigma-\mu}{2}-2j)\Gamma(\frac{\sigma}{2}-j)
      \Gamma(\frac{\mu}{2}+j)}{j!\Gamma(\frac{\mu}{2})^2
      \Gamma(\frac{\sigma-\mu}{2}-j+1)}g(\sigma-\mu-4j)}.
  \label{eq:DefMeasureMSigmaMu}
\end{multline}

\begin{cor}\label{cor:HypergeomTrafo}
  For $\sigma\in i\RR\cup(0,\infty)$ and $\mu\in\ZZ_+$ the map
  \begin{gather*}
    L^2(\RR_+,t^{\frac{\mu-2}{2}}(1+t)^{-\frac{\Real\sigma}{2}}\td t)
   \xrightarrow{\ \sim\ }
   L^2(T(\sigma,\mu),\td m_{\sigma,\mu}),
    \\
    f \mapsto
    g(\tau)=\int_0^\infty{F(t,\tau)f(t)
      t^{\frac{\mu-2}{2}}(1+t)^{-\frac{\sigma}{2}}\td  t}
  \end{gather*}
  is a unitary isomorphism with inverse
  \begin{gather*}
    L^2(T(\sigma,\mu),\td m_{\sigma,\mu}) \xrightarrow{\ \sim\ }
    L^2(\RR_+,t^{\frac{\mu-2}{2}}(1+t)^{-\frac{\Real\sigma}{2}}\td t),
    \\
    g \mapsto f(t)=\int_{T(\sigma,\mu)}{F(t,\tau)g(\tau)\td
      m_{\sigma,\mu}(\tau)}.
  \end{gather*}
\end{cor}

\begin{rem}
  For the discrete part, namely for $\tau=\sigma-\mu-4j$, $j\in
  \NN$, the Gau\ss\ hypergeometric function $F(t,\tau)$ degenerates
  to a polynomial in $t$ of degree $j$. More precisely, we have (see
  \eqref{eq:HypergeometricPolynomials})
  \begin{equation*}
    F(t,\sigma-\mu-4j) =
    \frac{j!}{(\frac{\mu}{2})_n}P_j^{(\frac{\mu-2}{2},-\frac{\sigma}{2})}(1+2t),
  \end{equation*}
  where $P_n^{(\alpha,\beta)}(z)$ denote the Jacobi polynomials.
\end{rem}

\begin{rem}
  For $\sigma\in(0,\infty)$ the results of
  Corollary~\ref{cor:HypergeomTrafo} can also be found in
  \cite[formula (A.11)]{Fle77} where the hypergeometric transform
  appears (essentially) as the radial part of the spherical Fourier
  transform on $\SU(1,n)/\SU(n)$. Since our approach provides a
  unified treatment of both complementary series,
  discrete series representations for the hyperboloid and principal series,
  including the case $\sigma\in i\RR$, we gave a detailed proof in
  this section for convenience.
\end{rem}

\section{Decomposition of representations and the Plancherel
  formula}\label{sec:BranchingLawPlancherelFormula}

Using the spectral decomposition of $\calD_{\sigma,\mu}$ obtained in
Corollary~\ref{cor:HypergeomTrafo} we find in this section the
explicit Plancherel formula for the decomposition of
$\rho_{\sigma,\varepsilon}^G\restrictedto H$.

Let us first consider the action of $O(n-m)$ on
$L^2(\RR^n,\abs{ (x,y) }^{-\Real\sigma}\td x\td y)$ which gives the
following decomposition as $O(n-m)$-representations:
\begin{align}
  \MoveEqLeft   L^2(\RR^n,\abs{ (x,y) }^{-\Real\sigma}\td x\td y)
  \nonumber
  \\
  &=
  \sideset{}{^\oplus}\sum_{k=0}^\infty
  L^2(\RR^m\times\RR_+,(\abs{ x }^2+r^2)^{-\frac{\Real\sigma}{2}}r^{2k+n-m-1}\td
    x\td r)\boxtimes\calH^k(\RR^{n-m}),
  \label{eq:On-mDecomposition}
\end{align}
where $r=\abs{ y }$. We fix a summand for some $k\in\NN$ and put again
$\mu=2k+n-m$. The coordinate change $t:=\frac{r^2}{\abs{ x }^2}$ gives
\begin{align*}
  \MoveEqLeft
  L^2(\RR^m\times\RR_+,(\abs{ x }^2+r^2)^{-\frac{\Real\sigma}{2}}r^{\mu-1}\td
  x\td r)
  \\
  &=
  L^2(\RR^m\times\RR_+,\tfrac{1}{2}\abs{ x }^{-\Real\sigma+\mu}
  t^{\frac{\mu-2}{2}}(1+t)^{-\frac{\Real\sigma}{2}}\td
  x\td t).
\end{align*}
Since
\begin{align*}
  \MoveEqLeft
  L^2(\RR^m\times\RR_+,\tfrac{1}{2}\abs{ x }^{-\Real\sigma+\mu}
  t^{\frac{\mu-2}{2}}(1+t)^{-\frac{\Real\sigma}{2}}\td
  x\td t)
  \\
  & \cong L^2(\RR^m,\tfrac{1}{2}\abs{ x }^{-\Real\sigma+\mu}\td
  x)\widehat{\otimes}L^2(\RR_+,t^{\frac{\mu-2}{2}}(1+t)^{-\frac{\Real\sigma}{2}}\td
  t)
\end{align*}
we can apply Theorem~\ref{thm:DiffOpSpectralDecomposition} to find
that the map
\begin{align*}
  \MoveEqLeft
  L^2(\RR^m\times\RR_+,\tfrac{1}{2}\abs{ x }^{-\Real\sigma+\mu}
  t^{\frac{\mu-2}{2}}(1+t)^{-\frac{\Real\sigma}{2}}\td
  x\td t)
  \\
  &\to
  \int^\oplus_{T(\sigma,\mu)}{L^2(\RR^m,\tfrac{1}{2}\abs{ x }^{-\Real\tau}\td
    x)\td m_{\sigma,\mu}(\tau)}
\end{align*}
given by
\begin{align*}
  f(x,t)\mapsto\hat{f}(x,\tau) &:=
  \abs{ x }^{-\frac{\sigma-\tau-\mu}{2}}
  \int_0^\infty{F(t,\tau)f(x,t)t^{\frac{\mu-2}{2}}
    (1+t)^{-\frac{\sigma}{2}}\td t}
  \\
  \intertext{is a unitary isomorphism, where $F(t,\tau)$ is defined by
    \eqref{eq:DefF} and the measure $\di m_{\sigma,\mu}$ is given by
    \eqref{eq:DefMeasureMSigmaMu}. Its inverse is given by} g(x,\tau)
  \mapsto \check{g}(x,t) &:=
  \int_{T(\sigma,\mu)}{\abs{ x }^{\frac{\sigma-\tau-\mu}{2}}F(t,\tau)g(x,\tau)\td
    m_{\sigma,\mu}(\tau)}.
\end{align*}

Now we put these things together. For $\sigma\in i\RR\cup(0,\infty)$ and
$k\in\NN$ we put $\mu:=2k+n-m$ and define an operator
\begin{align*}
  \MoveEqLeft   \Psi(\sigma,k):
  \left(\int_{T(\sigma,\mu)}^\oplus{L^2(\RR^m,\tfrac{1}{2}\abs{ x }^{-\Real\tau}\td
      x)\td m_{\sigma,\mu}(\tau)}\right)\boxtimes\calH^k(\RR^{n-m})
  \\
  &\to L^2(\RR^n,\abs{ (x,y) }^{-\Real\sigma}\td x\td y)
\end{align*}
by
\begin{align*}
  \MoveEqLeft \Psi(\sigma,k)\left(f\otimes\phi\right)(x,y)
  \\
  &:=
  \phi(y)
  \int_{T(\sigma,\mu)}{\abs{ x }^{\frac{\sigma-\tau-\mu}{2}}{_2F_1}
    \left(\tfrac{\mu-\sigma+\tau}{4},\tfrac{\mu-\sigma-\tau}{4};
      \tfrac{\mu}{2};-\tfrac{\abs{ y }^2}{\abs{ x }^2}\right)f(x,\tau)\td
    m_{\sigma,\mu}(\tau)}.
\end{align*}

\begin{thm}\label{thm:HIntertwiner}
  For $\sigma\in i\RR\cup(0,n)\cup(n+2\NN)$ and $\varepsilon\in\ZZ/2\ZZ$ the map
  $\Psi(\sigma,k)$ is $H$-equivariant between the representations
  \begin{equation*}
    \int_{T(\sigma,\mu)}^\oplus{\rho_{\tau,\varepsilon+k}^{O(1,m+1)}\td
      m_{\sigma,\mu}(\tau)}\boxtimes\calH^k(\RR^{n-m})\to
    \rho_{\sigma,\varepsilon}^G \restrictedto[\big] H
  \end{equation*}
  and constructs the $\calH^k(\RR^{n-m})$-isotypic component in
  $\rho_{\sigma,\varepsilon}^G \restrictedto H$. The following Plancherel formula
  holds:
  \begin{align*}
    \MoveEqLeft \norm{ \Psi(\sigma,k)(f\otimes\phi) }_{L^2(\RR^n,\abs{ (x,y) }^{-\Real\sigma}\td
      x\td y)}^2
    \\
    &=
    \int_{T(\sigma,\mu)}{\norm{ f(\emptyarg,\tau) }_{L^2(\RR^m,\frac{1}{2}\abs{ x }^{-\Real\tau}\td
        x)}^2\td
      m_{\sigma,\mu}(\tau)}\cdot\norm{ \phi }_{L^2(S^{n-m-1})}^2.
  \end{align*}
\end{thm}

\begin{proof}
  We have already seen that $\Psi(\sigma,k)$ gives a unitary
  isomorphism so that the Plancherel formula above holds. Further, by
  Proposition~\ref{prop:IntertwiningProperty} the map $\Psi(\sigma,k)$
  intertwines the actions of $M_HA\overline{N}_H$ on smooth vectors
  and hence on the Hilbert spaces. Since $H$ is generated by
  $M_HA\overline{N}_H$ and $N_H$ it remains to prove the intertwining
  property for $N_H$. For this we use the Lie algebra action.

\begin{lem}\label{lem:GroupIntertwining}
  Let $L$ be a connected Lie group with Lie algebra $\frakl$ and let
  $(\rho_1, \calH_1)$ and $(\rho_2, \calH_2)$ be unitary
  representations of $L$. Suppose that a continuous linear map
  $\varphi:\calH_1\to \calH_2$ is given and there exist subspaces
  $V_1\subset \calH_1$ and $V_2\subset \calH_2$ such that
  \begin{enumerate}
  \item\label{item:lem:GroupIntertwining:i} $V_i$ is dense in
    $\calH_i$ for $i=1,2$,
  \item\label{item:lem:GroupIntertwining:ii} $V_i$ is contained in the
    space of analytic vectors $\calH_i^\omega$ for $i=1,2$,
  \item\label{item:lem:GroupIntertwining:iii} $V_i$ is
    $\di\rho_i$-stable for $i=1,2$,
  \item\label{item:lem:GroupIntertwining:iv}
    $\Hermit{\varphi(\mathrm{d}\rho_1(X)v_1)}{v_2}_{\calH_2}
    =-\Hermit{\varphi(v_1)}{\mathrm{d}\rho_2(X)v_2}_{\calH_2}$
    for $v_1\in V_1$, $v_2\in V_2$ and $X\in\frakl$.
  \end{enumerate}
  Then $\varphi$ is $L$-equivariant.
\end{lem}

\begin{proof}
  For $v_1\in V_1$ and $v_2\in V_2$ we put
  \begin{alignat*}{2}
    f_{v_1,v_2}(g) &:= \Hermit{\varphi(\rho_1(g)v_1)}{v_2}_{\calH_2},
    &\qquad g&\in L,
    \\
    h_{v_1,v_2}(g) &:= \Hermit{\rho_2(g)\varphi(v_1)}{v_2}_{\calH_2} =
    \Hermit{\varphi(v_1)}{\rho_2(g^{-1})v_2}_{\calH_2}, & g&\in L,
  \end{alignat*}
  which are analytic functions on $L$ by
  \ref{item:lem:GroupIntertwining:ii}. For a smooth function $f$ on
  $L$ and $X\in\frakl$ we define derivatives by
  \begin{equation*}
    (R(X)f)(g) := \lim_{t\to 0} \frac{f(ge^{tX})-f(g)}{t}, \qquad
    (L(X)f)(g) := \lim_{t\to 0} \frac{f(e^{-tX}g)-f(g)}{t}.
  \end{equation*}
  We have $R(X)f(e)=-L(X)f(e)$ for the identity element $e\in L$ and
  $R(X)$ commutes with $L(X')$ for any $X, X'\in\frakl$. Hence
  \begin{align*}
    R(X_1)R(X_2)\cdots R(X_k)f(e) ={}& -L(X_1)R(X_2)\cdots R(X_k)f(e)
    \\
    ={}& -R(X_2)\cdots R(X_k)L(X_1)f(e)
    \\
    & \mathmakebox[5cm]\vdots
    \\
    ={}& (-1)^kL(X_k)\cdots L(X_2)L(X_1)f(e)
  \end{align*}
  for $X_1,\dots,X_k\in \frakl$. Then
  \ref{item:lem:GroupIntertwining:iv} implies
  \begin{align*}
    R(X_1)\cdots R(X_k)f_{v_1,v_2}(e) &=
    f_{\mathrm{d}\rho_1(X_1)\cdots\mathrm{d}\rho_1(X_k)v_1,v_2}(e)
    \\
    &= (-1)^k h_{v_1,
      \mathrm{d}\rho_2(X_k)\cdots\mathrm{d}\rho_2(X_1)v_2}(e)
    \\
    &= (-1)^k L(X_k)\cdots L(X_1)h_{v_1,v_2}(e)
    \\
    &= R(X_1)\cdots R(X_k)h_{v_1,v_2}(e).
  \end{align*}
  Since $f_{v_1,v_2}$ and $h_{v_1,v_2}$ are analytic functions, they
  coincide. Therefore $\varphi(\rho_1(g)v_1)= \rho_2(g)\varphi(v_1)$
  for $v_1 \in V_1$ and hence $\varphi(\rho_1(g)v)=
  \rho_2(g)\varphi(v)$ for any $v\in\calH_1$
  by~\ref{item:lem:GroupIntertwining:i}.
\end{proof}

We apply the lemma to the map
$\varphi=\Psi(\sigma,k):\calH_1\to\calH_2$ where
\begin{align*}
  \calH_1 &:=
  \left(\int_{T(\sigma,\mu)}^\oplus{L^2(\RR^m,\tfrac{1}{2}\abs{ x }^{-\Real\tau}\td
      x)\td m_{\sigma,\mu}(\tau)}\right)\boxtimes\calH^k(\RR^{n-m}),
  \\
  \calH_2 &:= L^2(\RR^n,\abs{ (x,y) }^{-\Real\sigma}\td x\td y).
\end{align*}
So let $\rho_1$ and $\rho_2$ be the restrictions of
\begin{equation*}
  \left(\int_{T(\sigma,\mu)}^\oplus{\rho_{\tau,\varepsilon+k}^{O(1,m+1)}\td
      m_{\sigma,\mu}(\tau)}\right) \boxtimes \1 \qtextq{and} 
  \rho_{\sigma,\varepsilon}^G
\end{equation*}
to $L=N_H$, respectively. To define $V_1$, we regard an element
\begin{equation*}
  f \in
  \int_{T(\sigma,\mu)}^\oplus{L^2(\RR^m,\tfrac{1}{2}\abs{ x }^{-\Real\tau}\td
    x)\td m_{\sigma,\mu}(\tau)}
\end{equation*}
as a function $f(x,\tau)$ on $(\RR^m\minuszero )\times T(\sigma,\mu)$.
Let $V_{1,c}$ be the space consisting of linear combinations of the
functions on $(\RR^m\minuszero )\times i\RR_+\times\RR^{n-m}$
of the form
\begin{equation*}
  (x,\tau,y) \mapsto (\di\rho_\tau^{O(1,m+1)}(X)\psi_\tau^{O(1,m+1)})(x)\phi(y)\chi(\tau),
\end{equation*}
where $X\in\calU(\frakh)$, $\psi_\tau^{O(1,m+1)}$ is the spherical
vector of $\rho_{\tau,\varepsilon+k}^{O(1,m+1)}$ as defined in
\eqref{eq:L2SphericalVector}, $\phi\in\calH^k(\RR^{n-m})$ and $\chi\in
C_c(i\RR_+)$, i.e.\ $\chi$ is a continuous function on
$i\RR_+$ with compact support. 
Let $V_{1,d}$ be the space consisting of sum of functions on
 $(\RR^m\minuszero )\times (T(\sigma,\mu)\cap (0,\infty))\times\RR^{n-m}$
 of the form 
\begin{equation*}
  (x,\tau,y) \mapsto f_{\tau}(x)\phi(y),
\end{equation*}
where $f_{\tau}\in L^2(\RR^m,\abs{ x }^{-\Real\tau}\td x)_{K\cap O(1,m+1)}$, a $(K\cap O(1,m+1))$-finite vector in $\rho_{\tau,\varepsilon+k}^{O(1,m+1)}$, and $\phi\in\calH^k(\RR^{n-m})$.
Then we put $V_1:=V_{1,c}\oplus V_{1,d}$.
Further let $V_2$ be the space
of all $K$-finite vectors in $L^2(\RR^n,\allowbreak \abs{ (x,y) }^{-\Real\sigma}\td x\td
y)$. We now check conditions
\ref{item:lem:GroupIntertwining:i}--\ref{item:lem:GroupIntertwining:iv}:
\begin{enumerate}
\item $V_1$ is dense in $\calH_1$ since $C_c(i\RR_+)$ is dense
  in $L^2(i\RR_+,\td m_{\sigma,\mu})$ and the space of $(K\cap
  O(1,m+1))$-finite vectors for $\rho_\tau^{O(1,m+1)}$ is generated by
  $\psi_\tau^{O(1,m+1)}(x)$ and dense in
  $L^2(\RR^m,\abs{ x }^{-\Real\tau}\td x)$ for $\tau\in i\RR_+$.
 The space $V_2$ is dense in
  $\calH_2$ since it is the space of $K$-finite vectors for
  $\rho_{\sigma,\varepsilon}^{O(1,n+1)}$.
\item $K$-finite vectors are analytic vectors for $G$ and in
  particular for $N_H\subseteq G$, hence
  $V_2\subseteq\calH_2^\omega$. Similarly, $V_{1,d}\subseteq\calH_1^\omega$.
The inclusion
  $V_{1,c}\subseteq\calH_1^\omega$ follows from the lemma below.
\item It is clear that $V_2$ is $\di\rho_2$-stable since the space of
  $K$-finite vectors is $\di\rho_\sigma^{O(1,n+1)}$-stable. That $V_1$
  is $\di\rho_1$-stable follows from the definition of $V_1$.
\end{enumerate}

\begin{lem}
  Let
  \begin{equation*}
    (\rho_1', \calH_1') :=
    \left(\int_{T(\sigma,\mu)}^\oplus{\rho_{\tau,\varepsilon+k}^{O(1,m+1)}\td
        m_{\sigma,\mu}(\tau)},
      \int_{T(\sigma,\mu)}^\oplus{L^2(\RR^m,\tfrac{1}{2}\abs{ x }^{-\Real\tau}\td
        x)\td m_{\sigma,\mu}(\tau)}\right).
  \end{equation*}
  A function $f(x,\tau)$ on $(\RR^m\minuszero ) \times i\RR_+$
  of the form
  \begin{equation*}
    f(x,\tau) :=
    (\di\rho_\tau^{O(1,m+1)}(X)\psi_\tau^{O(1,m+1)})(x)\chi(\tau)
  \end{equation*}
  for $X\in\calU(\frakh)$ and $\chi\in C_c(i\RR_+)$ is an
  analytic vector of $\rho_1'$.
\end{lem}

\begin{proof}
  It is enough to prove that for any $g_0\in
  O(1,m+1)$ there exists a neighborhood $0\in U \subset \so(1,m+1)$
  such that
  \begin{equation*}
    a_N := \norm[\Big]{ \rho_1'(\exp Y)\rho_1'(g_0)f(x,\tau)-\sum_{l=0}^N
      \frac{1}{l!}\di\rho_1'(Y)^l
      \rho_1'(g_0)f(x,\tau) }_{\calH_1'}^2 \to 0
  \end{equation*}
  as $N\to\infty$ for $Y\in U$.  Consider the Euclidean Fourier transform
  $\calF_{\RR^m}$ with respect to the variable $x$ (see
  \eqref{eq:DefFourierTransform}) which gives a unitary equivalence
  between
  \begin{equation*}
    \rho_1' =
    \int_{T(\sigma,\mu)}^\oplus{\rho_{\tau,\varepsilon+k}^{O(1,m+1)}\td
      m_{\sigma,\mu}(\tau)}  \qtextq{and}  \pi_1 :=
    \int_{T(\sigma,\mu)}^\oplus{\pi_{\tau,\varepsilon+k}^{O(1,m+1)}\td
      m_{\sigma,\mu}(\tau)}.
  \end{equation*}
  Put
  $h(x,\tau):=\calF_{\RR^m}(\di\rho_\tau^{O(1,m+1)}(X)\psi_\tau^{O(1,m+1)})(x)$
  then
  \begin{align*}
    \MoveEqLeft a_N = \int_{i\RR_+}\Bigl\lVert\pi_1(\exp
    Y)\pi_1(g_0)h(x,\tau)
    \\
    &\qquad-\sum_{l=0}^N \frac{1}{l!}\di\pi_1(Y)^l
    \pi_1(g_0)h(x,\tau)\Bigr\rVert_{L^2(\RR^m,\frac{1}{2}\abs{ x }^{-\Real\tau}\td
      x)}^2 \abs{ \chi(\tau) }^2\td m_{\sigma,\mu}(\tau).
  \end{align*}
  As in Section~\ref{sec:PrincipalSeriesReps} the function $h(x,\tau)$
  corresponds to a function $\tilde{h}(g,\tau)$ on $O(1,m+\nobreak 1)\times
  i\RR_+$ satisfying
  $\tilde{h}(gman,\tau)=\xi_{\varepsilon+k}(m)^{-1}a^{-\tau-\rho}\tilde{h}(g,\tau)$
  for $m\in O(1,m+1)\cap M$, $a\in A$ and $n\in N_H$.  Consequently,
  $a_N$ is given as
  \begin{align*}
    \MoveEqLeft \int_{i\RR_+}\Bigl(\int_{O(1)\times O(m+1)}
    \Bigl\lvert\pi_{\tau,\varepsilon+k}^{O(1,m+1)}(g_0)\tilde{h}(\exp (-Y)
    k,\tau)
    \\
    &-\sum_{l=0}^N \frac{1}{l!}\di\pi_\tau^{O(1,m+1)}(Y)^l
    \pi_{\tau,\varepsilon+k}^{O(1,m+1)}(g_0)\tilde{h}(k,\tau)\Bigr\rvert^2
    \td k\Bigr) \abs{ \chi(\tau) }^2\td m_{\sigma,\mu}(\tau)
  \end{align*}
  up to a constant factor, where $\di k$ is the Haar measure on
  $O(1)\times O(m+1)$. Since
  $\pi_{\tau,\varepsilon+k}^{O(1,m+1)}(g_0)\tilde{h}$ is analytic on
  $O(1,m+1)\times i\RR_+$, the sequence
  \begin{equation*}
    \sum_{l=0}^N \frac{1}{l!}\di\pi_\tau^{O(1,m+1)}(Y)^l
    \pi_{\tau,\varepsilon+k}^{O(1,m+1)}(g_0)\tilde{h}(k,\tau)
  \end{equation*}
  converges uniformly to
  $\pi_{\tau,\varepsilon+k}^{O(1,m+1)}(g_0)\tilde{h}(\exp (-Y)
  k,\tau)$ on the compact set $(k,\tau)\in (O(1)\times O(m+1))\times
  \supp \chi$, which proves $a_N\to 0$.
\end{proof}

To verify the intertwining condition
\ref{item:lem:GroupIntertwining:iv} we first prove the intertwining
property for each single space $L^2(\RR^m,\abs{ x }^{-\Real\tau}\td x)$ for
fixed $\tau$ by embedding it into the $\CC$-antilinear algebraic dual
of the Harish-Chandra module $L^2(\RR^n,\abs{ (x,y) }^{-\Real\sigma}\td x\td
y)_K$ of $K$-finite vectors. For $\tau\in i\RR_+$ and
$X\in\calU(\frakh)$ let
\begin{equation*}
  f_{\tau,X}(x) :=
  (\di\rho_\tau^{O(1,m+1)}(X)\psi_\tau^{O(1,m+1)})(x), 
  \qqtext{$x\in\RR^m\minuszero $}.
\end{equation*}

\begin{prop}\label{prop:EmbeddingIntoHCDual}
  Let $\phi\in\calH^k(\RR^{n-m})$ and $g\in
  L^2(\RR^n,\abs{ (x,y) }^{-\Real\sigma}\td x\td y)_K$.
  \begin{enumerate}
  \item\label{item:prop:EmbeddingIntoHCDual:i} For every $X\in\calU(\frakh)$
  and $\tau\in i\RR_+$ the integral
    \begin{equation*}
      \int_{\RR^n}{\abs{ x }^{\frac{\sigma-\tau-\mu}{2}}
        F(\tfrac{\abs{ y }^2}{\abs{ x }^2},\tau)f_{\tau,X}(x)\phi(y)
        \overline{g(x,y)}\abs{ (x,y) }^{-\Real\sigma}\td  x\td y}
    \end{equation*}
    converges absolutely and defines a continuous function in $\tau$.
  \item\label{item:prop:EmbeddingIntoHCDual:ii} For every $X\in\calU(\frakh)$,
  $\tau\in i\RR_+$ and $j=1,\ldots,m$ we have
    \begin{align}
      \MoveEqLeft[1] \int_{\RR^n}{\abs{ x }^{\frac{\sigma-\tau-\mu}{2}}
        F(\tfrac{\abs{ y }^2}{\abs{ x }^2},\tau)(\calB_j^{m,\tau}
        f_{\tau,X})(x)\phi(y)\overline{g(x,y)}\abs{ (x,y) }^{-\Real\sigma}\td
        x\td y}
      \nonumber
      \\
      &=
      \int_{\RR^n}{\abs{ x }^{\frac{\sigma-\tau-\mu}{2}}
        F(\tfrac{\abs{ y }^2}{\abs{ x }^2},\tau)f_{\tau,X}(x)\phi(y)
        \overline{(\calB_j^{n,\sigma}g)(x,y)}
        \abs{ (x,y) }^{-\Real\sigma}\td x\td y}.
      \label{eq:IntertwiningPropertyHCDual}
    \end{align}
  \item\label{item:prop:EmbeddingIntoHCDual:iii} 
 For $\tau\in T(\sigma,\mu)\cap (0,\infty)$ and 
 $f_{\tau}\in L^2(\RR^m,\abs{ x }^{-\Real\tau}\td x)_{K\cap O(1,m+1)}$,
 the integral
    \begin{equation*}
      \int_{\RR^n}{\abs{ x }^{\frac{\sigma-\tau-\mu}{2}}
        F(\tfrac{\abs{ y }^2}{\abs{ x }^2},\tau)f_{\tau}(x)\phi(y)
        \overline{g(x,y)}\abs{ (x,y) }^{-\Real\sigma}\td  x\td y}
    \end{equation*}
    converges absolutely.
  \item\label{item:prop:EmbeddingIntoHCDual:iv}
  For $\tau\in T(\sigma,\mu)\cap (0,\infty)$,
  $f_{\tau}\in L^2(\RR^m,\abs{ x }^{-\Real\tau}\td x)_{K\cap O(1,m+1)}$,
  and $j=1,\ldots,m$, 
  we have
    \begin{align}
      \MoveEqLeft[1] \int_{\RR^n}{\abs{ x }^{\frac{\sigma-\tau-\mu}{2}}
        F(\tfrac{\abs{ y }^2}{\abs{ x }^2},\tau)(\calB_j^{m,\tau}
        f_{\tau})(x)\phi(y)\overline{g(x,y)}\abs{ (x,y) }^{-\Real\sigma}\td
        x\td y}
      \nonumber
      \\
      &=
      \int_{\RR^n}{\abs{ x }^{\frac{\sigma-\tau-\mu}{2}}
        F(\tfrac{\abs{ y }^2}{\abs{ x }^2},\tau)f_{\tau}(x)\phi(y)
        \overline{(\calB_j^{n,\sigma}g)(x,y)}
        \abs{ (x,y) }^{-\Real\sigma}\td x\td y}.
      \label{eq:IntertwiningPropertyHCDualDiscrete}
    \end{align}
  \end{enumerate}
\end{prop}

\begin{proof}
  We first note that by \eqref{eq:gKmodule} the function
  $f_{\tau,X}(x)$ is a linear combination of functions of the form
  \begin{equation*}
    f(x) = \widetilde{K}_{-\frac{\tau}{2}+a}(\abs{ x })\abs{ x }^{2a}p(x)
  \end{equation*}
  for $a\in\NN$ and $p\in\CC[x]$ with coefficients depending
  smoothly on $\tau$.
 Similarly, by \eqref{eq:gKmodule} and \eqref{eq:gKmoduleSub},
 $f_{\tau}(x)$ is a linear combination of the form $f(x)$ above.
 Note that in the case $\tau=m+2v\in m+2\NN$ we additionally
  have $p\in\CC[x]_{>v-a}$. 
 We may replace $f_{\tau,X}(x)$ and $f_{\tau}(x)$ by one of
 these functions $f(x)$. For the same reason we may assume that
  \begin{equation*}
    g(x,y) =
    \widetilde{K}_{-\frac{\sigma}{2}+b}(\abs{ (x,y) })\abs{ (x,y) }^{2b}q(x,y)
  \end{equation*}
  for some $b\in\NN$ and $q\in\CC[x,y]$ where for $\sigma=n+2u\in n+2\NN$ we
  additionally have $q\in\CC[x,y]_{>u-b}$.
  \begin{enumerate}[labelindent=0pt]
  \item[(i)\&(iii)] \label{item:prop:EmbeddingIntoHCDual:proof:i} By
    \eqref{eq:KBesselAsymptotics0} and
    \eqref{eq:KBesselAsymptoticsInfty} there exists a continuous
    function $C_1(\tau)>0$ on $T(\sigma,\mu)$ and $N_1>0$ such that for $x\neq0$:
    \begin{multline*}
      \abs{ \widetilde{K}_{-\frac{\tau}{2}+a}(\abs{x})\abs{x}^{2a}p(x) }\\
      \leq
      C_1(\tau)|x|^{-\delta_1}(1+|x|)^{N_1}e^{-|x|}
      \begin{cases}
        1&\mbox{for $0\leq\Real\tau<m$,}\\
        \abs{x}^{\frac{\Real\tau-m}{2}+1}&\mbox{for $\Real\tau\geq m$,}
      \end{cases}
    \end{multline*}
    for some arbitrarily small $\delta_1>0$ (covering the possible
    $\log$-term for $\tau=2a$).
    For the hypergeometric function we have by \eqref{eq:2F1Identity1}
    and \eqref{eq:HypergeometricPolynomials} (checking the cases
    $\tau\in i\RR_+$ and $\tau\in(\Real\sigma-\mu-4\NN)\cap\RR_+$
    separately)
    \begin{equation*}
      \abs{ F(t,\tau) } \leq
      C_2(\tau)(1+t)^{\frac{\Real\sigma-\Real\tau-\mu}{4}}, \qqtext{$ t>0$},
    \end{equation*}
    for some continuous function $C_2(\tau)>0$ on $T(\sigma,\mu)$. We
    estimate
    $$ \abs{ \phi(y) } \leq C_3\abs{ y }^k \leq C_3\abs{ (x,y) }^k. $$
    Further, for the $K$-Bessel function of parameter
    $-\frac{\sigma}{2}+b$ we again find by \eqref{eq:KBesselAsymptotics0}
    and \eqref{eq:KBesselAsymptoticsInfty} that for $(x,y)\neq0$:
    \begin{multline*}
      \abs{ \widetilde{K}_{-\frac{\sigma}{2}+b}(\abs{(x,y)})\abs{(x,y)}^{2b}q(x,y) }\\
      \leq
      C_4\abs{(x,y)}^{-\delta_2}(1+\abs{(x,y)})^{N_2}e^{-\abs{(x,y)}}
      \begin{cases}
        1&\mbox{for $0\leq\Real\sigma<n$,}\\
        \abs{(x,y)}^{\frac{\Real\sigma-n}{2}+1}&\mbox{for $\Real\sigma\geq n$,}
      \end{cases}
    \end{multline*}
    for some arbitrarily small $\delta_2>0$ (covering the possible
    $\log$-term for $\sigma=2b$) and $C_4,N_2>0$. Now assume $0\leq\Real\tau<m$ and
    $0\leq\Real\sigma<n$ then we obtain
    \begin{align*}
      \MoveEqLeft[3]
      \abs*{ \abs{ x }^{\frac{\sigma-\tau-\mu}{2}}F(\tfrac{\abs{ y }^2}{\abs{ x }^2},\tau)
        f(x)\phi(y)\overline{g(x,y)} }
      \\
      \leq{}&
      C_1(\tau)C_2(\tau)C_3C_4\abs{ x }^{\frac{\Real\sigma-\Real\tau-\mu}{2}}
      \left(1+\tfrac{\abs{ y }^2}{\abs{ x }^2}\right)^{\frac{\Real\sigma-\Real\tau-\mu}{4}}
      \\
      &
      \times\abs{x}^{-\delta_1}(1+\abs{ x })^{N_1}e^{-\abs{ x }}\abs{ (x,y)
      }^{k-\delta_2}
      (1+\abs{ (x,y) })^{N_2}e^{-\abs{ (x,y) }}
      \\
      \leq{}&
      C(\tau)\abs{x}^{-\delta_1}\abs{ (x,y)
      }^{\frac{\Real\sigma-n+m-\Real\tau}{2}-\delta_2}
      (1+\abs{ (x,y) })^Ne^{-\abs{ (x,y) }}
    \end{align*}
    with $C(\tau)=C_1(\tau)C_2(\tau)C_3C_4$ and
    $N=N_1+N_2$. Since $\delta_1$ and $\delta_2$ can be chosen arbitrarily
    small the right hand side is integrable on $\RR^n$ with respect to
    the measure $\abs{ (x,y) }^{-\Real\sigma}$ if and only if
    \begin{equation*}
      \frac{-\Real\sigma-n+m-\Real\tau}{2} > -n.
    \end{equation*}
    This inequality holds by assumption and hence the integral converges absolutely. Moreover, we even have
    $n-\Real\sigma+m-\Real\tau>n-\Real\sigma>0$ for all $\tau$ and
    hence the convergence is uniformly in $\tau$ varying in a compact
    subset of $T(\sigma,\mu)$. The other two possibilities $0\leq\Real\tau<m$,
    $\Real\sigma\geq n$ and $\Real\tau\geq m$, $\Real\sigma\geq n$ are treated
    similarly which finishes the proof of~\ref{item:prop:EmbeddingIntoHCDual:i}
 \& \ref{item:prop:EmbeddingIntoHCDual:iii}.
  \item[(ii)\&(iv)]
 First recall from Proposition~\ref{prop:IntertwiningProperty}
    that
    \begin{equation*}
       \abs{ x }^{\frac{\sigma-\tau-\mu}{2}}
      F(\tfrac{\abs{ y }^2}{\abs{ x }^2},\tau)(\calB_j^{m,\tau}f)(x)\phi(y)
      = \calB_j^{n,\sigma}\left[\abs{ x }^{\frac{\sigma-\tau-\mu}{2}}
        F(\tfrac{\abs{ y }^2}{\abs{ x }^2},\tau)f(x)\phi(y)\right].
    \end{equation*}
    Therefore we have to show that
    \begin{align}
      \MoveEqLeft
      \int_{\RR^n}{\calB_j^{n,\sigma}\Phi(x,y)\cdot\overline{g(x,y)}
        \cdot\abs{ (x,y) }^{-\Real\sigma}\td x\td y}
      \nonumber
      \\
      &\stackrel{!}{=}
      \int_{\RR^n}{\Phi(x,y)\cdot\overline{\calB_j^{n,\sigma}g(x,y)}
        \cdot\abs{ (x,y) }^{-\Real\sigma}\td  x\td y},
      \label{eq:BesselSAOnKfinite1}
    \end{align}
    where we abbreviate
    \begin{equation*}
      \Phi(x,y) =
      \abs{ x }^{\frac{\sigma-\tau-\mu}{2}}F(\tfrac{\abs{ y }^2}{\abs{ x }^2},\tau)f(x)\phi(y).
    \end{equation*}
    The operator $\calB_j^{n,\sigma}$ is formally self-adjoint with
    respect to $\abs{ (x,y) }^{-\Real\sigma}$ since
    $\di\rho_\sigma^G(N_j)=-i\calB_j^{n,\sigma}$ is, as part of the
    Lie algebra action, formally skew-adjoint on
    $C_c^\infty(\RR^n\minuszero )\subseteq
    L^2(\RR^n,\abs{ (x,y) }^{-\Real\sigma}\td x\td y)^\infty$. Therefore it
    remains to show that we can integrate by parts without leaving any
    boundary terms. Fix $j\in\{1,\ldots,m\}$ and consider the domain
    \begin{equation*}
      \Omega_{j,\varepsilon} := \Set{ (x,y)\in\RR^n }{
        \abs{ x_j }>\varepsilon } \subseteq \RR^n
    \end{equation*}
    for $\varepsilon>0$. Clearly
    $\RR^n\setminus\bigcup_{\varepsilon>0}{\Omega_{j,\varepsilon}}$ is
    of measure zero and hence \eqref{eq:BesselSAOnKfinite1} is
    equivalent to
    \begin{align}
      \MoveEqLeft
      \lim_{\varepsilon\to0}\int_{\Omega_{j,\varepsilon}}{\calB_j^{n,\sigma}
        \Phi(x,y)\cdot\overline{g(x,y)}\cdot\abs{ (x,y) }^{-\Real\sigma}\td
        x\td y}
      \nonumber
      \\
      &\stackrel{!}{=}
      \lim_{\varepsilon\to0}\int_{\Omega_{j,\varepsilon}}{\Phi(x,y)
        \cdot\overline{\calB_j^{n,\sigma}g(x,y)}\cdot\abs{ (x,y) }^{-\Real\sigma}\td
        x\td y}.\label{eq:BesselSAOnKfinite2}
    \end{align}
    On $\Omega_{j,\varepsilon}$ both $\abs{ x }$ and $\abs{ (x,y) }$ are bounded
    from below by $\varepsilon$. Hence,
    by~\eqref{eq:DerivativeKBessel}
    and~\eqref{eq:DerivativeHypergeometricFunction}, all factors in
    the integrand
    \begin{equation*}
      \abs{ x }^{\frac{\sigma-\tau-\mu}{2}}F(\tfrac{\abs{ y }^2}{\abs{ x }^2},\tau)
      f(x)\phi(y)\overline{g(x,y)}\abs{ (x,y) }^{-\Real\sigma}
    \end{equation*}
    can be arbitrarily often differentiated in $x$ and $y$ and the
    result is a smooth function on
    $\overline{\Omega_{j,\varepsilon}}$. Since further the
    hypergeometric function grows at most polynomially and the
    $K$-Bessel functions decay exponentially near $\infty$, all such
    differentiated terms decay exponentially as $\abs{ (x,y) }\to\infty$ and
    are hence integrable on
    $\overline{\Omega_{j,\varepsilon}}$. Therefore we can arbitrarily
    integrate by parts and all intermediate integrals exist. It
    remains to show that for $\varepsilon\to0$ all boundary terms that
    occur while integrating by parts vanish. By the asymptotic
    behaviour of the $K$-Bessel functions at $\infty$ the boundary
    terms at $\infty$ always vanish. Hence, by the choice of
    $\Omega_{j,\varepsilon}$, the only boundary terms that occur are
    for derivatives in $x_j$ at $x_j=\pm\varepsilon$. Therefore we
    only need to consider the parts
    \begin{equation*}
      x_j\frac{\partial^2}{\partial x_j^2}, \frac{\partial}{\partial
        x_j} \qtextq{and} E\frac{\partial}{\partial x_j}
    \end{equation*}
    of $\calB_j^{n,\sigma}$. We treat these three parts
    separately. Here we start with the right hand side of
    \eqref{eq:BesselSAOnKfinite2} and then integrate by parts once or
    twice. We only carry out the details for the case $0\leq\Real\tau<m$
    and $0\leq\Real\sigma<n$, the other cases are treated similarly with
    the corresponding estimates derived in part \ref{item:prop:EmbeddingIntoHCDual:i} \& \ref{item:prop:EmbeddingIntoHCDual:iii}.
    \begin{enumerate}[labelindent=0pt,ref=(\alph*)] 
    \item $\frac{\partial}{\partial x_j}$. The boundary terms that
      occur when integrating by parts are (up to multiplication with a
      constant) of the form
      \begin{align*}
        \MoveEqLeft \int_{\RR^{n-1}} \Bigl(\Phi(x',\varepsilon,y)
        \overline{g(x',\varepsilon,y)}\abs*{ (x',\varepsilon,y)
        }^{-\Real\sigma}
        \\
        &\qquad-\Phi(x',-\varepsilon,y)\overline{g(x',-\varepsilon,y)}
        \abs*{ (x',-\varepsilon,y) }^{-\Real\sigma}\Bigr)\td x'\td
        y
      \end{align*}
      where we write $x=(x',x_j)$ with
      $x'=(x_1,\ldots,\widehat{x_j},\ldots,x_m)\in\RR^{m-1}$. The
      integrand obviously converges pointwise almost everywhere to $0$
      as $\varepsilon\to0$ and it suffices to find an integrable
      function independent of $\varepsilon$ dominating the integrand
      to apply the Dominated Convergence Theorem. For this note that
      in both $\Phi(x,y)$ and $g(x,y)$ the only terms dependent on the
      sign of $x_j$ are the polynomials $p(x)$ and $q(x,y)$,
      respectively. Using the same estimates as in the proof of part
      \ref{item:prop:EmbeddingIntoHCDual:i}
     \& \ref{item:prop:EmbeddingIntoHCDual:iii}
      we find that
      \begin{align*}
        \MoveEqLeft 
        \abs*{ \Phi(x',\varepsilon,y)\overline{g(x',\varepsilon,y)}
          \abs{ (x',\varepsilon,y) }^{-\Real\sigma}
        -\Phi(x',-\varepsilon,y)
          \overline{g(x',-\varepsilon,y)}\abs{ (x',-\varepsilon,y) }^{-\Real\sigma}
        }
        \\
        \leq{}&
        C\abs{ (x',\varepsilon,y) }^{\frac{-\Real\sigma-n+m-\Real\tau}{2}-\delta}
        (1+\abs{ (x',\varepsilon,y) })^Ne^{-\abs{ (x',\varepsilon,y) }}
        \\
        &
        \times\abs*{ p(x',\varepsilon)\overline{q(x',\varepsilon,y)}
          -p(x',-\varepsilon)\overline{q(x',-\varepsilon,y)} }
        \intertext{for some $N>0$ and an arbitrarily small
          $\delta>0$. Now note that
          $p(x',\varepsilon)\overline{q(x',\varepsilon,y)}
          -p(x',-\varepsilon)\overline{q(x',-\varepsilon,y)}$
          is an odd polynomial in $\varepsilon$ and hence of the form
          $\varepsilon\cdot r(x',\varepsilon,y)$. For the extra
          $\varepsilon$ from this observation we use the estimate
          $\abs{ \varepsilon }\leq\abs{ (x',\varepsilon,y) }$. We further estimate
          $\abs{ r(x',\varepsilon,y) }\leq C'(1+\abs{ (x',\varepsilon,y) })^{N'}$
          for some $C',N'>0$ and find (assuming $\varepsilon\leq1$)}
        \leq{}&
        CC'\abs{ (x',\varepsilon,y) }^{\frac{-\Real\sigma-n+m-\Real\tau}{2}+1-\delta}
        (1+\abs{ (x',1,y) })^{N+N'}e^{-\abs{ (x',y) }}.
        \intertext{Now suppose the exponent
          $\frac{-\Real\sigma-n+m-\Real\tau}{2}+1$ is $\leq0$. Then we
          can estimate} \leq{}&
        CC'\abs{ (x',y) }^{\frac{-\Real\sigma-n+m-\Real\tau}{2}
          +1-\delta}(1+\abs{ (x',1,y) })^{N+N'}e^{-\abs{ (x',y) }},
      \end{align*}
      which is independent of $\varepsilon\in(0,1)$ and integrable on
      $\RR^{n-1}$ for small $\delta>0$ since $\Real\sigma<n$ and
      $\Real\tau<m$. If the exponent
      $\frac{-\Real\sigma-n+m-\Real\tau}{2}+1-\delta$ is positive the
      estimate $\varepsilon\leq1$ also yields a dominant integrable
      function independent of $\varepsilon$. Therefore, in both cases
      we can apply the Dominated Convergence Theorem and obtain that
      as $\varepsilon\to0$ the boundary terms vanish.
    \item\label{item:prop:EmbeddingIntoHCDual:proof:ii:b}
      $x_j\frac{\partial^2}{\partial x_j^2}$. Integrating by part once
      gives (up to multiplication by a constant) the boundary terms
      \begin{align}
        \int_{\RR^{n-1}}
        \biggl(&\Phi(x',\varepsilon,y)\left(x_j\overline{\frac{\partial
              g}{\partial
              x_j}(x,y)}\right)_{x_j=\varepsilon}\abs{ (x',\varepsilon,y) }^{-\Real\sigma}
        \nonumber
        \\
        &-\Phi(x',-\varepsilon,y)\left(x_j\overline{\frac{\partial
              g}{\partial  x_j}(x,y)}
        \right)_{x_j=-\varepsilon}\abs{ (x',-\varepsilon,y) }^{-\Real\sigma}
            \biggr)\td  x'\td y.\label{eq:BdyTerms2}
      \end{align}
      We have
      \begin{equation*}
        g(x,y) =
        \widetilde{K}_{-\frac{\sigma}{2}+b}(\abs{ (x,y) })\abs{ (x,y) }^{2b}q(x,y)
      \end{equation*}
      and use the product rule to find $x_j\frac{\partial g}{\partial
        x_j}(x,y)$. The first term is by \eqref{eq:DerivativeKBessel}
      \begin{align*}
        \MoveEqLeft[3] 
        -\frac{x_j^2}{2}\widetilde{K}_{-\frac{\sigma}{2}+b+1}
        (\abs{ (x,y) })\abs{ (x,y) }^{2b}q(x,y)
        \\
        ={}&
        -\frac{x_j^2}{2\abs{ (x,y) }^2}\widetilde{K}_{-\frac{\sigma}{2}+b+1}
        (\abs{ (x,y) })\abs{ (x,y) }^{2(b+1)}q(x,y)
        \\
        \intertext{and putting $x_j=\pm\varepsilon$ gives}
        ={}&
        -\frac{\varepsilon^2}{2\abs{ (x',\varepsilon,y) }^2}
        \widetilde{K}_{-\frac{\sigma}{2}+b+1}(\abs{ (x',\varepsilon,y) })
        \abs{ (x',\varepsilon,y) }^{2(b+1)}q(x',\pm\varepsilon,y).
      \end{align*}
      Again $\varepsilon^2$ can be estimated by
      $\abs{ (x',\varepsilon,y) }^2$ and we find that
      \begin{equation*}
        \frac{\varepsilon^2}{2\abs{ (x',\varepsilon,y) }^2}
        \widetilde{K}_{-\frac{\sigma}{2}+b+1}
        (\abs{ (x',\varepsilon,y) })\abs{ (x',\varepsilon,y) }^{2(b+1)}
        \qtextq{and} 
        \widetilde{K}_{-\frac{\sigma}{2}+b}(\abs{ (x,y) })\abs{ (x,y) }^{2b}
      \end{equation*}
      satisfy the same estimates (see the proof of 
     \ref{item:prop:EmbeddingIntoHCDual:i}
  \& \ref{item:prop:EmbeddingIntoHCDual:iii}). The same
      argument applies to the other two terms in the product
      rule. Therefore the same argument as in (a) yields the vanishing
      of the boundary terms \eqref{eq:BdyTerms2}. Similar arguments
      yield the vanishing of the boundary terms that occur when
      integrating by parts for the second time. For this note that the
      formal adjoint of $\frac{\partial}{\partial x_j}$ on
      $L^2(\RR^n,\abs{ (x,y) }^{-\Real\sigma}\td x\td y)$ is
      $-\frac{\partial}{\partial
        x_j}+(\Real\sigma)\frac{x_j}{\abs{ (x,y) }^2}$. Both summands are
      treated separately as above.
    \item $E\frac{\partial}{\partial x_j}$. We have
      \begin{equation*}
        E\frac{\partial}{\partial x_j} =
        x_j\frac{\partial^2}{\partial x_j^2}+\sum_{k\neq
          j}{x_k\frac{\partial}{\partial x_k}\frac{\partial}{\partial
            x_j}}+\sum_k{y_k\frac{\partial}{\partial
            y_k}\frac{\partial}{\partial x_j}}.
      \end{equation*}
      The first term was already treated in part
      \ref{item:prop:EmbeddingIntoHCDual:proof:ii:b}. For the other
      two terms note that we can first integrate by parts the
      derivatives with respect to $x_k$ ($k\neq j$) and $y_k$ without
      any boundary terms occurring. Secondly, integration by parts of
      the derivative with respect to $x_j$ is dealt with as in
      part~\ref{item:prop:EmbeddingIntoHCDual:proof:ii:b}. This
      finishes the proof.\qedhere
    \end{enumerate}
  \end{enumerate}
\end{proof}

\begin{rem}
  It is necessary in the proof of
  Proposition~\ref{prop:EmbeddingIntoHCDual}~\ref{item:prop:EmbeddingIntoHCDual:ii} \& \ref{item:prop:EmbeddingIntoHCDual:iv}
  to restrict integration to the domain $\Omega_{j,\varepsilon}$. This
  is because the operator $\calB_j^{n,\sigma}$ is of second order and
  we have to integrate by parts twice. The intermediate result, i.e.\
  after integrating by parts once, may not be integrable on $\RR^n$
  and hence we need to restrict to a subdomain on which these
  intermediate results are integrable. The same problem occurs when
  one considers the two summands $x_j\Delta$ and
  $-(2E-\sigma+n)\frac{\partial}{\partial x_j}$ separately. Here the
  integral over $\RR^n$ for each of the two summands may not converge
  while the integral for the sum $\calB_j^{n,\sigma}$ does by
  Proposition~\ref{prop:EmbeddingIntoHCDual}~\ref{item:prop:EmbeddingIntoHCDual:i} \& \ref{item:prop:EmbeddingIntoHCDual:iii}.
\end{rem}

\begin{rem}
  The assertions \ref{item:prop:EmbeddingIntoHCDual:i}
  \& \ref{item:prop:EmbeddingIntoHCDual:iii} of
  Proposition~\ref{prop:EmbeddingIntoHCDual} construct an embedding
  of
  \begin{equation*}
    L^2(\RR^m,\abs{ x }^{-\Real\tau}\td x)_{K\cap
      O(1,m+1)}\boxtimes\calH^k(\RR^{n-m})
  \end{equation*}
  into the $\CC$-antilinear algebraic dual of
  $L^2(\RR^n,\abs{ (x,y) }^{-\Real\sigma}\td x\td y)_K$ for every
  $\tau\in
  T(\sigma,\mu)$. By \ref{item:prop:EmbeddingIntoHCDual:ii}
  \& \ref{item:prop:EmbeddingIntoHCDual:iv} this
  embedding is $\frakh$-equivariant.
\end{rem}

Let us now continue the proof of Theorem~\ref{thm:HIntertwiner} by
showing property~\ref{item:lem:GroupIntertwining:iv} in
Lemma~\ref{lem:GroupIntertwining}.  Let $v_1\in V_{1,c}$ and $v_2\in
V_2$. Suppose that
\begin{equation*}
  v_1(x,\tau,y) = f_{\tau,X}(x)\phi(y)\chi(\tau) \qtextq{and}
  v_2(x,y) = g(x,y)
\end{equation*}
with $X\in\calU(\frakh)$, $\chi\in C_c(T(\sigma,\mu))$,
$\phi\in\calH^k(\RR^{n-m})$, and $g\in
L^2(\RR^n,\abs{ (x,y) }^{-\Real\sigma}\td x\td y)_K$. We have
\begin{align*}
  \MoveEqLeft[3] \Hermit{\varphi(\di\rho_1(N_j)v_1)}{v_2}_{\calH_2}
  \\
  ={}&
  -i\int_{\RR^n}{\int_{T(\sigma,\mu)}{\abs{ x }^{\frac{\sigma-\tau-\mu}{2}}
      F(\tfrac{\abs{ y }^2}{\abs{ x }^2},\tau)(\calB_j^{m,\tau}
      f_{\tau,X})(x)\phi(y)\overline{g(x,y)}\abs{ (x,y) }^{-\Real\sigma}}}
  \\
  & \hspace{8cm} \chi(\tau)\td m_{\sigma,\mu}(\tau)\td x\td y
  \\
  ={}& -i\int_{T(\sigma,\mu)}
  \int_{\RR^n}{\abs{ x }^{\frac{\sigma-\tau-\mu}{2}}
    F(\tfrac{\abs{ y }^2}{\abs{ x }^2},\tau)(\calB_j^{m,\tau}f_{\tau,X})(x)
    \phi(y)\overline{g(x,y)}\abs{ (x,y) }^{-\Real\sigma}}
  \\
  & \hspace{8cm} \chi(\tau)\td x\td y \td m_{\sigma,\mu}(\tau),
  \intertext{where we were able to change the order of integration,
    because by
    Proposition~\ref{prop:EmbeddingIntoHCDual}~\ref{item:prop:EmbeddingIntoHCDual:i}
    the inner integral in the last line converges absolutely and is
    continuous in $\tau$ and the integration is only over the compact
    subset $\supp\chi\subseteq T(\sigma,\mu)$. Now, by
    Proposition~\ref{prop:EmbeddingIntoHCDual}~\ref{item:prop:EmbeddingIntoHCDual:ii}
    we find} ={}& -i\int_{T(\sigma,\mu)}
  \int_{\RR^n}{\abs{ x }^{\frac{\sigma-\tau-\mu}{2}}
    F(\tfrac{\abs{ y }^2}{\abs{ x }^2},\tau)f_{\tau,X}(x)\phi(y)
    \overline{\calB_j^{n,\sigma}g(x,y)}\abs{ (x,y) }^{-\Real\sigma}}
  \\
  & \hspace{8cm} \chi(\tau)\td x\td y \td m_{\sigma,\mu}(\tau)
  \\
  ={}& -i\int_{\RR^n}
  \int_{T(\sigma,\mu)}{\abs{ x }^{\frac{\sigma-\tau-\mu}{2}}
    F(\tfrac{\abs{ y }^2}{\abs{ x }^2},\tau)f_{\tau,X}(x)\phi(y)
    \overline{\calB_j^{n,\sigma}g(x,y)}\abs{ (x,y) }^{-\Real\sigma}}
  \\
  & \hspace{8cm} \chi(\tau)\td m_{\sigma,\mu}(\tau) \td x\td y
  \\
  ={}& \Hermit{\varphi(v_1)}{\di\rho_2(N_j)v_2}_{\calH_2},
\end{align*}
again using
Proposition~\ref{prop:EmbeddingIntoHCDual}~\ref{item:prop:EmbeddingIntoHCDual:i}
to change the order of integration.  This shows property
\ref{item:lem:GroupIntertwining:iv} of
Lemma~\ref{lem:GroupIntertwining}
for $v_1\in V_{1,c}$.
A similar argument with Proposition~\ref{prop:EmbeddingIntoHCDual}~\ref{item:prop:EmbeddingIntoHCDual:iii} and \ref{item:prop:EmbeddingIntoHCDual:iv} shows
Lemma~\ref{lem:GroupIntertwining}~\ref{item:lem:GroupIntertwining:iv}
for $v_1\in V_{1,d}$.
We therefore obtain that
$\varphi=\Psi(\sigma,k)$ intertwines the group action of $N_H$ and
hence of $H$. Thus the proof of Theorem~\ref{thm:HIntertwiner} is
complete.
\end{proof}

We obtain the whole spectral decomposition of
$\rho_{\sigma,\varepsilon}^G \restrictedto H$ from
\eqref{eq:On-mDecomposition} and Theorem~\ref{thm:HIntertwiner}.

\begin{thm}\label{thm:RepDecomp}
  For $\sigma\in i\RR\cup(-n,n)\cup(n+2\NN)$ the representation
  $\rho_{\sigma,\varepsilon}^G$ decomposes under the restriction to
  $H=O(1,m+1)\times O(n-m)$, $0<m<n$, as
  \begin{align*}
     \rho_{\sigma,\varepsilon}^G \restrictedto[\big] H \cong {} &
    \sideset{}{^\oplus}\sum_{k=0}^\infty
  \Bigl(\int_{i\RR_+}^\oplus{\rho_{\tau,\varepsilon+k}^{O(1,m+1)}\td\tau}
    \\
    &\oplus\bigoplus_{j\in\ZZ\cap
      \left[0,\frac{\abs{ \Real\sigma }-n+m-2k}{4}\right)}
    {\rho_{\abs{ \Real\sigma }-n+m-2k-4j,\varepsilon+k}^{O(1,m+1)}}\Bigr)
    \boxtimes\calH^k(\RR^{n-m}).
  \end{align*}
\end{thm}

\section{Intertwining operators in the non-compact  picture}
\label{sec:IntertwinersNoncptPicture}

In Proposition~\ref{prop:IntertwiningProperty} we explicitly found an
intertwining operator $C^\infty(\RR^m\minuszero
)\boxtimes\calH^k(\RR^{n-m})\to C^\infty(\RR^n\setminus\{x=0\})$. In
the Fourier transformed picture this operator is given by
\begin{equation*}
  A(\sigma,\tau)(f\otimes\phi)(x,y) =
  \abs{ x }^{\frac{\sigma-\tau-\mu}{2}}{_2F_1}
  \left(\frac{\mu-\sigma+\tau}{4},\frac{\mu-\sigma-\tau}{4};
    \frac{\mu}{2};-\frac{\abs{ y }^2}{\abs{ x }^2}\right)f(x)\phi(y),
\end{equation*}
where again $\mu=2k+n-m$. In
Proposition~\ref{prop:EmbeddingIntoHCDual} we even showed that for
fixed $\sigma\in i\RR\cup(-n,n)\cup(n+2\NN)$, $k\in\NN$ and $\tau\in
T(\sigma,2k-n+m)$ the operator $A(\sigma,\tau)$ is intertwining
between the Harish-Chandra module of
$\rho_{\tau,\varepsilon+k}^{O(1,m+1)}\boxtimes\calH^k(\RR^{n-m})$ and
the $\CC$-antilinear algebraic dual of the Harish-Chandra module of
$\smash{\rho_{\sigma,\varepsilon}^{O(1,n+1)}}$. We now find a formal
expression for this intertwiner in the non-compact picture.

Consider the following diagram
\begin{equation*}
  \xymatrix{
    **[l] C_c^\infty(\RR^m\minuszero )\otimes\calH^k(\RR^{n-m}) \ar[r]^{A(\sigma,\tau)} \ar@<-1.68cm>[d]_{\calF_{\RR^m}\otimes\id} & \calS'(\RR^n) \ar[d]^{\calF_{\RR^n}}\\
    **[l] \calF_{\RR^m}C_c^\infty(\RR^m\minuszero )\otimes\calH^k(\RR^{n-m}) \ar[r]_{I(\sigma,\tau)} & \calS'(\RR^n).
  }
\end{equation*}
We extend the operator $A(\sigma,\tau)$ for all $\sigma, \tau\in \CC$
and determine the operator $I(\sigma,\tau)$ for $\Real \sigma \ll
\Real \tau \ll 0$.  We have
\begin{align*}
  \MoveEqLeft[3] \calF_{\RR^n}A(\sigma,\tau)(f\otimes\phi)(\xi,\eta)
  \\
  ={}&
  (2\pi)^{-\frac{n}{2}}\int_{\RR^m}
  \int_{\RR^{n-m}} e^{-ix\cdot\xi-iy\cdot\eta}
    \abs{ x }^{\frac{\sigma-\tau-\mu}{2}}
    \\
    &\qquad\qquad
    \times{_2F_1}\left(\frac{\mu-\sigma+\tau}{4},
      \frac{\mu-\sigma-\tau}{4};\frac{\mu}{2};-\frac{\abs{ y }^2}{\abs{ x }^2}\right)
    f(x)\phi(y)\td y \td x.
\end{align*}
We first calculate the integral over $y\in\RR^{n-m}$. Using
Appendix~\ref{app:FourierHankel} and the integral formula
\eqref{eq:IntFormulaJBesselHypergeometric} we find
\begin{align*}
  \MoveEqLeft[3]
  (2\pi)^{-\frac{n-m}{2}}\int_{\RR^{n-m}}
  e^{-iy\cdot\eta}
    {_2F_1}\left(\frac{\mu-\sigma+\tau}{4},
      \frac{\mu-\sigma-\tau}{4};\frac{\mu}{2};
      -\frac{\abs{ y }^2}{\abs{ x }^2}\right)\phi(y)\td y
  \\
  ={}&
  i^{-k}\phi(\eta)\abs{ \eta }^{-\frac{\mu-2}{2}}
  \int_0^\infty
  J_{\frac{\mu-2}{2}}(\abs{ \eta }s)
  {_2F_1}\left(\frac{\mu-\sigma+\tau}{4},
    \frac{\mu-\sigma-\tau}{4};\frac{\mu}{2};
    -\frac{s^2}{\abs{ x }^2}\right)s^{\frac{\mu}{2}}\td s
  \\
  ={}&
  i^{-k}\phi(\eta)\abs{ \eta }^{-\frac{\mu-2}{2}}
  \frac{2^{\frac{\sigma+2}{2}} \Gamma(\frac{\mu}{2})}{
    \Gamma(\frac{\mu-\sigma+\tau}{4})
    \Gamma(\frac{\mu-\sigma-\tau}{4})}
  \abs{ x }^{\frac{\mu-\sigma}{2}}
  \abs{ \eta }^{-\frac{\sigma+2}{2}}
  K_{\frac{\tau}{2}}(\abs{ x }\cdot\abs{ \eta }).
\end{align*}
If we let
\begin{equation*}
  \psi(x,\eta) :=
  \frac{2^{\frac{\sigma+2}{2}}i^{-k}
    \Gamma(\frac{\mu}{2})}{\Gamma(\frac{\mu-\sigma+\tau}{4})
    \Gamma(\frac{\mu-\sigma-\tau}{4})}\abs{ \eta }^{-\frac{\sigma+\mu}{2}}
  \abs{ x }^{-\frac{\tau}{2}}K_{\frac{\tau}{2}}(\abs{ x }\cdot\abs{ \eta })
\end{equation*}
then we find that
\begin{align*}
  \calF_{\RR^n}A(\sigma,\tau)(f\otimes\phi)(\xi,\eta) &=
  \calF_{\RR^m}(f\cdot\psi(-,\eta))(\xi)\cdot\phi(\eta)
  \\
  &= (2\pi)^{-\frac{m}{2}}(\calF_{\RR^m}\psi(-,\eta)*\calF_{\RR^m}f)(\xi)\cdot\phi(\eta).
\end{align*}
Therefore we compute, using again Appendix~\ref{app:FourierHankel} and
the integral formula \eqref{eq:IntFormulaJBesselKBessel} (noticing
that $K_\nu(x)=K_{-\nu}(x)$)
\begin{align*}
  \MoveEqLeft (\calF_{\RR^m}\psi(-,\eta))(\xi)
  \\
  ={}&
  \frac{2^{\frac{\sigma+2}{2}}i^{-k}
    \Gamma(\frac{\mu}{2})}{\Gamma(\frac{\mu-\sigma+\tau}{4})
    \Gamma(\frac{\mu-\sigma-\tau}{4})}
  \abs{ \eta }^{-\frac{\sigma+\mu}{2}}
  \abs{ \xi }^{-\frac{m-2}{2}}
  \int_0^\infty
  J_{\frac{m-2}{2}}(\abs{ \xi }s)s^{-\frac{\tau}{2}}
  K_{\frac{\tau}{2}}(\abs{ \eta }s)s^{\frac{m}{2}}\td  s
  \\
  ={}&
  \frac{2^{\frac{\sigma-\tau+m}{2}}i^{-k}
    \Gamma(\frac{\mu}{2})\Gamma(\frac{m-\tau}{2})}{
    \Gamma(\frac{\mu-\sigma+\tau}{4})
    \Gamma(\frac{\mu-\sigma-\tau}{4})}
  \abs{ \eta }^{-\frac{\sigma+\tau+\mu}{2}}(\abs{ \xi }^2+\abs{ \eta }^2)^{\frac{\tau-m}{2}}.
\end{align*}
Altogether we see that $I(\sigma,\mu)$ is a partial convolution
operator combined with a multiplication operator
\begin{equation*}
  I(\sigma,\tau)(f\otimes\phi)(\xi,\eta) =
  \const\cdot\abs{ \eta }^{-\frac{\sigma+\tau+\mu}{2}}
  \phi(\eta)\int_{\RR^m}
  (\abs{ \xi-\xi' }^2+\abs{ \eta }^2)^{\frac{\tau-m}{2}}f(\xi')\td\xi'.
\end{equation*}
For $m=n-1$ this operator appears in \cite{KS13,KS14} as a special case.
This expression for $I(\sigma,\tau)$ is valid for $\Real \sigma \ll
\Real \tau \ll 0$. It has a holomorphic extension to all
$\sigma,\tau\in \CC$ for $f\in\calF_{\RR^m}C_c^\infty(\RR^m\minuszero
)$.

\appendix

\section{Decomposition of principal series}
\label{app:DecompPrincipalSeries}

We give a short alternative proof for the decomposition of the
principal series $\pi_{\sigma,\varepsilon}^G$, $\sigma\in i\RR$,
$\varepsilon\in\ZZ/2\ZZ$, into irreducible $H$-representations. This
decomposition turns out to be essentially equivalent to the Plancherel
formula for $L^2(O(1,m+1)/(O(1)\times O(m+1)),\calL'_\delta)$, where
$\calL'_\delta$ are the line bundles over the Riemannian symmetric
space $O(1,m+1)/(O(1)\times O(m+1))$ induced by the characters
$(a,g)\mapsto a^\delta$ of $O(1)\times O(m+1)$, $\delta\in\ZZ/2\ZZ$.

Consider the flag variety $X=G/P$. Since $G/P\cong K/M$ we can
identify $X$ with the unit sphere $S^n\subseteq\RR^{n+1}$. For this we
define a $G$-action on $S^n$ by the formula
\begin{equation*}
  g\circ x := \frac{\pr_x(g(1,x))}{\pr_0(g(1,x))},  \qqtext{$x\in S^n$},
\end{equation*}
where $\pr_0:\RR^{n+2}\to\RR$ and $\pr_x:\RR^{n+2}\to\RR^{n+1}$ denote
the projections onto the first coordinate and the last $n+1$
coordinates, respectively, and $g(1,x)$ is the usual action of $g$ on
$(1,x)\in\RR\times\RR^{n+1}\cong\RR^{n+2}$. Then it is easy to prove
the following:

\begin{lem}
  The operation $\circ$ defines a transitive group action of $G$ on
  $S^n$. The stabilizer of the point $e_{n+1}=(0,\ldots,0,1)\in S^n$
  is equal to the parabolic subgroup~$P$. The maximal compact subgroup
  $K$ also acts transitively on $S^n$ and the stabilizer subgroup of
  the point $x_0$ is equal to $M$.
\end{lem}

Let us consider a slightly different embedding of $O(1,m+1)\times
O(n-m)$ into $G=O(1,n+1)$. Let
\begin{equation*}
  H' := \Set{ \diag(g,h) }{ g\in O(1,m+1),h\in O(n-m) }.
\end{equation*}
Then clearly $H$ and $H'$ are conjugate and hence the branching to $H$
is equivalent to the branching to $H'$. We shall therefore only deal
with $H'$ in this section.

\begin{lem}
  Under the action $\circ$ of the group $H'$ the sphere $S^n$
  decomposes into the two orbits
  \begin{align*}
    \calO_0 &:= H'\circ e_1 =\Set{ (x',0) }{ x'\in S^m },
    \\
    \calO_1 &:= H'\circ e_{n+1} = \Set{ (x',x'')\in S^n }{
      x'\in\RR^{m+1},x''\in\RR^{n-m},x''\neq0 }.
  \end{align*}
  The orbit $\calO_1$ is open and dense in $S^n$. The isotropy group
  of $e_{n+1}$ in $H'$ is
  \begin{equation*}
    S =\Set{ (a,g,h,a) }{ a\in O(1),g\in O(m+1),h\in O(n-m-1) }.
  \end{equation*}
\end{lem}

Now consider the realization of $\pi_{\sigma,\varepsilon}$ in the
compact picture, i.e.\ on $L^2(G/P,\calL_{\sigma,\varepsilon})$, where
$\calL_{\sigma,\varepsilon}$ denotes the line bundle over $G/P$
associated to the character
$man\mapsto\xi_\varepsilon(m)a^{\sigma+\rho}$ of $P$. Since the orbit
$\calO_1\subseteq G/P$ is open and dense we have
\begin{equation*}
  L^2(G/P,\calL_{\sigma,\varepsilon}) \cong
  L^2(\calO_1,\calL_{\sigma,\varepsilon} \restrictedto {\calO_1}).
\end{equation*}
Now the stabilizer $S$ of $eP\in G/P$ in $H$ is contained in $P$ and
hence the restriction of the line bundle $\calL_{\sigma,\varepsilon}$
to $\calO_1\cong H'/S$ is induced by the restriction of the
corresponding character of $P$ to $S$ which is simply
$\xi_\varepsilon \restrictedto S$. Therefore we find
\begin{equation*}
  L^2(G/P,\calL_{\sigma,\varepsilon}) \cong
  L^2(\calO_1,\calL_\varepsilon),
\end{equation*}
where $\calL_\varepsilon$ is the line bundle over $\calO_1\cong H'/S$
induced by the character $\xi_\varepsilon \restrictedto S$. Using the
decomposition of $L^2(S^{n-m-1})$ into spherical harmonics we find
\begin{equation*}
  L^2(\calO_1,\calL_\varepsilon) \cong
  \sideset{}{^\oplus}\sum_{k=0}^\infty L^2(O(1,m+1)/(O(1)\times
  O(m+1)),\calL'_{\varepsilon+k})\boxtimes\calH^k(\RR^{n-m})
\end{equation*}
as $H'$-representations, where for $\delta\in(\ZZ/2\ZZ)$ we denote by
$\calL'_\delta$ the line bundle over the symmetric space
$O(1,m+1)/(O(1)\times O(m+1))$ induced by the character $(a,g)\mapsto
a^\delta$ of $O(1)\times O(m+1)$. Together we obtain
\begin{equation*}
  \pi_{\sigma,\varepsilon}^G \restrictedto[\big] H \cong
  \sideset{}{^\oplus}\sum_{k=0}^\infty L^2(O(1,m+1)/(O(1)\times
  O(m+1)),\calL'_{\varepsilon+k})\boxtimes\calH^k(\RR^{n-m})
\end{equation*}
and hence the decomposition of $\pi_{\sigma,\varepsilon}^G
\restrictedto H$ into
irreducible $H$-representations is equivalent to the decomposition of
$L^2(O(1,m+1)/(O(1)\times O(m+1)),\calL'_\delta)$ into irreducible
$O(1,m+1)$-representations, $\delta\in\ZZ/2\ZZ$. Since
$O(1,m+1)/(O(1)\times O(m+1))$ is a Riemannian symmetric space of rank
one the decomposition of $L^2(O(1,m+1)/(O(1)\times
O(m+1)),\calL'_\delta)$ is well-known and given by
\begin{equation*}
  L^2(O(1,m+1)/(O(1)\times O(m+1)),\calL'_\delta) \cong
  \int_{i\RR_+}^\oplus{\pi_{\tau,\delta}^{O(1,m+1)}\td\tau},
\end{equation*}
the unitary isomorphism established by the spherical Fourier
transform. This proves Theorem~\ref{thm:RepDecomp} for the special
case $\sigma\in i\RR$.

\section{Special functions}

For the sake of completeness we collect here the necessary formulas
for certain special functions needed in this paper.

\subsection{The $K$-Bessel function}\label{app:KBessel}

We renormalize the classical $K$-Bessel function $K_\alpha(z)$ by
\begin{equation*}
  \widetilde{K}_\alpha(z) :=
  \left(\frac{z}{2}\right)^{-\alpha}K_\alpha(z).
\end{equation*}
Then $\widetilde{K}_\alpha(z)$ solves the differential equation
\begin{equation}
  \frac{\di^2u}{\di z^2}+\frac{2\alpha+1}{z}\frac{\di u}{\di z}-u =
  0.\label{eq:DiffEqKBessel}
\end{equation}
It has the following asymptotic behaviour as $x\to0$ (see
\cite[Chapters~III~\&~VII]{Wat44}):
\begin{align}
  \widetilde{K}_\alpha(x) &=
  \begin{cases}
    \frac{\Gamma(\alpha)}{2}\left(\frac{x}{2}\right)^{-2\alpha}+o(x^{-2\alpha}),
    & \text{for $\Real\alpha>0$,}
    \\[1ex]
    -\log\left(\frac{x}{2}\right)+o\left(\log\left(\frac{x}{2}\right)\right),
    & \text{for $\Real\alpha=0$,}
    \\[1ex]
    \frac{\Gamma(-\alpha)}{2}+o(1), & \mbox{for $\Real\alpha<0$.}
  \end{cases}\label{eq:KBesselAsymptotics0}
  \intertext{Further, as $x\to\infty$ we have} 
  \widetilde{K}_\alpha(x)
  &=
  \frac{\sqrt{\pi}}{2}\left(\frac{x}{2}\right)^{-\alpha-\frac{1}{2}}
  e^{-x}\left(1+\calO\left(\frac{1}{x}\right)\right).
  \label{eq:KBesselAsymptoticsInfty}
\end{align}
For the derivative of $\widetilde{K}_\alpha(z)$ the following identity
holds (see \cite[equation III.71~(6)]{Wat44}):
\begin{equation}
  \frac{\di}{\di z}\widetilde{K}_\alpha(z) =
  -\frac{z}{2}\widetilde{K}_{\alpha+1}(z).\label{eq:DerivativeKBessel}
\end{equation}
This identity can be used to write the differential equation \eqref{eq:DiffEqKBessel} as the three-term recurrence relation (see \cite[equation III.71~(6)]{Wat44}):
\begin{equation}
  z^2 \widetilde{K}_{\alpha+1}(z) = 4\alpha \widetilde{K}_\alpha(z) + 4\widetilde{K}_{\alpha-1}(z).\label{eq:RecRelKBessel}
\end{equation}

\subsection{The Gau\ss\ hypergeometric  function}
\label{app:GaussHypergeometric}

Consider the classical Gau\ss\ hypergeometric function
\begin{equation*}
  {_2F_1}(a,b;c;z) =
  \sum_{n=0}^\infty \frac{(a)_n(b)_n}{n!(c)_n}z^n,
\end{equation*}
where $(a)_n=a(a+1)\cdots(a+n-1)$ denotes the Pochhammer symbol. The
function ${_2F_1}(a,b;c;z)$ is holomorphic in $z$ for
$z\notin[1,\infty)$ and meromorphic in the parameters
$a,b,c\in\CC$. It solves the differential equation
\begin{equation}
  (1-z)z\frac{\di^2 u}{\di z^2}+(c-(a+b+1)z)\frac{\di u}{\di z}-abu =
  0.\label{eq:DiffEqHypergeometric}
\end{equation}
The following formula allows to study the asymptotic behaviour of the
Gau\ss\ hypergeometric function near $z=-\infty$ (see \cite[equation
9.132~(2)]{GR65}):
\begin{align}\label{eq:2F1Identity1}
  {_2F_1}(a,b;c;z) ={}&
  \frac{\Gamma(b-a)\Gamma(c)}{\Gamma(b)\Gamma(c-a)}(-z)^{-a}{_2F_1}(a,a-c+1;a-b+1;\tfrac{1}{z})\\
  \nonumber &
  +\frac{\Gamma(a-b)\Gamma(c)}{\Gamma(a)\Gamma(c-b)}(-z)^{-b}{_2F_1}(b,b-c+1;b-a+1;\tfrac{1}{z}).
\end{align}
Both summands on the right hand side of \eqref{eq:2F1Identity1} are
generically linear independent solutions to
\eqref{eq:DiffEqHypergeometric}. Their Wronskian is given by
\begin{gather*}
  W(z^{-a}{_2F_1}(a,a-c+1;a-b+1;-\tfrac{1}{z}),z^{-b}{_2F_1}(b,b-c+1;b-a+1;-\tfrac{1}{z}))
  \\
  = (a-b)(1+z)^{c-a-b-1}z^{-c}.
\end{gather*}
The following simple formula for the derivative of the hypergeometric
function holds:
\begin{equation}
  \frac{\di}{\di z}{_2F_1}(a,b;c;z) =
  \frac{ab}{c}{_2F_1}(a+1,b+1;c+1;z).\label{eq:DerivativeHypergeometricFunction}
\end{equation}
We recall Kummer's transformation formula (see
\cite[equation~9.131~(1)]{GR65}):
\begin{equation}
  {_2F_1}(a,b;c;z) =
  (1-z)^{c-a-b}{_2F_1}(c-a,c-b;c;z).\label{eq:KummerFormula}
\end{equation}
For $a\in-\NN$ the hypergeometric function ${_2F_1}(a,b;c;z)$
degenerates to a polynomial which can be expressed in terms of the
Jacobi polynomials $P_n^{(a,b)}(z)$ (see
\cite[equation~8.962~(1)]{GR65}):
\begin{equation}
  {_2F_1}(-n,b;c;z) = \frac{n!}{(c)_n}P_n^{(c-1,b-c-n)}(1-2z), 
  \qqtext{$n\in\NN$},\label{eq:HypergeometricPolynomials}
\end{equation}
where
\begin{equation*}
  P_n^{(a,b)}(z) =
  \frac{1}{n!}\sum_{k=0}^n
  \frac{(-n)_k(a+b+n+1)_k(a+k+1)_{n-k}}{k!}\left(\frac{1-z}{2}\right)^k.
\end{equation*}

\subsection{Integral formulas}

We consider the $J$-Bessel function $J_\nu(z)$ and the $K$-Bessel
function $K_\nu(z)$. For the $J$-Bessel function and the
hypergeometric function the following integral formula holds for
$y>0$, $\Real\lambda>0$ and
$-1<\Real\nu<2\max(\Real\alpha,\Real\beta)-\frac{3}{2}$ (see
\cite[equation 7.542~(10)]{GR65})
\begin{align}
  \MoveEqLeft \int_0^\infty
  {_2F_1}(\alpha,\beta;\nu+1;-\lambda^2x^2)J_\nu(xy)x^{\nu+1}\td x
  \nonumber
  \\
  &=
  \frac{2^{\nu-\alpha-\beta+2}\Gamma(\nu+1)}{\lambda^{\alpha+\beta}
    \Gamma(\alpha)\Gamma(\beta)}
  y^{\alpha+\beta-\nu-2}K_{\alpha-\beta}
  \left(\frac{y}{\lambda}\right).\label{eq:IntFormulaJBesselHypergeometric}
\end{align}
For the $J$-Bessel function and the $K$-Bessel function we have the
following integral formula for $\Real\mu>\abs{ \Real\nu }-1$ and $\Real
b>\abs{ \Imaginary a }$ (see \cite[equation~6.576~(7)]{GR65})
\begin{equation}
  \int_0^\infty x^{\mu+\nu+1}{J_\mu(ax)K_\nu(bx)\td x} =
  2^{\mu+\nu}a^\mu
  b^\nu\frac{\Gamma(\mu+\nu+1)}{(a^2+b^2)^{\mu+\nu+1}}.\label{eq:IntFormulaJBesselKBessel}
\end{equation}

\subsection{Fourier and Hankel transform}\label{app:FourierHankel}

Let $\calF_{\RR^n}$ denote the Euclidean Fourier transform on $\RR^n$
as defined in \eqref{eq:DefFourierTransform}. Let $k\in\NN$ and
$\phi\in\calH^k(\RR^n)$. For $f\in L^2(\RR_+,r^{n+2k-1}\td r)$ denote
by $f\otimes\phi\in L^2(\RR^n)$ the function
\begin{equation*}
  (f\otimes\phi)(x) := f(\abs{ x })\phi(x),  \qqtext{$x\in\RR^n$}.
\end{equation*}
Then by \cite[Chapter~IV, Theorem~3.10]{SW71}
\begin{equation*}
  \calF_{\RR^n}(f\otimes\phi) =
  i^{-k}(\calH_{\frac{n+2k-2}{2}}f)\otimes\phi,
\end{equation*}
where $\calH_\nu$ is the modified Hankel transform of parameter
$\nu\geq-\frac{1}{2}$
\begin{equation*}
  \calH_\nu f(r) = r^{-\nu}\int_0^\infty{J_\nu(rs)f(s)s^{\nu+1}\td
    s},
\end{equation*}
which is a unitary isomorphism (up to a scalar multiple) on
$L^2(\RR_+,r^{2\nu+1}\td r)$.

\bibliographystyle{amsplain} 
\bibliography{bibdb}

\vspace{30pt}

\noindent
\textsc{Jan M\"ollers\\Department of Mathematics, The Ohio State University, 231 West 18th Avenue, Columbus, OH 43210, USA}\\
\emph{E-mail address:} \texttt{mollers.1@osu.edu}

\bigskip
\noindent
\textsc{Yoshiki Oshima\\Kavli IPMU (WPI), The University of Tokyo, 5-1-5 Kashiwanoha, Kashiwa, 277-8583, Japan}\\
\emph{E-mail address:} \texttt{yoshiki.oshima@ipmu.jp}

\end{document}